\documentclass[10pt,reqno]{amsart}

\usepackage{amssymb, amsfonts, amsmath, mathrsfs, bbm, dsfont, textcomp, stmaryrd, verbatim, diagrams, setspace, mathtools, lineno, cite}
\usepackage{tikz}
\usepackage{tikz-cd}

\usepackage[all]{xy}

\usepackage{amsrefs}

\usepackage[colorlinks, colorlinks, linkcolor=blue, anchorcolor=blue, urlcolor=blue, citecolor=blue]{hyperref}

\newcommand{\C}{{\mathbb C} }
\newcommand{\R}{{\mathbb R} }
\newcommand{\fK}{{\mathbb K } }
\newcommand{\bS}{{\mathbb S } }
\newcommand{\Z}{{\mathbb Z} }

\newcommand{\cO}{{\mathcal O}}
\newcommand{\cM}{{\mathcal M}}

\newcommand{\cD}{{\mathcal D}}

\newcommand{\cE}{{\mathcal E}}
\newcommand{\cF}{{\mathcal F}}

\newcommand{\cX}{{\mathfrak X}}

\newcommand{\XX}{\mathcal{X} }
\newcommand{\DD}{\mathcal{D} }
\newcommand{\TT}{\mathcal{T} }

\newcommand{\smf}{C^{\infty}}

\newcommand{\at}{\alpha}

\newcommand{\embk}{[\text{ }, \text{ }]}

\newcommand{\nbk}{\textup{\textlbrackdbl ~,~  \textrbrackdbl} }

\newcommand{\Mod}{
(C_{\cM}^{\infty}, \, Q)\mathrm{-}\mathbf{mod}
}

\newcommand{\Dp}{\DD_{\mathrm{poly}}^1}

\newcommand{\Dpp}{\DD_{\mathrm{poly}}}

\newcommand{\Lp}{\functorL(\DD_{\mathrm{poly}}^1) }

\newcommand{\tot}{\mathrm{tot}\, }

\newcommand{\hkr}{\mathrm{hkr}\, }

\newcommand{\pbw}{\mathrm{pbw}\, }

\newcommand{\mult}{\mathrm{m} }

\newcommand{\Bott}{\nabla^{\mathrm{Bott}} }

\newcommand{\fedo}{d_{L}^{\nabla^{\lightning}}}

\newcommand{\Mfedo}{\cM_{\mathrm{F}}}

\newcommand{\Qfedo}{Q_{\mathrm{F}}}
\newcommand{\LQF}{L_{\Qfedo}}

\newcommand{\cMmodule}{\mathfrak{N}}
\newcommand{\abs}[1]{{|#1|}}

\newcommand{\Tp}{\TT_{\mathrm{poly}}^1}
\newcommand{\Tpp}{\TT_{\mathrm{poly}}}

\newcommand{\id}{\mathrm{id}}
\newcommand{\Hom  }{\mathrm{Hom }}
\newcommand{\End}{\mathrm{End}}
\newcommand{\Ext }{\mathrm{Ext }}
\newcommand\functorL{\mathrm{L}}
\newcommand{\Verticalthings}[1]{\mathsf{#1}}

\newtheorem{definition}{Definition}[section]
\newtheorem{lemma}[definition]{Lemma}
\newtheorem{proposition}[definition]{Proposition}
\newtheorem{corollary}[definition]{Corollary}
\newtheorem{theorem}[definition]{Theorem}\newtheorem{remark}[definition]{Remark}

\begin{document}

\title{Hopf algebras arising from dg manifolds}

\thanks{Research partially supported by NSFC grants 12071241 and 11471179.}

\author{Jiahao Cheng}
\address{College of Mathematics and Information Science, Nanchang Hangkong University}
\email{jiahaocheng\underline{\hspace{0.15cm}}math@yeah.net \quad jiahaocheng@nchu.edu.cn}

\author{Zhuo Chen}
\address{Department of Mathematics, Tsinghua University}
\email{chenzhuo@tsinghua.edu.cn}

\author{Dadi Ni}
\address{Department of Mathematics, Tsinghua University}
\email{ndd18@mails.tsinghua.edu.cn}
 
\date{}

\begin{abstract}
Let $(\cM, \, Q)$ be a dg manifold.
The space of  vector fields with shifted degrees $(\XX(\cM)[-1], \, L_Q)$ is a Lie algebra object in the homology category
$\mathrm{H}(\Mod)$ of dg modules over $(\cM, \, Q)$, the Atiyah class $\at_{\cM}$ being its Lie bracket.
The triple $(\XX(\cM)[-1], \, L_Q;$
$\at_{\cM})$ is also a Lie algebra object  in the Gabriel-Zisman homotopy category $\Pi(\Mod)$.
In this paper, we describe the universal enveloping algebra of $(\XX(\cM)[-1], \, L_Q; \, \at_{\cM})$ and prove that it is a Hopf algebra object in $\Pi(\Mod)$.
As an application, we study Fedosov dg Lie algebroids and recover a result of Sti\'enon, Xu, and the second author on the Hopf algebra arising from a Lie pair.
\end{abstract}

\maketitle

~\\
\hspace*{1cm} {\it Keywords:~}\\\hspace*{1.2cm} Atiyah class, dg manifold, Hopf algebra, HKR-theorem, Fedosov dg manifold.\\
\hspace*{1cm} {\it AMS subject classification(2020):~}\\
\hspace*{1.2cm} Primary 58A50; Secondary 16E45,16T05, 17B01,17B70, 53C12.

\tableofcontents

\section{Introduction}
Atiyah classes were introduced by Atiyah \cite{Atiyah} as the obstruction to the existence of a holomorphic connection on a complex manifold.
It was shown by Kapranov \cite{Kapranov} that the Atiyah class of a complex manifold $X$ endows the shifted holomorphic tangent bundle $T_{X}[-1]$ with a Lie algebra structure in the derived category of coherent sheaves of $\cO_X$-modules, and thus Atiyah classes form  a bridge between complex geometry and Lie theory.
This structure plays an important role in the construction of Rozansky-Witten invariants \cite{Kontsevich1}. In his work on the deformation quantization of Poisson manifolds \cite{Kontsevich2}, Kontsevich indicated a deep link between the Todd genus of complex manifolds and the Duflo element of Lie algebras.
For a wealth of further investigation, see \cites{Calaque-VandeBergh, Costello, G-G, Shoikhet}.

Lie algebroid pairs  (Lie pairs in short) arise naturally in a number of classical areas of mathematics such as Lie theory, complex geometry,
foliation theory, and Poisson geometry.
Recall that a Lie algebroid over
$\fK$ where $\fK$ denotes either of the fields $\mathbb{R}$ or $\mathbb{C}$, is a $\fK$-vector bundle $L \to M$ together with a bundle map $\rho$:~ $L\to T_M\otimes_{\mathbb{R}}\fK$  called anchor and a Lie bracket $[~,~]$ on the sections of $L$ such that $\rho$: $\Gamma(L)\to \mathcal{X}(M)\otimes_{\mathbb{R}}\fK$ is a morphism of Lie algebras and
$[X, fY ] = f[X, Y ] +
(\rho(X)f)  Y$,
for all $X$, $Y\in \Gamma(L)$   and $f \in C^\infty(M,\fK)$. By a Lie pair $(L,A)$, we mean an inclusion $A\hookrightarrow L$ of Lie
algebroids over a smooth manifold $M$.

 A complex manifold $X$ determines a Lie pair over $\mathbb{C}$: $L = T_X\otimes \mathbb{C}$
and $A = T^{0,1}_X$. A foliation $\mathfrak{F}$ on a smooth manifold $M$ determines a Lie pair over $\mathbb{R}$: $L = T_M$ and $A = T_\mathfrak{F}$
(the integrable distribution on $M$ tangent to the foliation $\mathfrak{F}$). A $\mathfrak{g}$-manifold, i.e., a smooth manifold equipped with an infinitesimal action of a Lie
algebra $\mathfrak{g}$,
gives rise to a Lie pair in a natural way
(see \cite{Mokri}*{Example 5.5},
\cite{Mackenzie}, and \cite{L-S-X3}).

Sti\'enon, Xu, and the second author introduced the Atiyah class of Lie   pairs in \cite{C-S-X2}, which encodes both the Atiyah class of complex manifolds and the Molino class \cite{Molino} of foliations as special cases. 
A recent work \cite{LJL} on the Atiyah class of generalized holomorphic vector bundles  is also  connected with Lie pairs.  
Towards a different direction, Mehta, Sti\'enon, and Xu \cite{M-S-X} introduced the Atiyah class and the Todd class of dg manifolds. By a dg manifold, we mean a $\mathbb{Z}$-graded manifold $\cM$ endowed with a homological vector field, i.e. a vector field
$Q\in \mathcal{X}(\cM)$ of degree $(+1)$ satisfying $ Q^2 = 0$.
They are sometimes referred to as $Q$-manifolds in the mathematical
physics literature, for instance, in Schwarz's pioneering work on geometry of Batalin-Vilkovisky
	quantization \cite{Schwarz}, and that of Alexandrov-Kontsevich-Schwarz-Zaboronsky on the geometry of the master 	equation and topological quantum field theory (known as AKSZ) \cite{AKSZ}. They also arise
naturally in many situations in geometry, Lie theory, and mathematical physics.
For instance, see \cite{L-S-X} for the formality theorem for smooth dg manifolds, which implies the Kontsevich-Duflo theorem for Lie algebras and Kontsevich's theorem for complex manifolds \cite{Kontsevich2} under a unified framework.

In the present paper, we  study   a dg manifold
$(\cM, \, Q)$, and the associated Atiyah class, which
is   an element
\[\at_{\cM}\in
\Hom_{\mathrm{H}(\Mod)}( (\XX(\cM)[-1], \, L_Q)\otimes_{\smf_{\cM}}(\XX(\cM)[-1], \, L_Q), (\XX(\cM)[-1], \, L_Q)),
\]
where $\Mod$ is the category of dg modules over
$(\cM, \, Q)$ and
$\mathrm{H}(\Mod)$ is the homology category of $\Mod$.
In Mehta-Sti\'enon-Xu \cite{M-S-X}, it is shown that
the $\fK$-vector space $ \XX(\cM)[-1] $ admits an $L_{\infty}$-algebra structure with
$L_Q=[Q, \, \cdot \,]$ as the unary  bracket and  the Atiyah cocycle that represents the Atiyah class $\at_{\cM}$ as the binary bracket. Consequently,  the triple
$(\XX(\cM)[-1], \, L_Q; \, \at_{\cM})$ is a Lie algebra object in the homology category $\mathrm{H}(\Mod)$  and in the Gabriel-Zisman  homotopy category $\Pi(\Mod)$ as well,
where the Atiyah class $\at_{\cM}$ is regarded as a Lie bracket on $\XX(\cM)[-1]$  (see Sections \ref{Section4} and \ref{Section5} for more details).

We are motivated by Ramadoss's work \cite{Ramadoss} where it is proved that for a   smooth scheme $(X, \cO_{X})$ over $\C$  the universal enveloping algebra of the Lie algebra object $T_X[-1]$ in $ D^{+}(\cO_X) $, the derived category of bounded below complexes of $\cO_X$-modules, is the Hochschild cochain complex $(D_{\mathrm{poly}}^{\bullet}, \, d_H)$ of polydifferential operators on $X$, which is also a Hopf algebra object in $ D^{+}(\cO_X) $.   This result  played an important role in the study of the Riemann-Roch theorem \cite{MR2472137}, the Chern character \cite{Ramadoss}, and the Rozansky-Witten invariants \cite{MR2661534}.

It is well known that every ordinary Lie algebra $\mathfrak{g}$ admits a universal enveloping algebra $U(\mathfrak{g})$, which is a Hopf algebra.
 We are thus led to the natural question: does there exist  a universal enveloping algebra for $(\XX(\cM)[-1], \, L_Q;
 \, \at_{\cM})$ in the homotopy category $\Pi(\Mod)$ and, if so, is it a Hopf algebra object?

In this paper, we give a positive answer to the questions above.
 Consider the space of differential operators on $\cM$ which we denote by $\cD_{\cM}$.
Let $(\DD_{\mathrm{poly}}^1, \, L_Q, \, 0)$ be
the
$(\smf_{\cM}, \, Q)$-complex (see Section \ref{Section2} for this notion)
which is a copy of $(\cD_{\cM}, \, L_Q)$ concentrated in horizontal degree $(+1)$.
We then form the space of polydifferential operators $
 \DD_{\mathrm{poly}} :~~=\widetilde{\otimes}^\bullet \DD_{\mathrm{poly}}^1
 $.
 Together with $L_Q$ and the Hochschild differential $d_H$,
 the triple $(\Dpp, \, L_Q, \, d_H)$ is a Hopf algebra object in the category $Ch(\Mod)$ of
 $(\smf_{\cM}, \, Q)$-complexes.
Let $(\Lp; \, \nbk)$ be the free Lie algebra
generated over $\smf_{\cM}$ by
 $\Dp$.
We have the Lie algebra object
$(\Lp, \, L_Q, \, d_H; \, \nbk)$
in $Ch(\Mod)$
(see  Sections \ref{Section2} and \ref{Section3}  for details).

The first main result  of this paper is the following theorem.
\begin{theorem}\emph{[see Section \ref{Section5}, Theorem \ref{maintheorem}]}
\label{maintheoreminintroduction}\
	
	\begin{itemize}
		\item[(1)]
		The natural inclusion map $\theta:~ (\XX(\cM)[-1], \, L_Q) \rightarrow (\tot\Lp, \, L_Q+d_H)$ is an isomorphism in the homotopy category $\Pi(\Mod)$.
		Moreover, it is an isomorphism of Lie algebra objects
$(\XX(\cM)[-1], \, L_Q; \, \at_{\cM}) \rightarrow (\tot\Lp, \, L_Q+d_H; \, \nbk)$,
i.e., the following diagram commutes in $\Pi(\Mod)$:
		\begin{equation}\label{maindiagraminintroduction}
		\xymatrix{
			(\XX(\cM)[-1], \, L_Q) \otimes_{\smf_{\cM}} (\XX(\cM)[-1], \, L_Q)  \ar[r]^>>>>>>>>>>>{\at_{\cM}} \ar[d]_{\theta \otimes \theta} &
			(\XX(\cM)[-1], \, L_Q)\ar[d]^{\theta}\\
			(\tot\Lp, \, L_Q+d_H)
			\otimes_{\smf_{\cM}}
			(\tot\Lp, \, L_Q+d_H)\ar[r]^>>>>>>{\nbk} & (\tot \functorL(\DD_{\mathrm{poly}}^1), \, L_Q+d_H) \, .\\
		}
		\end{equation}
		
		\item[(2)]
		The $(\smf_{\cM}, \, Q)$-module $(\tot \Dpp, \, L_Q+d_H)$ is the universal enveloping algebra of
		the Lie algebra object $(\XX(\cM)[-1], \, L_Q; \,\at_{\cM} )$, and is a Hopf algebra object in $\Pi(\Mod)$.	
	\end{itemize}
\end{theorem}

 This result is an analogue of Ramadoss's Theorem $2$ in \cite{Ramadoss} in the context of dg manifolds.
In Diagram \eqref{maindiagraminintroduction}, the Atiyah class
$\at_{\cM}$ on the upper side reflects the geometric information of $(\cM, \, Q)$,
while the Lie bracket of the lower side is the naive free Lie algebra bracket over $\smf_{\cM}$, which is a pure algebraic operation.
Thus Diagram \eqref{maindiagraminintroduction} can be seen as a universal representation of Atiyah classes of dg manifolds.

We are also motivated by Chen-Sti\'enon-Xu's \cites{C-S-X,C-S-X2}  where it is shown that the quotient $L/A[-1]$ of a Lie pair $(L, \, A)$ is a Lie algebra object in the derived category $D^b(\mathcal{A})$ of the category of $A$-modules, the Atiyah class
\[
\at_{L/A}\in H^1_{\mathrm{CE}}(A, \Hom   (L/A \otimes L/A, L/A))\cong \Hom_{D^b(\mathcal{A})}(\Gamma(L/A)[-1]\otimes_{\smf_M}\Gamma(L/A)[-1],\Gamma(L/A)[-1])
\]
being its Lie bracket
  (see Section \ref{Section6} for more details).
Recall also the constructions of
a Hopf algebra object
$(D^{\bullet}_{\mathrm{poly}}({L/A}), \, d_H)$
and
a Lie algebra object
$(\functorL(D_{\mathrm{poly}}^{1}(L/A)),\, d_H ; \, \nbk)$
in the category of cochain complexes of $A$-modules in \cite{C-S-X}. The graded Lie algebra
$\functorL(D_{\mathrm{poly}}^{1}(L/A))$
is freely
generated over $\smf_{M}$ by
$D_{\mathrm{poly}}^{1}(L/A)$, where $D_{\mathrm{poly}}^{1}(L/A)$ is placed in degree $(+1)$.
Moreover, we have:
\begin{theorem}\cite{C-S-X}*{Proposition 4.2}\label{C-S-Xinintroduction}\

Let $\beta:~~ \Gamma(L/A)[-1] \rightarrow (\functorL(D_{\mathrm{poly}}^{1}(L/A)), \, d_H)$ be the natural inclusion map.
\begin{itemize}
		\item[(1)]	
		The map $\beta$ is an isomorphism in the derived category $D^b(\mathcal{A})$ of $A$-modules. Moreover, it is an isomorphism of Lie algebra objects, i.e., the following diagram commutes in $D^b(\mathcal{A})$:
\begin{equation}\label{Eqt:olddiagram}		\xymatrix{
			\Gamma(L/A)[-1] \otimes_{\smf_{M}} \Gamma(L/A)[-1] \ar[r]^-{\at_{L/A}} \ar[d]_{\beta \otimes \beta} &
			\Gamma(L/A)[-1]\ar[d]^{\beta}\\
			(\functorL(D_{\mathrm{poly}}^{1}(L/A)), \, d_H)
			\otimes_{\smf_{M}} (\functorL(D_{\mathrm{poly}}^{1}(L/A)), \, d_H)\ar[r]^-{\nbk} & (\functorL(D_{\mathrm{poly}}^{1}(L/A)), \, d_H) \, .\\
		}
		\end{equation}
		\item[(2)]
		The cochain complex of $A$-modules
		$(D^{\bullet}_{\mathrm{poly}}({L/A}), \, d_H)$ is the universal enveloping algebra of the Lie algebra object
		$( \Gamma(L/A)[-1]; \, \at_{L/A})$, and is a Hopf algebra object in $D^b(\mathcal{A})$.
	\end{itemize}	
\end{theorem}
Certainly, one can regard Diagram \eqref{Eqt:olddiagram} as a universal representation of Atiyah classes of Lie pairs.
It is therefore natural to ask how the Diagrams \eqref{maindiagraminintroduction} and \eqref{Eqt:olddiagram} are related and how the Hopf algebras $(\tot \Dpp, \, L_Q+d_H)$ and $(D^{\bullet}_{\mathrm{poly}}({L/A}), \, d_H)$ are related.

This question can indeed be answered by considering the
 Fedosov dg manifold and the
 Fedosov dg Lie algebroid arising from a Lie pair --- 
Sti\'enon and Xu introduced these objects
in \cites{S-X}. 
 Given a Lie pair $(L, \, A)$ and having chosen some additional geometric data, one can endow the graded manifold
 $\Mfedo: =L[1]\oplus L/A$ with a structure of dg manifold $(\Mfedo, \, \Qfedo)$, called a Fedosov dg manifold. It turns
 out that there exists a natural dg integrable distribution
 $\cF\hookrightarrow T_{\Mfedo}$ on
 $(\Mfedo, \, \Qfedo)$. They call this dg Lie algebroid $(\cF, \, \LQF)$ a Fedosov dg Lie algebroid. In this situation, we  need to consider the space
 $\Verticalthings{D}_{\mathrm{poly}}$ of vertical polydifferential operators along the natural projection 
 $\Mfedo \rightarrow L[1]$,
and we have a Lie algebra object
$(\functorL(\Verticalthings{D}_{\mathrm{poly}}^1), \, L_{\Qfedo}, \, d_H; \, \nbk)$ in the category
$Ch( (\smf_{\Mfedo}, \, \Qfedo)\mathrm{-}\mathbf{mod} )$
of
$(\smf_{\Mfedo}, \, \Qfedo)$-complexes.
See Section \ref{Section6} for details.

Our second main result is as follows:
\begin{theorem}\emph{[see Section \ref{Section6}, Theorem \ref{bigdiagram}]}\label{bigdiagraminintroduction}

The natural inclusion map
 	\[
 	\beta:~~ (\Gamma(L/A)[-1]\otimes_{\smf_M}\Omega(A), \, d_{\mathrm{CE}}^{L/A}) \rightarrow (\functorL(D_{\mathrm{poly}}^{1}(L/A))\otimes_{\smf_M}\Omega(A),  \, d_{\mathrm{CE}}^{L/A}+d_H)
 	\]
 	is an isomorphism in the homotopy category
 	$\Pi((\smf_{\Mfedo}, \, \Qfedo)\mathrm{-}\mathbf{mod})$
 	of $(\smf_{\Mfedo}, \, \Qfedo)$-modules.
 	Moreover, we have the following commutative diagram
 	in $\Pi(  (\smf_{\Mfedo}, \, \Qfedo)\mathrm{-}\mathbf{mod} )$:
 	\begin{equation*}\label{CUBEDIAGRAMinintroduction}
 	\begin{tiny}
 	\xymatrixcolsep{-2.0pc}
 	 	\xymatrix{
 		\otimes_{\smf_{\Mfedo}}^{2}
 		(\Gamma(\cF)[-1], \, \LQF)\ar[rd]^{(\iota^{\ast})^{\otimes 2}}
 		\ar[rr]^{\theta^{\otimes 2}}
 		\ar[dd]^{\alpha_{\cF}}&&
 		\otimes_{\smf_{\Mfedo}}^{2}(\tot \functorL(\Verticalthings{D}_{\mathrm{poly}}^1), \, \LQF+d_H )\ar[rd]^{I^{\otimes 2}} \ar[dd]_<<<<<<<{\textup{\textlbrackdbl}\; , \; \textup{\textrbrackdbl}}|(0.5)\hole\\
 		&\otimes_{\Omega(A)}^{2}(\Gamma(L/A)[-1]{\otimes_{\smf_M}}\Omega(A), \, d_{\mathrm{CE}}^{L/A})\ar[rr]^>>>>>>>{\beta^{\otimes 2}}\ar[dd]_<<<<<<<{\alpha_{L/A}} &&
 		\otimes_{\Omega(A)}^{2}(\functorL(D_{\mathrm{poly}}^{1}(L/A)){\otimes_{\smf_M}}\Omega(A),  \, d_{\mathrm{CE}}^{L/A}+d_H) \ar[dd]_{\textup{\textlbrackdbl}\; , \; \textup{\textrbrackdbl}}
 		\\
 		(\Gamma(\cF)[-1], \, \LQF)\ar[rd]^{\iota^{\ast}}\ar[rr]^>>>>>>>>>>>{\theta}|(0.46)\hole && (\tot \functorL(\Verticalthings{D}_{\mathrm{poly}}^1), \, \LQF+d_H ) \ar[rd]^{I} \hole\\
 		&(\Gamma(L/A)[-1]{\otimes_{\smf_M}}\Omega(A), \, d_{\mathrm{CE}}^{L/A})\ar[rr]^>>>>>>>>>>>{\beta}&&
 		(\functorL(D_{\mathrm{poly}}^{1}(L/A)){\otimes_{\smf_M}}\Omega(A),  \, d_{\mathrm{CE}}^{L/A}+d_H) \, , \\
 	}
 	\end{tiny}
 	\end{equation*}
where the maps $\iota^{\ast}$ and $I$ are a pair of
natural isomorphisms in
$\Pi( (\smf_{\Mfedo}, \, \Qfedo)\mathrm{-}\mathbf{mod} )$.
 \end{theorem}

The left face of this cubic diagram is a relation of the Lie pair Atiyah class $\at_{L/A}$ and the Atiyah class $\alpha_{\cF}$ of the Fedosov dg Lie algebroid $(\cF, \, \LQF)$ discovered by Liao, Sti\'enon, and Xu \cite{L-S-X2}.
Our contribution is the right face of this diagram, namely
the relation between the two Lie algebras
$(\tot \functorL(\Verticalthings{D}_{\mathrm{poly}}^1),  \, \LQF+d_H; \, \nbk )$
and
$(\functorL(D_{\mathrm{poly}}^{1}(L/A)), \, d_H; \,  \nbk)$.
We also remark that
Theorem \ref{bigdiagraminintroduction}
implies Theorem \ref{C-S-Xinintroduction} (which was stated without proofs in \cite{C-S-X}).


This paper is structured as follows. In Section \ref{Section2},
we recall the definition of graded manifolds and the notions such as polyvector fields and polydifferential operators. In Section \ref{Section3}, we recall the notions of dg manifolds, dg structures on polyvector fields and polydifferential operators, and the HKR theorem. In Section \ref{Section4}, we study the Atiyah class of a dg manifold which is the main subject in this paper.
Section \ref{Section5} is devoted to proving our first main result (Theorem \ref{maintheorem}).
Section \ref{Section6}, which begins by recalling Atiyah classes of Lie pairs and Fedosov dg Lie algebroids associated with Lie pairs, is devoted to proving our second main result
(Theorem \ref{bigdiagram}). In Section \ref{Section7}, we compare Ramadoss's work with our results. Appendix \ref{Section8} is an appendix for technical formulas.

\subsection*{Terminology and notations}\

\smallskip

\textit{Field $\fK$ and the ring $R$.}
We use the symbol $\fK$ to denote the field of either real or complex numbers.
The symbol $R$ always denotes the algebra of smooth functions on a manifold $M$ with values in $\fK$.

\bigskip

\textit{Gradings.}
Unless specified otherwise, all vector spaces, algebras, modules, etc. are understood to be 
$\Z$-graded objects over the field $\fK$.
Given a graded object $V=\oplus_{i\in \Z} V^{i}$,
we write $V[k]$ to denote the graded object obtained by shifting the grading on $V$
according to the rule $(V[k])^i=V^{k+i}$.
An element $x\in V$ is called homogeneous, if it belongs to $V^{i}$ for some $i$.

Given a bigraded object
$\Upsilon=\oplus_{(p,q)\in \Z^2}\Upsilon^{p,q}$,
an element $x\in \Upsilon$ is called homogeneous if it belongs to $\Upsilon^{p,q}$ for some
$(p,\, q)$.

\bigskip

\textit{Tensor products and free algebras.}
Unless specified otherwise, we adopt the following conventions.
Given a graded commutative ring $\mathcal{A}$
and two graded
$\mathcal{A}$-modules $V_1$ and $V_2$,
the notation $V_1 \otimes_{\mathcal{A}} V_2$
denotes the tensor product of
$V_1$ and $V_2$ as $\mathcal{A}$-modules,
and the notation $a {\otimes_{\mathcal{A}}}b$
denotes the tensor product of elements
$a \in V_1$ and $b \in V_2$ over $\mathcal{A}$.
For any graded $\mathcal{A}$-module $V$,
the symbol $U(V)$ (resp. $S(V)$) denotes the free associative algebra (resp. free symmetric algebra) generated over
$\mathcal{A}$ by $V$.
The symbol ${\otimes_{\mathcal{A}}}$
(resp. ${\odot_{\mathcal{A}}} $) also denotes
the associative product of $U(V)$
(resp. commutative product of $S(V)$).

The symbol $\hat{S}(V)$ denotes the
$\mathfrak{m}$-adic completion of the symmetric algebra $S(V)$,
where $\mathfrak{m}$ is the ideal of $S(V)$
generated by $V$.

\bigskip

\textit{Shuffles.}
A $(p, \, q)$-shuffle is a permutation
$\sigma$ of the set
$\{1, 2, \cdots, p+q\}$
such that
$\sigma(1)<\cdots < \sigma(p)$
and
$\sigma(p+1)<\cdots < \sigma(p+q)$.
The symbol $\mathrm{Sh}(p, \, q)$ denotes the set of $(p, \, q)$-shuffles.

\bigskip

\textit{Koszul sign.}
The Koszul sign
$\kappa(\sigma; \, D_1, \,  \cdots, \, D_n)$ of
a permutation $\sigma$ of 
homogeneous elements $D_1, \cdots , D_n$
of a graded object $V=\oplus_{i\in\Z}V^{i}$
is determined by the equality
\begin{equation*}
D_1 {\odot_{\mathcal{A}}} \cdots {\odot_{\mathcal{A}}}  D_n=\kappa(\sigma; \, D_1, \, \cdots, \, D_n)
D_{\sigma(1)}{\odot_{\mathcal{A}}} \cdots {\odot_{\mathcal{A}}}  D_{\sigma(n)}	
\end{equation*}
in the graded commutative algebra $S(V)$.
Unless specified otherwise,
we will use the abbreviated notation
$\kappa(\sigma)$.

\bigskip

\textit{Lie algebroids.}
In this paper `Lie algebroid' always means
`Lie $\fK$-algebroid'.

\subsection*{Acknowledgements}
We would like to thank Ping Xu and Mathieu Sti\'enon for fruitful discussions and useful comments.
We gratefully acknowledge the critical reviews by the two anonymous reviewers on earlier versions of the manuscript.

\section{Graded manifolds, polyvector fields and polydifferentials}\label{Section2}

We recall the notion of dg manifold.
For more about this subject, see \cite{M-S-X}.
By a finite dimensional \textbf{graded manifold}
$\cM$, we mean a pair
$(M, \cO_{\cM})$, where $M$ is a smooth manifold and
$\cO_{\cM}$ is a sheaf of graded commutative algebras over $M$
locally of the form
\[
\cO_{\cM}(U)
\simeq C^{\infty}(U, \fK)\otimes_{\fK} \hat{S}(W^{\vee}) \, ,
\]
for a finite dimensional graded vector space
$W$ over $\fK$.

We denote by $C^{\infty}_{\cM}$ the space of global sections of
$\cO_{\cM}$, by $T_{\cM}$ the tangent vector bundle over $\cM$, and by $\XX(\cM):~=\mathrm{Der}(C^{\infty}_{\cM})$ the space of vector fields on $\cM$. The degree of a homogeneous function $f\in C^{\infty}_{\cM}$ is denoted by $\widetilde{f}$. Similarly,
the degree of a homogeneous vector field $X \in \XX(\cM)$ is denoted by $\widetilde{X}$.

\subsection{$\smf_{\cM}$-modules and complexes}\label{Sec:gradedstuff}\

By a \textbf{$\smf_{\cM}$-module}, we mean a graded (left) $\smf_{\cM}$-module, say $\cMmodule$. The degree of an element $\xi\in \cMmodule$ is denoted by $\widetilde{\xi}$.  The tensor product of two $\smf_{\cM}$-modules $\cMmodule$ and $\cMmodule'$ is denoted by $\cMmodule\otimes_{\smf_{\cM}} \cMmodule'$.
By a morphism of $\smf_{\cM}$-modules $\phi:\cMmodule \rightarrow \cMmodule'$,
we mean a $\smf_{\cM}$-linear and degree $0$ map $\phi$ from $\cMmodule$ to $\cMmodule'$.

By a \textbf{$\smf_{\cM}$-complex} $(\Upsilon, \, \delta)$, we mean a sequence of $\smf_{\cM}$-modules $\Upsilon=\{\Upsilon^p\}_{p\in \Z}$ together with a
degree $0$ operator $\delta: ~\Upsilon^{\bullet}\to \Upsilon^{\bullet+1}$ which satisfies
$\delta^2=0$ and is  $\smf_{\cM}$-linear in the sense that
\[
\delta(f\xi)=(-1)^{\widetilde{f}}f\delta(\xi) \, , \quad
\forall \xi \in \Upsilon^{\bullet} \, ,
f\in \smf_{\cM} \, .
\]

  Let $\Upsilon^{p,q}$ be the degree $q$ subspace of $\Upsilon^p$. Thus, $\Upsilon$ is a
bigraded object
$$
\Upsilon=\{ \Upsilon^p \}_{p\in \Z}=
\{\oplus_{q\in \Z}\Upsilon^{p,q}\}_{p\in \Z} \, .
$$
It is convenient to denote such a
$\smf_{\cM}$-complex
$(\Upsilon \, , \delta)$
by a diagram of the form:~
\[\begin{array}{*{20}{c}}
  {}&{}&{\oplus}&{}&{\oplus}&{} \\
  {\cdots}&{\stackrel{\delta}{\longrightarrow}    }&{\Upsilon^{p, q+1}}&{\stackrel{\delta}{\longrightarrow}    }&{\Upsilon^{p+1, q+1}}&{\stackrel{\delta}{\longrightarrow}}&{\cdots} \\
  {}&{}&{  \oplus}&{}&{ \oplus}&{} \\
  {\cdots}&{\stackrel{\delta}{\longrightarrow}        }&{\Upsilon^{p,q}}&{\stackrel{\delta}{\longrightarrow}       }&{\Upsilon^{p+1,q}}&{\stackrel{\delta}{\longrightarrow}      }&{\cdots} \\
  {}&{}&{\oplus}&{}&{\oplus}&{}&{\quad\quad .}
\end{array}\]	
\

For $x \in \Upsilon^{p,q}$, we will refer to $p$ as the \textbf{horizontal degree}, $q$ the \textbf{vertical degree}, $(p,q)$ the bi-degree,
and $p+q$ the \textbf{total degree}. The total degree is also denoted by $|x|=p+q$.
We will call $\delta$ the \textbf{horizontal differential}
as it only raises the horizontal degrees.
Note that $\smf_{\cM}$-multiplications to
$\Upsilon$
affect the vertical degrees and
preserve the horizontal degrees.

By a morphism of $\smf_{\cM}$-complexes
$\varphi:~ (\Upsilon_1, \, \delta_1) \rightarrow (\Upsilon_2, \,  \delta_2)$, we mean
a degree $(0,0)$ morphism of bigraded
vector spaces
$\varphi:~ \Upsilon_1^{\bullet, \bullet} \rightarrow \Upsilon_2^{\bullet, \bullet}$ such that
the $(p,q)$-components
$\varphi^{p,q}:~ \Upsilon_1^{p,q} \rightarrow \Upsilon_2^{p,q}$ satisfy the following conditions:
\begin{itemize}
\item[(1)]for each  $p \in \Z$, $\varphi^{p}=\oplus_{q\in \Z}\, \varphi^{p,q}:~ \Upsilon_1^{p,\bullet} \rightarrow \Upsilon_2^{p,\bullet}$ is a
$\smf_{\cM}$-linear map;
\item[(2)]
$\varphi$ intertwines $\delta_1$ and
$\delta_2$:
$\delta_2 \circ \varphi^{p,q}=\varphi^{p+1,q} \circ \delta_1$.
\end{itemize}

We now describe the tensor product $(\Upsilon_1 \widetilde{\otimes}  \Upsilon_2,  \, \widetilde{\delta}   )  $ of two $\smf_{\cM}$-complexes $(\Upsilon_1, \, \delta_1)$ and $(\Upsilon_2, \, \delta_2)$. Here and in the sequel, the symbol $\widetilde\otimes$ is reserved for such a tensor product.
The $n$-th horizontal term of $\Upsilon_1 \widetilde{\otimes}  \Upsilon_2$ is the $\smf_{\cM}$-module
$$(\Upsilon_1 \widetilde{\otimes}  \Upsilon_2)^n=\oplus_{p+q=n}(\Upsilon_1^p[-p] \otimes_{\smf_{\cM}}\Upsilon_2^q[-q])[n]\,.$$
The $\smf_{\cM}$-module structure of $(\Upsilon_1 \widetilde{\otimes}  \Upsilon_2)^n$ is given by
\[
f(x\widetilde{\otimes}y):~~=
fx \widetilde{\otimes} y = (-1)^{\widetilde{f}|x|}x \widetilde{\otimes} fy \, ,
\quad \forall f\in \smf_{\cM} \, ,
\quad x\widetilde{\otimes} y \in  (\Upsilon_1 \widetilde{\otimes}  \Upsilon_2)^n \, .
\]
We also have a bigrading on $\Upsilon_1 \widetilde{\otimes} \Upsilon_2$.  The convention is that 
$(\Upsilon_1 \widetilde{\otimes} \Upsilon_2)^{n,m}$ consists of $\fK$-linear combinations of elements of the form $x\widetilde\otimes y$, for $x\in \Upsilon_1^{p,r}$ and $y\in\Upsilon_2^{q,s}$ such that $p+q=n, r+s=m$.
Finally, the horizontal differential $\widetilde\delta:~(\Upsilon_1 \widetilde{\otimes}  \Upsilon_2)^n\to (\Upsilon_1 \widetilde{\otimes}  \Upsilon_2)^{n+1}$ is defined by:
\[
\widetilde{\delta} (x \widetilde{\otimes} y):~~=\delta_1(x) \widetilde{\otimes} y + (-1)^{|x|}x \widetilde{\otimes} \delta_2(y) \, .
\]


For a $\smf_{\cM}$-complex $(\Upsilon, \, \delta)$, we always regard the symmetric group 
$\mathrm{S}_p$ 
as acting on $\widetilde{\otimes}^{p}(\Upsilon, \, \delta)$  
 by
$$
\sigma(x_1 \widetilde{\otimes} x_2 \widetilde{\otimes}\cdots \widetilde{\otimes} x_p)=
\kappa(\sigma; \, x_1, \, \cdots, \, x_p)
x_{\sigma(1)} \widetilde{\otimes} x_{\sigma(2)} \widetilde{\otimes}\cdots \widetilde{\otimes} x_{\sigma(p)} \, , \quad \sigma\in \mathrm{S}_p \, ,
$$ where each $x_i\in \Upsilon^{\bullet,\bullet}$ has homogenous total degree. 
 To  compute the Koszul sign $\kappa(\sigma; \, x_1, \, \cdots, \, x_p)$, one should use  the total degrees of $x_1$, $\cdots, x_p$.

Let us denote by
\[
S^{\bullet}(\Upsilon, \, \delta):=\oplus_{p\geqslant 0}\widetilde{\odot}^{p}(\Upsilon, \, \delta)
\]
the free symmetric algebra generated by
$(\Upsilon, \, \delta)$ over $\smf_{\cM}$. 
In other words, 
the $\smf_{\cM}$-complex 
$\widetilde{\odot}^{p}(\Upsilon, \, \delta)$ is the subspace of $ \widetilde{\otimes}^{p}(\Upsilon, \, \delta)$ which is fixed under the above action of the symmetric group $\mathrm{S}_p$. 
We call $\widetilde{\odot}^{p}(\Upsilon, \, \delta)$ the $p$-th symmetric product of $(\Upsilon, \, \delta)$ over $\smf_{\cM}$.

The total space of a $\smf_{\cM}$-complex
$(\Upsilon, \, \delta)$ is the direct sum
$$
\tot\Upsilon=\oplus_{p\in \Z} \Upsilon^p[-p]
\, .
$$
With respect to the total degree, $\tot\Upsilon$ is   a $\smf_{\cM}$-module.
Moreover, $\delta$ can be considered as  a degree $(+1)$ $\smf_{\cM}$-endomorphism on $ \tot\Upsilon$.
In the sequel, for $x\in \Upsilon$, the notation $\overline{x}$ denotes the corresponding element in $\tot\Upsilon$.
In particular, if $x\in \Upsilon^{p,q}$,
we have $\widetilde{x}=q$ and
$\overline{x} \in \Upsilon^p [-p]$.
We also have
$$\tot(\Upsilon_1 \widetilde{\otimes}  \Upsilon_2)=(\tot\Upsilon_1)\otimes_{\smf_{\cM}}(\tot\Upsilon_2) \, ,$$
for any two $\smf_{\cM}$-complexes
$(\Upsilon_1, \, \delta_1)$ and
$(\Upsilon_2, \, \delta_2)$.

\subsection{The functor $\mathrm{L}($\underline{$\, \cdot \,$}$)$ and the PBW map}\label{the functor L and the PBW map}\

Let $V$ be a $\smf_{\cM}$-module.
Denote by
$\underline{V}=(\underline{V}, \,  \,  \delta=0)$
the $\smf_{\cM}$-complex whose horizontal differential is trivial, and
$\underline{V}$ consists of
a single copy of $V$ concentrated in horizontal degree $(+1)$.
In other words,
$\underline{V}^{\, 1}=V$ and $\underline{V}^{\, n}=0$ for $n\neq 1$.
The tensor algebra of $\underline{V}$ is denoted by
\begin{equation*}
T^{\bullet}\underline{V}:~=\oplus_{p\geqslant 0} \underline{V}^{\widetilde{\otimes}p} \, ,
\end{equation*}
which is also
a $\smf_{\cM}$-complex with the trivial horizontal differential.
Note that in $T^{\bullet}\underline{V}$,
a homogeneous element $D\in T^i \underline{V}$ has total degree $|D|=\widetilde{D}+i$.
The space $T^{\bullet}\underline{V}$ is also a Lie algebra object in the category of $\smf_{\cM}$-complexes.
The canonical Lie bracket of any two homogeneous elements
$D\in T^i \underline{V}$ and $E\in T^j \underline{V}$ is defined by:
\begin{equation}\label{freeLiebracket}
\textup{\textlbrackdbl}D,E\textup{\textrbrackdbl}=D\widetilde{\otimes} E - (-1)^{|D| |E|}E\widetilde{\otimes} D \in T^{i+j}\underline{V} \, .
\end{equation}
Note that $\nbk$ is  $\smf_{\cM}$-bilinear.

We introduce a functor,
denoted by
$\functorL(\underline{\, \cdot \,})$ and called the \textbf{shifted free Lie algebra functor},
from
the category of $\smf_{\cM}$-modules
to the category of
$\smf_{\cM}$-complexes.
For each $\smf_{\cM}$-module $V$,
define $\functorL(\underline{V})$ to be the smallest Lie subalgebra of
$T^{\bullet}\underline{V}$ containing $\underline{V}=T^1 \underline{V}$.
In other words, the space
$\functorL(\underline{V})$ is spanned
(over $\fK$) by elements of the form
$\textup{\textlbrackdbl} D_1,\cdots , \textup{\textlbrackdbl}
D_{n-1}, D_{n} \textup{\textrbrackdbl},\cdots  \textup{\textrbrackdbl}$ with
$D_1 , \cdots ,D_n \in \underline{V}$.
For each morphism $\phi:~ V \rightarrow W$
of $\smf_{\cM}$-modules,
denote by $\underline{\phi}:~ \underline{V} \rightarrow \underline{W}$ the corresponding morphism of $\smf_{\cM}$-complexes, and
define
$\functorL(\underline{\phi}):~ \functorL(\underline{V}) \rightarrow \functorL(\underline{W})$ to be the unique Lie algebra morphism
which extends $\underline{\phi}$.

We note that both
$\tot T^{\bullet}\underline{V}$ and $\tot \functorL(\underline{V})$
are Lie algebra objects in the category of
$\smf_{\cM}$-modules.
The Lie bracket on
$\tot T^{\bullet}\underline{V}$ (resp. $\tot \functorL(\underline{V})$)
is also denoted by $\nbk$, i.e.,
\begin{equation*}\label{freeLiebracketonTOTALcomplex}
\textup{\textlbrackdbl}\overline{D},\overline{E}\textup{\textrbrackdbl}=\overline{D}\otimes_{\smf_{\cM}} \overline{E} - (-1)^{|{D}||{E}|}\overline{E}\otimes_{\smf_{\cM}} \overline{D} \, ,
\end{equation*}
for $\overline{D}, \, \overline{E} \in \tot T^{\bullet}\underline{V}$ (resp. $\overline{D}, \, \overline{E} \in \tot \functorL(\underline{V})$).\\

We have the standard \textbf{Poincar\'e-Birkhoff-Witt map}
\begin{equation}\label{Eqt:pbwstandard}
\pbw:~~ S^{\bullet}(\functorL(\underline{V})) \rightarrow T^{\bullet}\underline{V} \, ,
\end{equation}
\[
\pbw(D_1 \widetilde{\odot} \cdots  \widetilde{\odot} D_n)
:~~=\frac{1}{n!}\sum_{\sigma \in S_{n}}
\kappa(\sigma) D_{\sigma(1)}\widetilde{\otimes}\cdots
\widetilde{\otimes} D_{\sigma(n)}
\]
which is an isomorphism of
$\smf_{\cM}$-complexes (with trivial horizontal differentials).
Here $D_1, \cdots, D_n$
$\in \functorL(\underline{V})$ are homogeneous elements,
and $\kappa(\sigma)$ is the Koszul sign determined by the equality
\begin{equation}\label{Koszulsign}
D_1\widetilde{\odot}\cdots \widetilde{\odot} D_n=\kappa(\sigma)
D_{\sigma(1)}\widetilde{\odot}\cdots \widetilde{\odot} D_{\sigma(n)}
\end{equation}
in $S^{\bullet}(\functorL(\underline{V}))$.

\subsection{The Schouten-Nijenhuis algebra $\Tpp$ of polyvector fields}\

Let $\XX(\cM)$ be the space of vector fields on the graded manifold $\cM$.
We consider $\Tp:~=\underline{\XX(\cM)}$.
In other words,
$\Tp$ is the $\smf_{\cM}$-complex
$\Tp=(\Tp, \, \delta=0)$
which consists of a single copy of
$\XX(\cM)$ concentrated in horizontal degree
$(+1)$.
An element in $\Tp$ is just some vector field $X\in \XX(\cM)$. The only difference is that it has a double grading $(1,\widetilde{X})$. If we treat $X \in \Tp$, then $\widetilde{X}$ means the original degree of $X$ as a vector field and $\abs{X}=1+\widetilde{X}$ means the total degree.
Meanwhile we denote by
$\overline{X}$ the element in
$\tot\Tp=\XX(\cM)[-1]$ that corresponds to $X$.

The $\smf_{\cM}$-complex of \textbf{polyvector fields}, denoted by $\Tpp$, is the symmetric algebra of $\Tp$:
\[
\Tpp=(\Tpp, \, \delta=0)
:~~= S^{\bullet}(\Tp, \, \delta=0) \, .
\]
The symmetric product in $\TT_{\mathrm{poly}}$ is denoted by $\widetilde{\odot}$.
Elements of
$\TT_{\mathrm{poly}}^{n}:~=S^n(\Tp)$ are
called $n$-polyvector fields on $\cM$.
They are finite sums of elements of the form
\[
X=X_1\widetilde{\odot} X_2\widetilde{\odot} \cdots \widetilde{\odot} X_n \, ,
\]
where $X_i\in \TT_{\mathrm{poly}}^{1}$ are homogeneous elements.
The horizontal degree of $X$ is $n$,
whereas the vertical degree of $X$ is $\widetilde{X}_1+\cdots+\widetilde{X}_n$, and the total degree
of $X$ is $|X|=n+\widetilde{X}_1+\cdots+\widetilde{X}_n$.\\

The space $\Tpp$ is a degree $(-1)$ Poisson algebra as it admits a canonical
Schouten-Nijenhuis algebra structure.
See \cite{B-C-S-X3} for more details.

\subsection{The Hopf algebra $\Dpp$ of polydifferential operators}\label{Dpoly}\

Let $\cM$ be a graded manifold and let
$\cD_{\cM}$ denote its algebra of differential operators.
The tangent bundle $T_{\cM}$ is a Lie algebroid over $\cM$. Its universal enveloping algebra coincides with
$\cD_{\cM}$.
Similar to the relation between
$\XX(\cM)$ and $\TT_{\mathrm{poly}}^1$,
we use
$\DD_{\mathrm{poly}}^1$ to denote the bigraded object which consists of a single copy of
$\cD_{\cM}$ concentrated in horizontal degree
$(+1)$.
If a differential operator $D \in \cD_{\cM}$ is considered as in $\DD_{\mathrm{poly}}^1$,
then it has horizontal degree $(+1)$, vertical degree $\widetilde{D}$, and total degree $|D|=1+\widetilde{D}$.
The notation $\overline{D}$ denotes the corresponding element in
$\tot\DD_{\mathrm{poly}}^1=\cD_{\cM}[-1]$.

Elements of the space
$
\DD_{\mathrm{poly}}^n:~~=\widetilde{\otimes}^n \DD_{\mathrm{poly}}^1
$
are called $n$-\textbf{polydifferential operators}.
Denote by $\DD_{\mathrm{poly}}^{n,m}\subset \DD_{\mathrm{poly}}^n$ the subset consisting of polydifferential operators
with horizontal degree $n$ and vertical degree $m$.
In other words, $\DD_{\mathrm{poly}}^{n,m}$ is spanned (over $\fK$) by elements of the form
\[
D=D_1\widetilde{\otimes}\cdots \widetilde{\otimes}D_n
\]
where $D_i \in \DD_{\mathrm{poly}}^1$ are homogeneous elements and $\widetilde{D}_1+ \cdots  + \widetilde{D}_n =m$. Note that the total degree $|D|$ equals $n+m$.

A $n$-polydifferential operator $D\in \DD_{\mathrm{poly}}^n$ can be identified with a map
\[
D: \quad \otimes_{\fK}^n \smf_{\cM} \rightarrow \smf_{\cM}
\]
such that for each position $i$  and
$f_1, \cdots ,f_{i-1}, f_{i+1}, \cdots ,f_n\in \smf_{\cM}$, the map
\[\smf_{\cM}\to \smf_{\cM},\quad
f \mapsto 
D(f_1, \cdots  f_{i-1}, f, f_{i+1}, \cdots ,f_n)
\]
is a differential operator.
Following the convention of degrees and signs of
Tamarkin and Tsygan \cite{Tamarkin-Tsygan},
if $D=D_1\widetilde{\otimes} \cdots  \widetilde{\otimes}D_n \in \DD_{\mathrm{poly}}^n$, where
$D_i \in \DD_{\mathrm{poly}}^1$ are homogeneous,
then $D$ viewed as a map $\otimes_{\fK}^n \smf_{\cM} \rightarrow \smf_{\cM}$ is given by
\[
D(f_1, f_2, \cdots, f_n):~~=(-1)^{\star}D_1(f_1)\cdots  D_n(f_n) \, ,
\]
where
$\star=\sum_{i=1}^{n}\sum_{j=i+1}^{n}(\widetilde{f}_i +1)|D_j|$.

\begin{remark}\label{two f}
By convention,
we identify
$\DD_{\mathrm{poly}}^0$ with $\smf_{\cM}$.
However, a function $f \in \smf_{\cM}$
can also be regarded an element in
$\DD_{\mathrm{poly}}^1$, i.e., the differential operator
$f' \mapsto f\cdot f'$,
$\forall f' \in \smf_{\cM}$.
To avoid ambiguity, instead of writing
$f\in \DD_{\mathrm{poly}}^1$,
we use the notation
$f\cdot\id_{\smf_{\cM}}$
to denote the aforesaid differential operator.	
\end{remark}

The space of polydifferential operators
\[
\DD_{\mathrm{poly}}=\oplus_{n \geqslant 0}\DD_{\mathrm{poly}}^n
=\oplus_{n \geqslant 0, m\in \Z}\DD_{\mathrm{poly}}^{n,m}
\]
admits the natural cup product
\[
\cup:~~ \DD_{\mathrm{poly}}^p \widetilde{\otimes} \DD_{\mathrm{poly}}^q
\rightarrow
\DD_{\mathrm{poly}}^{p+q} \, , \qquad
D\cup E :~~ = D\widetilde{\otimes}E \, .
\]
By conventions in \cites{Tamarkin-Tsygan, Gerstenhaber}, $D\cup E$
in terms of a map from $\otimes_{\fK}^{p+q}\smf_{\cM}$
to $\smf_{\cM}$ reads
\begin{equation}\label{cupproductformula}
(D\cup E)(f_1, \cdots  f_{p+q})=
(-1)^{\flat}D(f_1,\cdots ,f_p)E(f_{p+1},\cdots ,f_{p+q}) \, ,
\end{equation}
for homogeneous elements
$f_1, \cdots , f_{p+q}\in \smf_{\cM}$, where
$\flat=|E|(\widetilde{f}_1+\cdots +\widetilde{f}_{p} + p)$.
Note that $\DD_{\mathrm{poly}}^{p,r}\cup \DD_{\mathrm{poly}}^{q,s} \subset \DD_{\mathrm{poly}}^{p+q, r+s}$.
\newline

The space $\Dp$ has a coassociative coproduct:
\begin{equation}\label{Delta}
\begin{split}	
\Delta:~
&  \Dp  \rightarrow \DD_{\mathrm{poly}}^2=\Dp\widetilde{\otimes}\Dp \, , \\
\Delta(D)(f_1, f_2):~
&
=(-1)^{\widetilde{f_1}+1}D(f_1 f_2) , \qquad \forall D\in \Dp, \quad \forall f_1, f_2 \in \smf_{\cM} \, .
\end{split}
\end{equation}
Note that we have $\Delta(X)=X \widetilde{\otimes} \id_{\smf_{\cM}} + (-1)^{\widetilde{X}}\id_{\smf_{\cM}} \widetilde{\otimes} X$,
for all homogeneous elements $X \in \Tp$.
A general formula of $\Delta(D)$, for $D\in \Dp$, is given by Proposition \ref{Deltaformula} in the Appendix.
	
The \textbf{Hochschild differential}
$d_H:~ \DD_{\mathrm{poly}}^{n} \rightarrow
\DD_{\mathrm{poly}}^{n+1}$ is defined as follows:
\begin{equation}\label{Hochschild differential definition}
\begin{split}
d_H(D_1\widetilde{\otimes} \cdots  \widetilde{\otimes} D_n)= & \,
(-1)^{\sum_{i=1}^{n}|D_i|}
\bigl(D_1 \widetilde{\otimes}\cdots  \widetilde{\otimes} D_n \widetilde{\otimes} \id_{\smf_{\cM}} 
\\
& \quad - \sum_{i=1}^{n} (-1)^{\sum^{n}_{j=i+1}|D_j|}D_1\widetilde{\otimes} \cdots \widetilde{\otimes} D_{i-1} \widetilde{\otimes} \Delta(D_i) \widetilde{\otimes} D_{i+1}  \widetilde{\otimes} \cdots \widetilde{\otimes} D_{n} \\
& \quad\quad
-(-1)^{\sum_{i=1}^n|D_i|}\id_{\smf_{\cM}}\widetilde{\otimes} D_1 \widetilde{\otimes}\cdots  \widetilde{\otimes} D_n \bigr)\, ,
\end{split}
\end{equation}
where
$D_i \in \DD_{\mathrm{poly}}^1$ are homogeneous elements.
In particular, for a homogeneous element
$D\in \DD_{\mathrm{poly}}^1$, we have
\begin{equation}\label{coproduct}
d_H(D)=(-1)^{|D|}(D\widetilde{\otimes} \id_{\smf_{\cM}} -\Delta(D) -(-1)^{|D|}\id_{\smf_{\cM}}\widetilde{\otimes} D) \, .
\end{equation}
For $f\in \DD_{\mathrm{poly}}^0=\smf_{\cM}$,
we have $d_H(f)=0$.
For
$f\cdot \id_{\smf_{\cM}} \in \DD_{\mathrm{poly}}^1$ (see Remark \ref{two f}),
we have
$d_H(f\cdot \id_{\smf_{\cM}})=
(-1)^{\widetilde{f}+1}f\cdot (\id_{\smf_{\cM}}\widetilde{\otimes}\id_{\smf_{\cM}})$.

It is standard that
the pair $(\Dpp,  \,  \delta=d_H)$ is a $\smf_{\cM}$-complex,
which we call the $\smf_{\cM}$-complex of polydifferential operators.
The following lemma is needed.
\begin{lemma}\cites{Gerstenhaber, Tamarkin-Tsygan}\label{Hochschildandcupproduct}
For all homogeneous elements $D,E\in \DD_{\mathrm{poly}}$, we have
\begin{equation*}
d_H(D\cup E)=d_H D\cup E +(-1)^{|D|}D \cup d_H E.
\end{equation*}	
\end{lemma}\

The \textbf{Gerstenhaber product} of two homogeneous elements $D\in \DD_{\mathrm{poly}}^n$ and $E \in \DD_{\mathrm{poly}}^m$, denoted by
$D\circ E \in \DD_{\mathrm{poly}}^{n+m-1}$, is the operator
\begin{equation}\label{Gerstenhaberproductformula}
(D\circ E)(f_1,\cdots ,f_{n+m-1})=
\sum_{j\geqslant 0}(-1)^{\star_j}
D(f_1,\cdots , f_j,E(f_{j+1},\cdots , f_{j+m}),\cdots ) \, ,
\end{equation}
where $\star_j=(|E|+1)(\widetilde{f}_1+\cdots +\widetilde{f}_j+j)$.
The \textbf{Gerstenhaber bracket} in
$\DD_{\mathrm{poly}}$ is defined by
\begin{equation*}\label{Gernstenhaberbracket}
[D,\, E]=D\circ E-(-1)^{(|D|+1)(|E|+1)}E\circ D \, ,
\end{equation*}
for homogeneous elements
$D\in \DD_{\mathrm{poly}}^n$ and $E \in \DD_{\mathrm{poly}}^m$.
It satisfies the following well known properties (see \cite{Tamarkin-Tsygan}):
\begin{equation*}
[D,E]=-(-1)^{(|D|+1)(|E|+1)}[E,D] \, ,
\end{equation*}
\begin{equation}\label{Jacobi}
\bigl[D,[E,F]\bigr]=\bigl[[D,E],F\bigr]
+(-1)^{(|D|+1)(|E|+1)}\bigl[E,[D,F]\bigr] \, .
\end{equation}\

Consider the multiplication
$\mult:~ \otimes_{\fK}^2\smf_{\cM} \rightarrow \smf_{\cM}$ defined by
\begin{equation*}\label{multsign}
\mult(f_1, f_2)
=(-1)^{\widetilde{f}_1}f_1 f_2 \, , \quad \forall f_1, f_2 \in \smf_{\cM} \, .
\end{equation*}
(Our convention of $\mult$ coincides with \cite{Tamarkin-Tsygan}.)
By Equation  \eqref{cupproductformula},
the operator $\mult$ is identically
the element
$- \mathrm{id}_{\smf_{\cM}}\widetilde{\otimes}\mathrm{id}_{\smf_{\cM}} \in \DD_{\mathrm{poly}}^{2, 0}$.
The Hochschild differential $d_H$ can also be expressed in terms of $\mult$ and the Gerstenhaber bracket $[\,\cdot\, , \cdot \,]$
(see \cite{Tamarkin-Tsygan}):
\begin{equation}\label{dH via bracket}
d_H=[\mult, \, \cdot \,]:~~ \DD_{\mathrm{poly}}^{n,\bullet} \rightarrow \DD_{\mathrm{poly}}^{n+1,\bullet} \, .
\end{equation}\

Finally, $(\DD_{\mathrm{poly}}, \, d_H)$ is a Hopf algebra object in the category of $\smf_{\cM}$-complexes:
\begin{itemize}
\item
The multiplication is the cup product
$\cup$ (see Equation  \eqref{cupproductformula}).

\item
The comultiplication is the shuffle coproduct
\[
\widetilde{\Delta}:~~ \DD_{\mathrm{poly}} \rightarrow \DD_{\mathrm{poly}} \widetilde{\otimes} \DD_{\mathrm{poly}} \, ,
\]
\begin{equation*}\label{coproductformulaofDpoly}
\widetilde{\Delta}(D_1\widetilde{\otimes} \cdots  \widetilde{\otimes} D_n)=\sum_{\substack{p+q=n, \\ \sigma\in \mathrm{Sh}(p, \, q)}}\kappa(\sigma)(D_{\sigma(1)}\widetilde{\otimes} \cdots  \widetilde{\otimes} D_{\sigma(p)}) \widetilde{\bigotimes}
(D_{\sigma(p+1)}\widetilde{\otimes} \cdots  \widetilde{\otimes} D_{\sigma(p+q)}) \, ,
\end{equation*}
where $D_i\in \DD_{\mathrm{poly}}^1$ are homogeneous elements,
$\mathrm{Sh}(p, \, q)$ stands for the set of
$(p,q)$-shuffles, and
$\kappa(\sigma)$ is the Koszul sign
(see Equation  \eqref{Koszulsign}).

\item
The unit is the natural inclusion
$\eta: \smf_{\cM}=\DD_{\mathrm{poly}}^0\hookrightarrow  \Dpp$.

\item
The counit
$\varepsilon: \DD_{\mathrm{poly}} \twoheadrightarrow  \DD_{\mathrm{poly}}^0=\smf_{\cM}$ is the natural projection.

\item
The antipode is the map
\[
t: \DD_{\mathrm{poly}} \rightarrow \DD_{\mathrm{poly}} \, ,
\]
\[
t(D_1 \widetilde{\otimes} D_2 \widetilde{\otimes} \cdots  \widetilde{\otimes} D_n)
=
(-1)^{\natural}D_n \widetilde{\otimes} \cdots \widetilde{\otimes} D_2 \widetilde{\otimes} D_1 \, ,
\]
where $\natural=\sum_{i=0}^{n-1}|D_{n-i}|(|D_1|+ \cdots  +|D_{n-i-1}|)$.
\end{itemize}


\section{Dg manifolds}\label{Section3}
\subsection{Dg manifolds and dg modules}\

A \textbf{dg manifold} is a pair $(\cM, \, Q)$ where $\cM$ is a graded manifold and $Q\in \XX(\cM)$ is a homological vector field, i.e. a degree $(+1)$ vector field such that $[Q,Q]=2Q^2=0$.
A \textbf{dg vector bundle} is a vector bundle object in the category of dg manifolds. In other words, a dg vector bundle over $(\cM, \, Q)$ is a pair $(\cE, \, L_Q)$ where
 $\cE$ is a vector bundle over $\cM$ and $L_Q:~ \Gamma(\cE) \rightarrow \Gamma(\cE)$
is a degree $(+1)$ differential, i.e.,
\[
L_Q^2=0
\]
and
\[
L_Q(fs)=Q(f)s+(-1)^{\widetilde{f}}fL_Q(s),
\]
for all homogeneous elements
$f \in \smf_{\cM}$  and
$s\in \Gamma(\cE)$.

In this note, we abuse the notation $L_Q$ to denote differentials of all dg vector bundles,   unless the differential  is  otherwise specified. For example, the tangent space $ T_{\cM} $ of $(\cM, \, Q)$ is a dg vector bundle, whose differential $L_{Q}$ equals $[Q, \cdot\,]$.

\begin{definition}
A \textbf{dg module over $(\cM, \, Q)$}, or a
\textbf{$(\smf_{\cM}, \, Q)$-module}, is a pair $(\cMmodule, \, L_Q)$ where
$\cMmodule$ is a  $\smf_{\cM}$-module and
$L_Q:~ \cMmodule \rightarrow \cMmodule$ is a degree $(+1)$ differential, i.e.,
\[
L_Q^2=0
\]
and
\[
L_Q(f\xi)=Q(f)\xi+(-1)^{\widetilde{f}}fL_Q(\xi)
\]
for all homogeneous elements $f\in \smf_{\cM}$ and $\xi \in \cMmodule$.
 \end{definition}

Again, the notation $L_Q$ is abused to denote differentials of all $(\smf_{\cM}, \, Q)$-modules,   unless stated otherwise.
For a dg vector bundle $(\cE, \, L_Q)$ over
$(\cM, \, Q)$, the pair
$(\Gamma(\cE), \, L_Q)$ is a
$(\smf_{\cM}, \, Q)$-module.
In particular,
$ (\XX(\cM), \, L_Q=[Q,\cdot\,])$
is a $(\smf_{\cM}, \, Q)$-module.

A morphism of $(\smf_{\cM}, \, Q)$-modules from $(\cMmodule, L_{Q})$ to
$(\cMmodule',L_{Q})$ is a $\smf_{\cM}$-linear and degree $0$ map $\phi:~ \cMmodule \rightarrow \cMmodule'$ that satisfies
$\phi \circ L_{Q}=L_{Q} \circ \phi$.
A morphism
$\phi:(\cMmodule, L_{Q}) \rightarrow (\cMmodule', L_{Q})$
is called a \textbf{quasi-isomorphism} if the morphism of cohomology groups
$H(\phi): H(\cMmodule, \, L_Q) \rightarrow H(\cMmodule', \, L_Q)$
induced by $\phi$ is an isomorphism.

Fix a dg manifold $(\cM, \, Q)$.
The category of $(\smf_{\cM}, \, Q)$-modules
is denoted by $\Mod$.
The \textbf{homology category $\mathrm{H}(\Mod)$}
is the category whose objects are
$(\smf_{\cM}, \, Q)$-modules,
and whose morphisms are morphisms of
$(\smf_{\cM}, \, Q)$-modules modulo cochain homotopies (see \cites{G-Z,Keller}).
The \textbf{homotopy category $\Pi(\Mod)$}
is the Gabriel-Zisman localization of
$\Mod$ by the set of quasi-isomorphisms
in $\Mod$ (see \cite{G-Z}).
We have a sequence of natural functors between these categories:
\[
\Mod \longrightarrow \mathrm{H}(\Mod) \longrightarrow \Pi(\Mod).
\]

\subsection{Dg complexes}\label{Sec:dgcomplexes}

\begin{definition}
A \textbf{dg complex
over $(\cM, \, Q)$},
or a \textbf{$(\smf_{\cM}, \, Q)$-complex},
is a $\smf_{\cM}$-complex
$(\Upsilon, \, \delta)$ such that each $\Upsilon^p$ ($p\in\Z$) is endowed with a
$(\smf_{\cM}, \, Q)$-module structure:
$$
L_Q:~\Upsilon^{p,\bullet} \rightarrow \Upsilon^{p,\bullet+1}
$$
which anti-commutes with  $\delta$, i.e.,
\begin{equation}\label{Eqn:LQdeltacompatible}
\delta \circ L_Q + L_Q \circ \delta =0.
\end{equation}
\end{definition}
So,
a $(\smf_{\cM}, \, Q)$-complex is indeed a triple
$(\Upsilon, \, L_Q,  \,  \delta)$.
It is convenient to denote such a
$(\smf_{\cM}, \, Q)$-complex
 by a commutative
 diagram:~
 $$
 \begin{tikzcd}[row sep=3em, column sep=2em]
 	& \, & \, &   \\
 	\cdots \arrow[r, "\delta"] & \Upsilon^{p, \, q+1} \arrow[r, "\delta"] \ar[u, shorten >= 1.5em] & \Upsilon^{p+1, \, q+1} \arrow[r, "\delta"] \ar[u, shorten >= 1.5em] & \cdots \\
 	\cdots \arrow[r, "\delta"] &
 	\Upsilon^{p, \, q} \arrow[u, "(-1)^p L_Q"'] \arrow[r, "\delta"] & \Upsilon^{p+1, \, q} \arrow[u, "(-1)^{p+1}L_Q"'] \arrow[r, "\delta"]  & \cdots \\
 	& \, \ar[u, shorten <= 1.5em] & \, \ar[u, shorten <= 1.5em] &
 \end{tikzcd}
$$

We call $\delta$ the \textbf{horizontal differential}  and $L_Q$ the \textbf{vertical differential}.
With respect to the total grading
$(\tot \Upsilon)^n=\oplus_{p+q=n}\Upsilon^{p,q}$,
$(\tot \Upsilon, \, L_Q+\delta)$
is a $(\smf_{\cM}, \,Q)$-module,
by the compatibility condition
Equation  \eqref{Eqn:LQdeltacompatible}.

A morphism of $(\smf_{\cM}, \, Q)$-complexes
$\varphi:~ (\Upsilon_1, \, L_Q,  \,  \delta_1) \rightarrow (\Upsilon_2, \, L_Q,  \,  \delta_2)$
is a degree $(0,0)$ morphism of bigraded
vector spaces
$\varphi:~ \Upsilon_1^{\bullet, \bullet} \rightarrow \Upsilon_2^{\bullet, \bullet}$ such that
the $(p,q)$-components
$\varphi^{p,q}:~ \Upsilon_1^{p,q} \rightarrow \Upsilon_2^{p,q}$ satisfy the following conditions:
\begin{itemize}
\item[(1)]for each  $p \in \Z$, $\varphi^{p}=\oplus_{q\in \Z}\, \varphi^{p,q}:~ \Upsilon_1^{p,\bullet} \rightarrow \Upsilon_2^{p,\bullet}$ is a morphism of
$(\smf_{\cM}, \, Q)$-modules, i.e.,
$\varphi^{p}$ is $\smf_{\cM}$-linear and satisfies
$L_Q \circ \varphi^{p}=\varphi^{p}\circ L_Q$;
	
\item[(2)]
$\varphi$ intertwines $\delta_1$ and
$\delta_2$:
\[
\delta_2 \circ \varphi^{p,q}=\varphi^{p+1,q} \circ \delta_1.
\]
\end{itemize}
Denote the
\textbf{category of $(\smf_{\cM}, \, Q)$-complexes} by $Ch(\Mod)$.
Given a morphism $\varphi:~ (\Upsilon_1, \, L_Q,  \,  \delta_1) \rightarrow (\Upsilon_2, \, L_Q,  \,  \delta_2)$ of
$(\smf_{\cM}, \, Q)$-complexes as above, it is clear that
$\tot \varphi:~= \oplus_{p,q \in \Z}\,\varphi^{p,q}
:~
(\tot\Upsilon_1, \, L_Q+\delta_1)\to (\tot\Upsilon_2, \, L_Q+\delta_2)$
is a morphism of
$(\smf_{\cM}, \, Q)$-modules.
Accordingly, $\tot$is a functor from the category
$Ch(\Mod)$ to the category $\Mod$.

\begin{definition}
A \textbf{quasi-isomorphism} $\varphi:~ (\Upsilon_1, \, L_Q,  \,  \delta_1) \rightarrow (\Upsilon_2, \, L_Q,  \,  \delta_2)$ of $(\smf_{\cM}, \, Q)$-complexes is a morphism in $Ch(\Mod)$ that induces a quasi-isomorphism
$\tot \varphi :~ (\tot\Upsilon_1,  L_Q+ \delta_1)
\rightarrow (\tot\Upsilon_2,  L_Q+\delta_2)$
of $(\smf_{\cM}, \, Q)$-modules.
The \textbf{derived category $D(\Mod)$} of
$(\smf_{\cM}, \, Q)$-complexes
is the Gabriel-Zisman localization of $Ch(\Mod)$ by the set of quasi-isomorphisms in $Ch(\Mod)$.
\end{definition}
The functor
$\tot:~ Ch(\Mod) \longrightarrow \Mod$
induces a functor
\[
\tot:~~ D(\Mod) \longrightarrow \Pi(\Mod) \, .
\]

We now explain the tensor product
 $(\Upsilon_1 \widetilde{\otimes}  \Upsilon_2, \, L_Q, \, \widetilde{\delta}   )$ of two
 $(\smf_{\cM}, \, Q)$-complexes
$(\Upsilon_1, \, L_Q, \, \delta_1)$
and $(\Upsilon_2, \, L_Q , \, \delta_2)$. The $\smf_{\cM}$-complex $(\Upsilon_1 \widetilde{\otimes}  \Upsilon_2,  \, \widetilde{\delta}   )$ is already defined in Section \ref{Sec:gradedstuff}. The vertical differential $L_Q$ is defined by the equation analogous to that of $\widetilde{\delta}$:
\[
L_Q(x \widetilde{\otimes} y ):~~=L_Q(x) \widetilde{\otimes} y + (-1)^{|x|}x \widetilde{\otimes} L_Q(y) , \,
\]
for all homogeneous $x \in \Upsilon_1$ and
$y \in \Upsilon_2$.
It is straightforward to verify that
$(\Upsilon_1 \widetilde{\otimes} \Upsilon_2, \, L_Q, \widetilde{\delta})$ is a
$(\smf_{\cM}, \, Q)$-complex.	

We also have the $p$-th symmetric product
$\widetilde{\odot}^p(\Upsilon, \, L_Q,  \,  \delta)$ of a $(\smf_{\cM}, \, Q)$-complex
$(\Upsilon, \, L_Q,  \,  \delta)$,  
which is the $(\smf_{\cM}, \, Q)$-complex that upgrades the $p$-th symmetric product 
$\widetilde{\odot}^p(\Upsilon, \, \delta)$
of $(\Upsilon, \, \delta)$ defined in
Section \ref{Sec:gradedstuff}.
We denote by
\[
S^{\bullet}(\Upsilon, \, L_Q,  \,  \delta):=\oplus_{p\geqslant 0}\widetilde{\odot}^{p}(\Upsilon, \, L_Q,  \,  \delta)
\]
the free symmetric algebra object generated over $\smf_{\cM}$ by
$(\Upsilon, \, L_Q,  \,  \delta)$ in $Ch(\Mod)$.\\

The following diagram of categories with natural functors as arrows
summarizes the relations between all the categories that we have introduced.
\[
\begin{xy}
(0, 20)*+{\Mod}="a";
(40, 20)*+{\mathrm{H}(\Mod)}="b";
(80,20)*+{\Pi(\Mod)}="c";
(0,0)*+{Ch(\Mod)}="d";
(80,0)*+{D(\Mod)}="e";
{\ar "a"; "b"};
{\ar "b"; "c"};
{\ar "d"; "e"};
{\ar_{\tot} "e";"c"};
{\ar_{\tot} "d";"a"};
\end{xy}
\]

\subsection{Dg complex structures on $\Dpp$ and $\Tpp$, and the HKR theorem}\label{Section  Hopf algebra}\

In the previous Section \ref{Section2},
we have introduced the $\smf_{\cM}$-complexes
$(\Tpp, \, \delta=0)$ and $(\Dpp, \, \delta=d_H)$
arising from a graded manifold $\cM$.
Now suppose that $\cM$ is equipped with a homological vector field $Q$ so that
$(\cM, \, Q)$ is a dg manifold.
The $\smf_{\cM}$-complex
$(\Tpp, \, \delta=0)$ is naturally upgraded
to a $(\smf_{\cM}, \, Q)$-complex
$(\Tpp, \, L_Q:~=[Q, \, \cdot\,], \, \delta=0)$.
We call $(\Tpp, \, L_Q,  \,  0)$
the
$(\smf_{\cM}, \, Q)$-complex of polyvector fields.
Moreover, we have
\begin{proposition}\label{Dpolyconstruction}
The triple
$(\Dpp, \, L_Q:~=[Q, \, \cdot \, ], \, \delta=d_H)$
is
a $(\smf_{\cM}, \, Q)$-complex, where $[\, \cdot \, , \cdot\,]$ is the Gerstenhaber bracket.
\end{proposition}
\begin{proof}
By the definition of $d_H$ in
Equation  \eqref{Hochschild differential definition}, it satisfies $d_{H}^{2}=0$ and is
$\smf_{\cM}$-linear.
To verify that $d_H$ and $L_Q$ are anti-commutative, we observe that the following equality holds:
\[
d_H\circ L_Q+L_Q \circ d_H
=\bigl[\mult,[Q, \cdot\,]\bigr]+\bigl[Q,[\mult, \cdot\,]\bigr]
=\bigl[[\mult,Q],\cdot\, \bigr]=0 \, .
\]
Here we have used Equations   \eqref{Jacobi} and
\eqref{dH via bracket}, and the obvious fact:
\[
[\mult, \, Q]=-[\id_{\smf_{\cM}}\widetilde{\otimes}\id_{\smf_{\cM}}, \, Q]=0 \, .
\]
\end{proof}
We call $(\Dpp, \, L_Q,  \,  d_H)$
the
$(\smf_{\cM}, \, Q)$-complex of polydifferential operators.
It is a Hopf algebra object in the category $Ch(\Mod)$ of $(\smf_{\cM}, \, Q)$-complexes.
The
multiplication $\cup$,
the comultiplication $\widetilde{\Delta}$,
the unit $\eta$,
the counit $\varepsilon$,
and the antipole $t$
are the same as those of
the Hopf algebra object $(\Dpp, \, d_H)$ in the category of $\smf_{\cM}$-complexes
(see the end of Section \ref{Dpoly}).
Taking the $\tot$functor to
$(\Dpp, \, L_Q, \, d_H)$,
we obtain a Hopf algebra object
$(\tot \Dpp, \, L_Q+d_H$)
in the category
$\Mod$ of $(\smf_{\cM}, \, Q)$-modules,
with the same structure maps as those of
$(\Dpp, \, d_H)$ and $(\Dpp, \, L_Q, \, d_H)$.\\

The \textbf{Hochschild-Kostant-Rosenberg map}
\[
\hkr:~~ (\Tpp, \, L_Q, \, 0) \rightarrow (\Dpp, \, L_Q, \, d_H )
\]
is a morphism of $(\smf_{\cM}, \, Q)$-complexes
defined by
\[
\hkr(X_1\widetilde{\odot}\cdots \widetilde{\odot} X_n)
:~~=\frac{1}{n!}\sum_{\sigma \in S_{n}}
\kappa(\sigma) X_{\sigma(1)}\widetilde{\otimes}\cdots
\widetilde{\otimes} X_{\sigma(n)},
\]
where the vector fields $X_i \in \Tp$ are
homogeneous elements and $\kappa(\sigma)$ is the Koszul sign (see Equation  \eqref{Koszulsign}).
The following Hochschild-Kostant-Rosenberg type theorem is due to Liao, Sti\'enon, and Xu (see \cite[Proposition 4.1]{L-S-X}).

\begin{theorem}\label{LSX}
The map $\hkr$ is a quasi-isomorphism of
$(\smf_{\cM}, \, Q)$-complexes
from $(\Tpp, \, L_Q, \, 0)$ to
$(\Dpp, \, L_Q, \, d_H)$.
\end{theorem}

\subsection{The universal enveloping algebra of $(\Lp, \, L_Q, \, d_H)$}\

It is a well known fact in Lie theory that
the universal enveloping algebra of the free Lie algebra generated by a vector space $V$
is the free associative algebra generated by $V$.
Given a dg manifold $(\cM, \, Q)$,
in this part we describe a similar result
for the categories $Ch(\Mod)$ and $D(\Mod)$.
We are mainly inspired by Ramadoss's work \cite{Ramadoss}.

\begin{definition}\

If it exists, the universal enveloping algebra of a Lie algebra object $G$ in
$Ch(\Mod)$ (resp. $D(\Mod)$)
is an associative algebra object $U(G)$ in
$Ch(\Mod)$ (resp. $D(\Mod)$) together with a morphism of Lie algebra objects $i:~~G \rightarrow U(G)$ satisfying the following universal property:~~ given any associative algebra object $K$ and any morphism
of Lie algebra objects $\varphi:~~ G\rightarrow K$ in $Ch(\Mod)$ (resp. $D(\Mod)$), there exists a unique morphism  of associative algebra objects $\varphi':~~ U(G) \rightarrow K$
in $Ch(\Mod)$ (resp. $D(\Mod)$) such that $\varphi=\varphi'\circ i$.
\end{definition}
It can be easily verified that, if it exists, $U(G)$ is unique up to isomorphism in
$Ch( (\smf_{\cM}, \, Q)$
$\mathrm{-}\mathbf{mod})$
(resp. $D(\Mod)$).
	
The shifted free Lie algebra functor
$\functorL(\underline{\, \cdot \,})$
from
the category of $\smf_{\cM}$-modules
to the category of
$\smf_{\cM}$-complexes
(see Section \ref{the functor L and the PBW map})
can be naturally upgraded to a functor
from the category $\Mod$
of $(\smf_{\cM}, \,Q)$-modules
to the category $Ch(\Mod)$
of $(\smf_{\cM}, \, Q)$-complexes.
Applying this functor to a $(\smf_{\cM}, \,Q)$-module
$(V, \, \, L_Q)$,
we obtain the object
$\functorL(\underline{V}, \, L_Q)$
$:=(\functorL(\underline{V}), \, L_Q, \, 0)$
which is a $(\smf_{\cM}, \, Q)$-complex with trivial horizontal differential.
Here the vertical differential
$L_Q:~ \functorL(\underline{V}) \rightarrow \functorL(\underline{V})$
is obtained by extending the original $L_Q:~ V \rightarrow V$ in a unique manner such that
\begin{equation*}
L_Q(\textup{\textlbrackdbl}D, \, E\textup{\textrbrackdbl})
=
\textup{\textlbrackdbl}
L_QD, \, E
\textup{\textrbrackdbl}
+(-1)^{|D|}
\textup{\textlbrackdbl}
D, \, L_QE
\textup{\textrbrackdbl}
\, ,    	
\end{equation*}
for homogeneous $D, \, E \in \functorL(\underline{V})$.
The pair
$(\functorL(\underline{V}, \, L_Q);\, \nbk)$ is
a Lie algebra object in
$Ch( (\smf_{\cM}, \, Q)$
$\mathrm{-}\mathbf{mod})$.
Similarly,
we obtain
the tensor algebra
$(T^{\bullet}\underline{V}, \, \, L_Q, \, 0)$
which is an associative algebra object in $Ch(\Mod)$.
Moreover,
the universal enveloping algebra
$U(\functorL(\underline{V}, \, L_Q))$
of the Lie algebra object
$(\functorL(\underline{V}, \, L_Q);\, \nbk)\in Ch(\Mod)$
can be identified with
$(T^{\bullet}\underline{V}, \, L_Q, \, 0)$.
Note that both
$(\functorL(\underline{V}), \, L_Q, \, 0)$ and $(T^{\bullet}\underline{V}, \, L_Q, \, 0)$ have trivial horizontal differentials.
We then consider the general situation that the horizontal differentials in $\functorL(\underline{V})$
and $T^{\bullet}\underline{V}$ are non-trivial.
\begin{theorem}\label{UL}
Suppose that endowing 
$(T^{\bullet}\underline{V}, \, L_Q)$ with the horizontal differential $\delta$ makes it a $(\smf_{\cM}, \, Q)$-complex.
If
the horizontal differential $\delta$ satisfies
the following conditions:
\begin{itemize}
\item[(1)]
$\delta(D\widetilde{\otimes}E)=
\delta(D)\widetilde{\otimes}E+
(-1)^{|D|}D\widetilde{\otimes}\delta(E),
\quad \forall D, E \in T^{\bullet}\underline{V} \, ;$
\item[(2)]	
$\delta(\underline{V})
\, \subset
\functorL(\underline{V})\cap T^2 \underline{V}  \, ,$
\end{itemize}
then
\begin{itemize}
\item[(1)]
the horizontal differential $\delta$
in $T^{\bullet}\underline{V}$
 preserves
$\functorL(\underline{V})$,
$(\functorL(\underline{V}), \, L_Q,  \,  \delta)$
is a $(\smf_{\cM}, \, Q)$-subcomplex of
$(T^{\bullet}\underline{V}, \, L_Q,  \,  \delta)$  and 
$(\functorL(\underline{V}), \, L_Q,  \,  \delta; \, \nbk)$ is a Lie algebra object in $Ch(\Mod)$;

\item[(2)]
the universal enveloping algebra of $(\functorL(\underline{V}), \, L_Q,  \,  \delta; \, \nbk)$ is $(T^{\bullet}\underline{V}, \, L_Q, \, \delta)$ in the category $Ch(\Mod)$;

\item[(3)]
the universal enveloping algebra of $(\functorL(\underline{V}), \, L_Q,  \,  \delta; \, \nbk)$ is $(T^{\bullet}\underline{V}, \, L_Q, \, \delta)$ in the derived category $D(\Mod)$.
\end{itemize}
\end{theorem}

The preceding theorem is analogous to a result of Ramadoss \cite{Ramadoss}.
Its proof is a straightforward adaptation of Ramadoss's argument and is thus omitted.

\vskip 0.5cm

Now we apply Theorem \ref{UL} to the special
case
$(V, \, L_Q)=(\cD_{\cM}, \, L_Q)$
where $(\cM, \, Q)$ is the given dg manifold.
In this situation, we have
$\underline{V}=\Dp$,
$T^{\bullet}(\underline{V})=\Dpp$, and
$(\functorL (\underline{V}), \, L_Q, \, 0)=(\functorL (\Dp), \, L_Q, \, 0)$.
Here $(\functorL (\Dp), \, L_Q, \, 0)$
is the free Lie algebra object generated over
$\smf_{\cM}$ by $\Dp$ in the category
$Ch(\Mod)$.

Recall that we have the Hochschild differential
$d_H=[\mult \, , \, \cdot\,]\,:~ \DD_{\mathrm{poly}}^{\bullet} \rightarrow \DD_{\mathrm{poly}}^{\bullet+1}$.
We need the following fact:
\begin{lemma}
We have
$d_H(\DD_{\mathrm{poly}}^1)\subset \Lp \cap \DD_{\mathrm{poly}}^2$.
\end{lemma}
\begin{proof}
For a homogeneous element $D \in \Dp$,
by Equation  \eqref{coproduct} we have
\begin{equation*}
\begin{split}
d_H(D)&
=(-1)^{|D|}(D\widetilde{\otimes} \id_{\smf_{\cM}} -\Delta(D) -(-1)^{|D|}\id_{\smf_{\cM}}\widetilde{\otimes} D) \\
& \quad =(-1)^{|D|}(
\textup{\textlbrackdbl}
D, \, \id_{\smf_{\cM}}
\textup{\textrbrackdbl}
-\Delta(D)
) \, .
\end{split}	
\end{equation*}
The term $\textup{\textlbrackdbl}
D, \, \id_{\smf_{\cM}}
\textup{\textrbrackdbl}$ belongs to $\functorL(\Dp)$.
By Proposition \ref{Deltaformula} in the Appendix,
we also see that the term $\Delta(D)$ belongs to $\functorL(\Dp)$.
Thus the lemma is proved. 	
\end{proof}

A direct consequence of Theorem \ref{UL} and the above lemma is the following
\begin{corollary}\label{UL2}\ With the same notations as above, we have
\begin{itemize}
\item[(1)]
the quadruple
$(\Lp, \, L_Q, \, \delta=d_H; \, \nbk)$ is a Lie algebra object in $Ch(\Mod)$;

\item[(2)]
the universal enveloping algebra
of $(\Lp, \, L_Q, \, d_H; \, \nbk)$
is $(\Dpp, \, L_Q, \, d_H)$
in the category
$Ch(\Mod)$;

\item[(3)]
the universal enveloping algebra
of $(\Lp, \, L_Q, \, d_H; \, \nbk)$
is $(\Dpp, \, L_Q, \, d_H)$
in the derived category
$D(\Mod)$.
\end{itemize}
\end{corollary}

Consider the Poincar\'e-Birkhoff-Witt map
$\pbw:~S^{\bullet}( \Lp)\rightarrow U(\Lp) = \Dpp$.
By definition
(see Equation  \eqref{Eqt:pbwstandard}),
the map $\pbw$ commutes with $L_Q$ and $d_H$.
Thus it is an isomorphism of
$(\smf_{\cM}, \, Q)$-complexes:
\[
\pbw:~~ S^{\bullet}(\Lp, \, L_Q, \, d_H)
\rightarrow U((\Lp, \, L_Q, \, d_H))=( \Dpp, \, L_Q, \, d_H ).
\]
Denote by
$$\theta:~~(\Tp, \, L_Q , \,0) \rightarrow (\Lp, \, L_Q, \, d_H )$$
the natural inclusion map.

\begin{proposition}\label{betaquasi-isomorphism}
The map $\theta $
is a quasi-isomorphism in $Ch(\Mod)$.
\end{proposition}
\begin{proof}

Recall that we have the HKR quasi-isomorphism
$\hkr: (\Tpp, \, L_Q, \,0) \rightarrow (\Dpp, \, L_Q, \, d_H)$ (see Theorem \ref{LSX}).
We can decompose it as
\[
\hkr=\pbw \circ S^{\bullet}(\theta),
\]
where $S^{\bullet}(\theta)$ is the symmetrization of
$\theta$ in the   category
$Ch(\Mod) $:~~
\[
S^{\bullet}(\theta):~~
(\Tpp, \, L_Q, \, 0)=S^{\bullet}(\Tp, \, L_Q,\, 0)
\rightarrow
S^{\bullet}(\Lp, \, L_Q , \, d_H).
\]

As the map $\pbw$ is an isomorphism and the map $\hkr$ is a quasi-isomorphism of
$(\smf_{\cM}, \, Q)$-complexes,
the map $S^{\bullet}(\theta)$ is a quasi-isomorphism in 	 $Ch(\Mod)$ as well.
In particular,
  $\theta$,   the first component of $S^{\bullet}(\theta)$, is a quasi-isomorphism of
$(\smf_{\cM}, \, Q)$-complexes.
\end{proof}

The following theorem is an analogue of \cite[Theorem $1$]{Ramadoss}. The proof is essentially the same and thus omitted.
\begin{theorem}\label{R1}
The following diagram commutes in the category $Ch(\Mod)$:
\begin{equation*}\label{R1diagram}
\xymatrix{
S^{\bullet}(\Lp, \, L_Q, \, d_H) \widetilde{\otimes} (\Lp, \, L_Q, \, d_H)  \ar[r]^<<<<<{\mu \circ \frac{\omega}{1-e^{- \omega}}} \ar[d]_{\pbw \otimes \id} &
S^{\bullet}(\Lp, \, L_Q, \, d_H)\ar[d]^{\pbw}\\
(\Dpp, \, L_Q, \, d_H)
\widetilde{\otimes}
(\Lp, \, L_Q, \, d_H)\ar[r]^<<<<<<<<<<{\cup} & (\Dpp, \, L_Q, \, d_H) \, . \\
}
\end{equation*}
Here $\mu$ denotes the natural symmetric product of $S^\bullet(\Lp, \, L_Q, \, d_H)$ and the map
\[
\omega:~~ S^{\bullet}(\Lp, \, L_Q, \, d_H) \widetilde{\otimes} (\Lp, \, L_Q, \, d_H) \rightarrow
S^{\bullet}(\Lp, \, L_Q, \, d_H) \widetilde{\otimes} (\Lp, \, L_Q, \, d_H)
\]
is defined by
\begin{equation*}\label{omegaformula}
\omega( (G_1 \widetilde{\odot} \cdots \widetilde{\odot} G_n) \widetilde{\otimes} G):~~=
\sum_{i=1}^n (-1)^{(|G_{i+1}|+ \cdots +|G_n|)|G_i|}(G_1 \widetilde{\odot} \cdots \hat{G}_i \cdots \widetilde{\odot} G_n) \widetilde{\otimes}
\textup{\textlbrackdbl}
G_i, \, G
\textup{\textrbrackdbl} \, ,
\end{equation*}
for homogeneous elements $G_i, \, G \in \Lp$.
\end{theorem}

\section{Atiyah classes of dg manifolds}\label{Section4}
\subsection{The jet sequence}\

Given a graded manifold $\cM$, consider the space
$\cD_{\cM}^{\leqslant 1}$ of order $\leqslant 1$ differential operators on $\cM$.
Namely, $\cD_{\cM}^{\leqslant 1}$ consists of those operators on the algebra $\smf_{\cM}$ that are
the sum of a derivation and the multiplication by an element of $\smf_{\cM}$.
Indeed, $\cD_{\cM}^{\leqslant 1}$ can be identified to
$\XX(\cM)\oplus \smf_{\cM}$ in a canonical way.
Moreover, $\cD_{\cM}^{\leqslant 1}$
is a
$\smf_{\cM}$-bimodule whose bimodule structure is as follows:
\begin{equation*}
\begin{split}
f \cdot ( X +g)&=fX+fg \, , \\
(X +g)\cdot f &=X(f \cdot \,)  + gf =(-1)^{\widetilde{f}\widetilde{X}}fX +(X(f)+gf )  \,  ,
\end{split}	
\end{equation*}
where $X+g \in \cD_{\cM}^{\leqslant 1}=\XX(\cM)\oplus \smf_{\cM}$ and $f\in \smf_{\cM}$.

Let $\cMmodule$ be a $\smf_{\cM}$-module.
Since $\cD_{\cM}^{\leqslant 1}$ is a $\smf_{\cM}$-bimodule, we can form the space
$\cD_{\cM}^{\leqslant 1} \underset{\smf_{\cM}}{\otimes} \cMmodule$
which is a (left) $\smf_{\cM}$-module.
So we get an exact sequence of
$\smf_{\cM}$-modules:~~
\begin{equation}\label{Short}
0\rightarrow \cMmodule
\stackrel{i}{\longrightarrow}
\cD_{\cM}^{\leqslant 1} \underset{\smf_{\cM}}{\otimes} \cMmodule
\stackrel{j}{\longrightarrow}
\XX(\cM)\otimes_{\smf_{\cM}}\cMmodule
\rightarrow 0 \, .
\end{equation}
Here the maps $i$ and $j$ are defined by:~~
\begin{equation*}
\begin{split}
i(\xi) & =1 \underset{\smf_{\cM}}{\otimes}\xi
\, , \\
j(X\underset{\smf_{\cM}}{\otimes}\xi) & =X\otimes_{\smf_{\cM}}\xi \, , \\
j(f\underset{\smf_{\cM}}{\otimes}\xi) &=0 \, ,
\end{split}
\end{equation*}
for all $\xi \in \cMmodule$
, $X\in \XX(\cM)$, and $f\in \smf_{\cM}$.
We call Equation  \eqref{Short} the (first-)\textbf{jet sequence} associated with the $\smf_{\cM}$-module
$\cMmodule$.

Now suppose that $(\cM,\, Q)$ is a dg manifold 
and $(\cMmodule, \, L_Q)$ is a
$(\smf_{\cM}, \, Q)$-module.
Then the space
$\cD_{\cM}^{\leqslant 1} \underset{\smf_{\cM}}{\otimes} \cMmodule$
can be endowed with a $(\smf_{\cM}, \, Q)$-module structure by setting
\[
L_Q(D\underset{\smf_{\cM}}{\otimes} \xi):~~=[Q,\, D]\underset{\smf_{\cM}}{\otimes} \xi +(-1)^{\widetilde{D}}D\underset{\smf_{\cM}}{\otimes} L_Q(\xi) \, ,
\]
where $D\in \cD_{\cM}^{\leqslant 1}$ and $\xi \in \cMmodule$ are homogeneous elements.
Moreover, the jet sequence \eqref{Short} becomes an exact sequence in the category $\Mod$ of
$(\smf_{\cM}, \, Q)$-modules.

\subsection{Smooth connections and Atiyah classes of dg modules}
\begin{definition}\label{smooth connection}
A smooth connection $\nabla$ on
a $\smf_{\cM}$-module $\cMmodule$ is a degree $0$ map
\[
\nabla:~~ \XX(\cM)
\otimes_{\fK} \cMmodule
\rightarrow \cMmodule \, ,
\]
\[
(X, \, \xi) \mapsto
\nabla_{X}\xi
\]
satisfying
\[
\nabla_{fX}\xi=f\nabla_{X}\xi \, ,
\]
\[
\nabla_{X}(f\xi)=X(f)\xi+(-1)^{\widetilde{f}\widetilde{X}}f\nabla_{X}\xi \, ,
\]
for all homogeneous elements
$f\in \smf_{\cM}$,
$X \in \XX(\cM)$, and
$\xi\in \cMmodule$.
\end{definition}

\begin{lemma}
The set of smooth connections on $\cMmodule$ is in one-to-one correspondence with the set of $\smf_{\cM}$-linear splittings of the jet sequence \eqref{Short}.
\end{lemma}
\begin{proof}
Given a smooth connection
$\nabla:~ \XX(\cM)
\otimes_{\fK} \cMmodule
\rightarrow \cMmodule$
defined as earlier by Definition \ref{smooth connection},
we construct a map
\begin{equation*}
s:~ \XX(\cM)\otimes_{\smf_{\cM}}\cMmodule
\rightarrow
\cD_{\cM}^{\leqslant 1} \underset{\smf_{\cM}}{\otimes} \cMmodule
=
(\XX(\cM)\underset{\smf_{\cM}}{\otimes}\cMmodule)
\oplus \cMmodule
\end{equation*}
by setting
\begin{equation}\label{Eqt:nableetc}
s(X\otimes_{\smf_{\cM}}\xi)
=
X \underset{\smf_{\cM}}{\otimes}\xi \, -
\nabla_{X}\xi  \, 
\end{equation}
for all $X\in \XX(\cM)$ and $\xi\in \cMmodule$. 
It is easy to verify that $s$ is $\smf_{\cM}$-bilinear and thus it is well defined.
Clearly, $s$ is a splitting of the jet sequence \eqref{Short}.

Conversely, given a splitting
$s:~ \XX(\cM)\otimes_{\smf_{\cM}}\cMmodule
\rightarrow
\cD_{\cM}^{\leqslant 1} \underset{\smf_{\cM}}{\otimes} \cMmodule $
of the jet sequence \eqref{Short},
one can define
a smooth connection $\nabla$ by Equation \eqref{Eqt:nableetc}. 
 
\end{proof}

Now suppose that $\nabla$ is a smooth connection on
the $\smf_{\cM}$-module $\cMmodule$.
If $\cMmodule$ is endowed with a differential $L_Q$ so that
$(\cMmodule, \, L_Q)$ is a
$(\smf_{\cM}, \, Q)$-module,
then one can define a map
\[
\at_{\cMmodule}^{\nabla}:~
\XX(\cM)\otimes_{\fK}\cMmodule \rightarrow \cMmodule
\, ,
\]
by setting
\[
\at_{\cMmodule}^{\nabla}(X, \, \xi):~=[L_Q,\nabla](X,\xi) =L_Q(\nabla_X \xi)-\nabla_{[Q,X]}\xi-(-1)^{\widetilde{X}}\nabla_{X}L_Q(\xi) \, ,
\]
for homogeneous elements
$X \in \XX(\cM)$ and $\xi\in \cMmodule$.
It is easy to verify that
$\at_{\cMmodule}^{\nabla}$ is
a degree $(+1)$ and $\smf_{\cM}$-bilinear map.
Moreover, it is a degree $0$ cocycle in the cochain complex
$(\Hom_{\smf_{\cM}}^{\bullet}(\XX(\cM)\otimes_{\smf_{\cM}}\cMmodule, \cMmodule[1])	 , \, \, L_Q)$, i.e.,
$L_Q(\at_{\cMmodule}^{\nabla})=0$, or
\[
\at_{\cMmodule}^{\nabla} \in
\Hom_{\Mod}( (\XX(\cM), \, L_Q)\otimes_{\smf_{\cM}}
(\cMmodule, \, L_Q),
(\cMmodule[1], \, L_Q) )
\, .
\]
Instead, we will take the following convention in this paper:
\[
\at_{\cMmodule}^{\nabla} \in
\Hom_{\Mod}
(
(\XX(\cM)[-1], \, L_Q)
\otimes_{\smf_{\cM}}
(\cMmodule, \, L_Q),
(\cMmodule, \, L_Q)
)
\, .
\]
Note that a sign correction is necessary:
\begin{equation}\label{Atiyahcocycleformula}
\at_{\cMmodule}^{\nabla}(\overline{X}, \xi)=
(-1)^{\widetilde{X}}
(
L_Q(\nabla_X \xi)-\nabla_{[Q,X]}\xi-(-1)^{\widetilde{X}}\nabla_{X}L_Q(\xi) )
\, .
\end{equation}
Here the homogeneous element $\overline{X} \in \XX(\cM)[-1]$ corresponds to the homogeneous element $X \in \XX(\cM)$.

\begin{definition}
We call $\at^{\nabla}_{\cMmodule}$
defined in Equation  \eqref{Atiyahcocycleformula} the \textbf{Atiyah cocycle} of
the $(\smf_{\cM}, \, Q)$-module $(\cMmodule, \, L_Q)$
with respect to the smooth connection $\nabla$.
We call
\begin{equation*}
\begin{split}
\at_{\cMmodule}:~=[\at^{\nabla}_{\cMmodule}]
& \in
H^0(
\Hom^{\bullet}_{\smf_{\cM}}(\XX(\cM)[-1]\otimes_{\smf_{\cM}}\cMmodule, \cMmodule), \, L_Q
)
\\
& \quad =
\Hom_{\mathrm{H}(\Mod)}
(
(\XX(\cM)[-1], \, L_Q)
\otimes_{\smf_{\cM}}
(\cMmodule, \, L_Q),
(\cMmodule, \, L_Q)
)
\end{split}
\end{equation*}
the \textbf{Atiyah class} of
the
$(\smf_{\cM}, \, Q)$-module $(\cMmodule, \, L_Q)$.
\end{definition}
Here the Atiyah class $\at_{\cMmodule}$
is regarded as a morphism
\[
\at_{\cMmodule}: (\XX(\cM)[-1], \, L_Q)\otimes_{\smf_{\cM}}(\cMmodule, \, L_Q) \rightarrow (\cMmodule, \, L_Q)
\]
in the homology category $\mathrm{H}(\Mod)$.
By the natural functor
$\mathrm{H}(\Mod) \rightarrow \Pi(\Mod)$,
the above
$\at_{\cMmodule}$ can also be treated as a morphism in the homotopy category $\Pi(\Mod)$.

\begin{remark}
In general, the existence of a smooth connection on
the $\smf_{\cM}$-module $\cMmodule$
is not guaranteed.
So one can not use the definition of Atiyah class as above.
Instead,  we define the Atiyah class
$\at_{\cMmodule}$ of
the $(\smf_{\cM}, \, Q)$-module $(\cMmodule, \, L_Q)$ to be the element
\[
\at_{\cMmodule}\in
\Ext ^1_{\Mod}(
(\XX(\cM), \, L_Q)
\otimes_{\smf_{\cM}}
(\cMmodule, \, L_Q),
(\cMmodule, \, L_Q)
)
\]
corresponding  to the jet sequence \eqref{Short} in the category of $(\smf_{\cM}, \, Q)$-modules.
We have the standard isomorphism:
\begin{equation*}
\begin{split}
& \quad \Ext ^1_{\Mod}(
(\XX(\cM), \, L_Q)
\otimes_{\smf_{\cM}}
(\cMmodule, \, L_Q),
(\cMmodule, \, L_Q)
) \\
& \quad\quad\quad\quad
\simeq
\Hom_{\Pi(\Mod)}(
(\XX(\cM)[-1], \, L_Q)
\otimes_{\smf_{\cM}}
(\cMmodule, \, L_Q),
(\cMmodule, \, L_Q)
)
\, .
\end{split}
\end{equation*}
Thus $\at_{\cMmodule}$ can be considered as
a morphism from
$(\XX(\cM)[-1], \, L_Q)
\otimes_{\smf_{\cM}}
(\cMmodule, \, L_Q)$
to
$(\cMmodule, \, L_Q)$
in the homotopy category $\Pi(\Mod)$,
and it is represented by the following roof
in the category $\Mod$:
\begin{equation*}\label{roof}
\begin{xy}
(0,0)*+{(\XX(\cM)[-1], \, L_Q)\otimes_{\smf_{\cM}}(\cMmodule, \, L_Q)}="a";
(60,0)*+{(\cMmodule, \, L_Q) \,\,.}="b";
(30,20)*+{(\mathrm{Cone}(i), \, L_Q)}="c";
{\ar_{q} "c";"a"};
{\ar^{p} "c";"b"};
{\ar@{-->}^-{\at_{\cMmodule}} "a";"b"};
\end{xy}
\end{equation*}
Here
the $(\smf_{\cM}, \, Q)$-module
\[
(\mathrm{Cone}(i), \, L_Q):~=
(
(\cD_{\cM}^{\leqslant 1}[-1]\underset{\smf_{\cM}}{\otimes} {\cMmodule})
\oplus \cMmodule, \, L_Q
)
\]
is indeed the mapping cone of $i$
with the natural induced differential $L_Q$.
The left roof pane $q:=j\circ \mathrm{pr}_1$ is a quasi-isomorphism
in $\Mod$ and the right roof pane is $p:=\mathrm{pr}_2$, where
$\mathrm{pr}_1$ and $\mathrm{pr}_2$ are projections to respectively the first and second direct summands of $\mathrm{Cone}(i)$.
\end{remark}

The following lemma is easily verified, and is left to the reader.
 \begin{proposition}
Suppose that the jet sequence \eqref{Short}
splits in the category of $\smf_{\cM}$-modules. The following statements are equivalent.
\begin{itemize}
\item[(1)]
The jet sequence \eqref{Short} splits in the category $\Mod$ of $(\smf_{\cM}, \, Q)$-modules.
 \item[(2)]
There exists a smooth connection $\nabla$ on
$\cMmodule$ which is compatible with
the $(\smf_{\cM}, \, Q)$-module structures
of $(\XX(\cM), \, L_Q)$ and
$(\cMmodule, \, L_Q)$, i.e.,
the equality $L_Q \nabla = \nabla L_Q$ holds.

\item[(3)]
The Atiyah class $\at_{\cMmodule}$ vanishes.
\end{itemize}
\end{proposition}
\begin{proposition}\label{Atiyahcocyleoftensorproduct}
Let $\nabla^1$ and $\nabla^2$ be, respectively, smooth connections on $\smf_{\cM}$-modules
$\cMmodule_1$ and $\cMmodule_2$.
Then
\[
\nabla_X (\xi_1 \otimes_{\smf_{\cM}} \xi_2):~~ =(\nabla^1_X \xi_1)\otimes_{\smf_{\cM}} \xi_2 +(-1)^{\widetilde{X}\widetilde{\xi}_1} \xi_1 \otimes_{\smf_{\cM}} \nabla^2_X \xi_2 \, ,\quad \forall \xi_1 \in \cMmodule_1, \, \xi_2 \in \cMmodule_2 \, ,
\]
defines a smooth connection $\nabla$ on
$\cMmodule_1 \otimes_{\smf_{\cM}} \cMmodule_2$.
Moreover, if $(\cMmodule_1, \, L_Q)$ and
$(\cMmodule_2, \, L_Q)$ are
$(\smf_{\cM}, \, Q)$-modules,
then the Atiyah cocycle
\[
\at^{\nabla}_{\cMmodule_1 \otimes_{\smf_{\cM}} \cMmodule_2 }:~~ \XX(\cM)[-1] \otimes_{\smf_{\cM}} (\cMmodule_1 \otimes_{\smf_{\cM}} \cMmodule_2) \rightarrow
\cMmodule_1 \otimes_{\smf_{\cM}} \cMmodule_2
\]
is given by
\begin{equation*}\label{Atiyahcocyleoftensorproductformula}
\at^{\nabla}_{\cMmodule_1 \otimes_{\smf_{\cM}} \cMmodule_2 }(\overline{X}, \xi_1 \otimes_{\smf_{\cM}} \xi_2)=
\at^{\nabla^1}_{\cMmodule_1}(\overline{X} , \xi_1)\otimes_{\smf_{\cM}} \xi_2 + (-1)^{(\widetilde{X}+1)\widetilde{\xi}_1} \xi_1 \otimes_{\smf_{\cM}}
\at^{\nabla^2}_{\cMmodule_2}(\overline{X} , \xi_2) \, ,
\end{equation*}
for all homogeneous elements
$X \in \XX(\cM)$, $\xi_1 \in \cMmodule_1$, and $\xi_2 \in \cMmodule_2$.
\end{proposition}
\begin{proof}
It is easy to check that $\nabla$ is a smooth connection on
$\cMmodule_1\otimes_{\smf_{\cM}} \cMmodule_2$.
By Equation  \eqref{Atiyahcocycleformula}, we have:
\begin{equation*}
\begin{split}
\at_{\cMmodule_1 \otimes_{\smf_{\cM}} \cMmodule_2}^{\nabla}(\overline{X},
\xi_1 \otimes_{\smf_{\cM}} \xi_2)
& =
(-1)^{\widetilde{X}}
(
L_Q(
\nabla_X (\xi_1 \otimes_{\smf_{\cM}} \xi_2)
)-
\nabla_{[Q,X]}
(\xi_1 \otimes_{\smf_{\cM}} \xi_2) \\
& \quad -(-1)^{\widetilde{X}}\nabla_{X}L_Q(\xi_1 \otimes_{\smf_{\cM}} \xi_2) )  \\
& =
(-1)^{\widetilde{X}}
(
L_Q(\nabla_X^{1} \xi_1)-\nabla_{[Q,X]}^{1}\xi_1-(-1)^{\widetilde{X}}\nabla_{X}^{1}L_Q\xi_1)\otimes_{\smf_{\cM}}\xi_2 \\
& \quad +
(-1)^{(\widetilde{X}+1)\widetilde{\xi}_1 + \widetilde{X} }
\xi_1 \otimes_{\smf_{\cM}}
(L_Q(\nabla_X^{2} \xi_2)-\nabla_{[Q,X]}^{2}\xi_2-(-1)^{\widetilde{X}}\nabla_{X}^{2}L_Q\xi_2) \\
& =
\at^{\nabla^1}_{\cMmodule_1}(\overline{X} , \xi_1)\otimes_{\smf_{\cM}} \xi_2 + (-1)^{(\widetilde{X}+1)\widetilde{\xi}_1} \xi_1 \otimes_{\smf_{\cM}}
\at^{\nabla^2}_{\cMmodule_2}(\overline{X} , \xi_2)
\, .
\end{split}
\end{equation*}
\end{proof}

We also have the following functorial property of Atiyah classes.
\begin{proposition}\label{functoriality}
If both $(\cMmodule_1, \, L_Q)$ and $(\cMmodule_2, \, L_Q)$  admit smooth connections
and $\varphi:~~ (\cMmodule_1, \, L_Q)\rightarrow (\cMmodule_2, \, L_Q)$ is a morphism of
$(\smf_{\cM}, \, Q)$-modules,
then the following diagram
\[
\xymatrix{
(\XX(\cM)[-1], \, L_Q) {\otimes_{\smf_{\cM}}} (\cMmodule_1, \, L_Q)
\ar[r]^{\id \otimes \varphi} \ar[d]_{\at_{\cMmodule_1}} &
(\XX(\cM)[-1], \, L_Q){\otimes_{\smf_{\cM}}} (\cMmodule_2, \, L_Q)
\ar[d]^{\at_{\cMmodule_2}}\\
(\cMmodule_1, \, L_Q) \ar[r]^{\varphi} &
(\cMmodule_2, \, L_Q) ~~ \\
}
\]
commutes
in the homology category $\mathrm{H}(\Mod)$.
\end{proposition}
\begin{proof}
Let $\nabla_1$ and $\nabla_2$ be smooth connections
on $(\cMmodule_1, \, L_Q)$ and $(\cMmodule_2, \, L_Q)$
respectively.
We only need to show that the following diagram
\[
\xymatrix{
(\XX(\cM)[-1], \, L_Q) {\otimes_{\smf_{\cM}}} (\cMmodule_1, \, L_Q)
\ar[r]^{\id \otimes \varphi} \ar[d]_{\at_{\cMmodule_1}^{\nabla_1}} &
(\XX(\cM)[-1], \, L_Q){\otimes_{\smf_{\cM}}} (\cMmodule_2, \, L_Q)
\ar[d]^{\at_{\cMmodule_2}^{\nabla_2}}\\
(\cMmodule_1, \, L_Q) \ar[r]^{\varphi} &
(\cMmodule_2, \, L_Q) ~~ \\
}
\]
commutes up to a cochain homotopy
between $(\smf_{\cM}, \, Q)$-modules  in the category $\Mod$. In fact, we have the equality
\begin{equation*}
\varphi\circ \at_{\cMmodule_1}^{\nabla_1}
- \at_{\cMmodule_2}^{\nabla_2}\circ (\id_{\XX(\cM)[-1]}\otimes \varphi)
=
L_Q(\varphi \circ \nabla_1 - \nabla_2 \circ (\id_{\XX(\cM)[-1]} \otimes \varphi )) \, ,	
\end{equation*}
which can be verified directly.
\end{proof}

\subsection{Atiyah classes of dg complexes}\

Now we consider a $(\smf_{\cM}, \, Q)$-complex
$(\Upsilon, \, L_Q , \, \delta)$.
Recall that the corresponding total space
 $(\tot \Upsilon, \, L_Q+\delta)$ is a
 $(\smf_{\cM}, \, Q)$-module.

\begin{definition}
Let $(\Upsilon, \, L_Q, \, \delta)$ be a
$(\smf_{\cM}, \, Q)$-complex such that
$(\tot \Upsilon, \, L_Q+\delta)$ admits a smooth connection.
We define the \textbf{Atiyah class} $\at_{\Upsilon}$ of
$(\Upsilon, \, L_Q, \, \delta)$ to be the
Atiyah class of the $(\smf_{\cM}, \, Q)$-module
$(\tot\Upsilon, \, L_Q+\delta)$, i.e.,
\begin{equation*}
\begin{split}
\at_{\Upsilon}:~~ =\at_{\tot \Upsilon}
& \in
H^0(\Hom^{\bullet}_{\smf_{\cM}}(
\XX(\cM)[-1]\otimes_{\smf_{\cM}}\tot \Upsilon, \tot \Upsilon), \, L_Q+\delta) \\
& \quad \quad \simeq
\Hom_{\mathrm{H}(\Mod)}(
(\XX(\cM)[-1], \, L_Q)
\otimes_{\smf_{\cM}}
(\tot \Upsilon, \, L_Q+\delta),
(\tot \Upsilon, \, L_Q+\delta)
) \, .
\end{split}
\end{equation*}
\end{definition}

\begin{remark}\label{remark on connections on dg complexes}
Suppose that each component $\Upsilon^{n,\bullet}$ of
$(\Upsilon, \, L_Q, \, \delta)$ admits a smooth connection $\nabla_n$, then $\nabla^{\tot}=\oplus_n \nabla_n$ defines a smooth connection on $\tot \Upsilon$.
In this case, we have
\begin{equation*}
\begin{split}
\at^{\nabla^{\tot}}_{\tot\Upsilon}
&=[L_Q + \delta, \nabla^{\tot}]\\
& \quad\quad ={}
(L_Q\circ \nabla^{\tot} - \nabla^{\tot}\circ L_Q)+
(\delta\circ \nabla^{\tot} - \nabla^{\tot}\circ (\id  \otimes\delta) ) \\
& \quad\quad\quad ={}
\sum_n \at_{\Upsilon^{n,\bullet}}^{\nabla_n} +
\sum_n
(\delta\circ \nabla_n - \nabla_{n+1}\circ (\id  \otimes\delta)).
\end{split}
\end{equation*}
We observe that
the first part
$\sum\limits_n
\at_{\Upsilon^{n,\bullet}}^{\nabla_n}$ sends
the component
$(\Tp \widetilde{\otimes} \Upsilon)^{p,q}$ to
$\Upsilon^{p-1,q+1}$, while the second part
$\sum\limits_n (\delta\circ \nabla_n - \nabla_{n+1}\circ (\id  \otimes\delta))$ sends the component
$(\Tp \widetilde{\otimes} \Upsilon)^{p,q}$ to
$\Upsilon^{p,q}$.
For this reason, we can \emph{not} regard
the Atiyah cocycle
$\at^{\nabla^{\tot}}_{\tot\Upsilon}$
as a morphism of $(\smf_{\cM}, \, Q)$-complexes from $\Tp \widetilde{\otimes}
(\Upsilon, \, L_Q, \, \delta)$ to $(\Upsilon, \, L_Q, \, \delta)$.
\end{remark}

\section{Lie structures}\label{Section5}
\subsection{The Atiyah class as a Lie bracket on
$\XX(\cM)[-1]$}\

Let $(\cM, \, Q)$ be a dg manifold.
By results of Mehta, Sti\'enon, and Xu \cite{M-S-X},
the $\fK$-vector space $ \XX(\cM)[-1] $ admits an $L_{\infty}$-algebra structure. In fact, one can choose a torsion free smooth connection $\nabla$ on
$\XX(\cM)$ to construct a sequence of  operations:~~
\begin{equation*}\label{Linfinity}
\begin{split}
R_1=L_Q :~ & \XX({\cM})[-1] \rightarrow \XX({\cM})[-1]\, , \\
R_2=\at_{\XX(\cM)[-1]}^{\nabla}:~ &  \wedge^2(\XX({\cM})[-1]) \rightarrow
\XX({\cM})[-1] \, , \\
& \quad \quad \quad \quad \cdots \\
R_k :~ & \wedge^k(\XX({\cM})[-1]) \rightarrow \XX({\cM})[-1] \, , \cdots .
\end{split}
\end{equation*}
Then the operations $\{R_k\}_{k \geqslant 1 }$ turn
$\XX_{\cM}[-1]$ into an $L_{\infty}$-algebra.
One can find the details of
this construction in \cite{M-S-X}.

We denote by $\at_{\cM}$ the Atiyah class
  of the $(\smf_{\cM}, \, Q)$-module $(\XX(\cM)[-1], \, L_Q)$,
and treat it as a morphism
\[
\at_{\cM}:~~
(\XX({\cM})[-1], \, L_Q)\otimes_{\smf_{\cM}}(\XX({\cM})[-1], \, L_Q)
\rightarrow
(\XX({\cM})[-1], \, L_Q)
\]
in the homology category $\mathrm{H}(\Mod)$.
As $\XX({\cM})[-1]$ is an $L_{\infty}$-algebra whose binary bracket is the
Atiyah cocycle $\at_{\XX(\cM)[-1]}^{\nabla}$, we have
\begin{theorem}\label{Thm:AtiyahdefinesLiebracket}
The Atiyah class $\at_{\cM}$ defines a Lie bracket
on $(\XX({\cM})[-1], \, L_Q)$ in the homology category
$\mathrm{H}(\Mod)$. 	
\end{theorem}
Therefore, $\at_{\cM}$ also defines a Lie bracket on $(\XX({\cM})[-1], \, L_Q)$ in the homotopy category $\Pi(\Mod)$.
The following fact  is analogous to that of Kapranov \cite{Kapranov}, Chen-Sti\'enon-Xu \cite{C-S-X2} and Chen-Liu-Xiang \cite{C-L-X-HHA}.
\begin{theorem}\label{Atiyahclassasmodulestructuremap}
Let $(\cMmodule, \, L_Q)$ be a
$(\smf_{\cM}, \, Q)$-module.
The Atiyah class
\[
\at_{\cMmodule}:~
(\XX(\cM)[-1], \, L_Q)
\otimes_{\smf_{\cM}}
(\cMmodule, \, L_Q)
\rightarrow
(\cMmodule, \, L_Q)
\]
turns $(\cMmodule, \, L_Q)$ into a Lie algebra module object over the aforesaid Lie algebra object
$(\XX(\cM)[-1],$
$L_Q)$ in the homology category
$\mathrm{H}(\Mod)$. 	 	
\end{theorem}
Of course, one can also treat $(\cMmodule, \, L_Q)$ as a Lie algebra module object over $(\XX(\cM)[-1], \, L_Q )$ in the
homotopy category $\Pi(\Mod)$.

\subsection{Canonical Atiyah cocycle of the dg complex $(\Dpp, \, L_Q, \, d_H)$ }\

As a $\smf_{\cM}$-module, $\Dp$ admits the \textbf{canonical connection}
\[
\nabla^{\mathrm{can}}:~~ \XX(\cM)\otimes_{\fK} \Dp \rightarrow \Dp,
\]
\[
\nabla^{\mathrm{can}}_{X}D:~~=X\circ D,
\]
for all $X\in \XX(\cM)$ and $D\in \Dp$, where
$\circ$ is the Gerstenhaber product defined by Equation  \eqref{Gerstenhaberproductformula}.
The connection $\nabla^{\mathrm{can}}$ extends naturally to $\Dpp^n$, i.e.,
\begin{equation*}
\nabla^{\mathrm{can}}_{X}(D_1\widetilde{\otimes} \cdots  \widetilde{\otimes} D_n):~~=
\sum_{i=1}^{n}(-1)^{\sum_{j=1}^{i-1}|D_j|\widetilde{X}}D_1\widetilde{\otimes} \cdots  \widetilde{\otimes}\nabla^{\mathrm{can}}_XD_i \widetilde{\otimes} \cdots  \widetilde{\otimes}D_n \, ,
\end{equation*}
for all homogeneous elements
$X\in \XX(\cM)$ and
$D_i\in \Dp$.
The sum of all these connections
on $\DD_{\mathrm{poly}}^n$ ($n\geqslant 0$) gives a connection on $\tot \Dpp$
(see Remark \ref{remark on connections on dg complexes}),
which is also denoted by
$\nabla^{\mathrm{can}}$.
It is easy to see that the operator
$\nabla^{\mathrm{can}}_X$ preserves
the subspace $\Lp \subset \Dpp$, and thus
$\nabla^{\mathrm{can}}$ is also
a connection on $\tot\Lp$.

Recall the coproduct $\Delta:~~ \Dp \rightarrow \DD_{\mathrm{poly}}^2$ defined in
Equation  \eqref{Delta}. 
\begin{lemma}\label{canonicalconnectionandDelta}
We have $\nabla^{\mathrm{can}}_X \Delta = \Delta \nabla^{\mathrm{can}}_X,$ for all $X \in \XX(\cM)$.
\end{lemma}
\begin{proof}
Suppose that $X\in \XX(\cM)$ is homogeneous.
Take $D=X_1 \circ \cdots \circ X_n \in \Dp$,
where $X_i \in  \Tp$ are homogeneous.
On the one hand,
by Equations   \eqref{Eqt:Deltaexpressioninbrackets}
and \eqref{Gerstenhaberproductandcupproduct}
in the Appendix, we have
\begin{equation}\label{L.H.S}
\begin{split}
\nabla_{X}\Delta(D)=X \circ \Delta(D)
& =\sum (\pm) X \circ \textup{\textlbrackdbl}X_1 \circ X_{\sigma(2)} \circ \cdots \circ X_{\sigma(p)}, \, X_{\sigma(p+1)} \circ \cdots \circ X_{\sigma(n)}
\textup{\textrbrackdbl} \\
& \quad\quad =\sum (\pm) \textup{\textlbrackdbl}X \circ X_1 \circ X_{\sigma(2)} \circ \cdots \circ X_{\sigma(p)}, \, X_{\sigma(p+1)} \circ \cdots \circ X_{\sigma(n)}
\textup{\textrbrackdbl} \\
& \quad\quad\quad\quad + (\pm)\textup{\textlbrackdbl}
X \circ X_{\sigma(p+1)} \circ \cdots \circ X_{\sigma(n)}, \, X_1 \circ X_{\sigma(2)} \circ \cdots \circ X_{\sigma(p)}\textup{\textrbrackdbl}.
\end{split}
\end{equation}
Here we denote by $(\pm)$ the appropriate signs which can be worked out explicitly as explained in the Appendix.

On the other hand,
we have
\begin{equation}\label{R.H.S}
\begin{split}
\Delta(\nabla_{X}D)=\Delta(X\circ D)
& =
\Delta(X\circ X_1 \circ \cdots \circ X_n) \\
& =\sum (\pm) \textup{\textlbrackdbl}X \circ X_{\sigma(1)} \circ X_{\sigma(2)} \circ \cdots \circ X_{\sigma(p)}, \, X_{\sigma(p+1)} \circ \cdots \circ X_{\sigma(n)}
\textup{\textrbrackdbl}
\end{split}	
\end{equation}
by Equation  \eqref{Eqt:Deltaexpressioninbrackets}.
From the right hand sides of Equations   \eqref{L.H.S}
and \eqref{R.H.S},
we see that
$\nabla_{X}\Delta(D)=\Delta(\nabla_{X}D)$.
\end{proof}

Now we consider the Atiyah cocycles of
$(\Dpp, \, L_Q, \, d_H)$ and $(\Lp, \, L_Q, \, d_H)$ associated to the connection
$\nabla^{\mathrm{can}}$:~~
\begin{equation*}
\begin{split}
& \at_{\DD_{\mathrm{poly}}}^{\nabla^{\mathrm{can}}}:~~\XX(\cM)[-1]\otimes_{\smf_{\cM}}\tot \DD_{\mathrm{poly}} \rightarrow \tot \DD_{\mathrm{poly}} \,  \\
\text{and \quad } & \at_{\functorL(\DD_{\mathrm{poly}}^1)}^{\nabla^{\mathrm{can}}}:~~\XX(\cM)[-1]\otimes_{\smf_{\cM}}\tot \functorL(\DD_{\mathrm{poly}}^1) \rightarrow \tot \functorL(\DD_{\mathrm{poly}}^1) \,.
\end{split}
\end{equation*}
Recall the multiplication map
\begin{equation*}
\begin{split}
\cup:~
(\tot\DD_{\mathrm{poly}}, \, L_Q+d_H){\otimes_{\smf_{\cM}}} (\tot\DD_{\mathrm{poly}}, \, L_Q+d_H)
& \rightarrow (\tot\DD_{\mathrm{poly}}, \, L_Q+d_H) \, , \\
\cup(\overline{P_1}\otimes_{\smf_{\cM}}
\overline{P_1})=
\overline{P_1}\cup\overline{P_2}
:~ &=
\overline{P_1\cup P_2}
\, ,
\quad\quad \forall P_1, \, P_2 \in \Dpp, \,	
\end{split}
\end{equation*}
of the Hopf algebra object
$(\tot \Dpp, \, L_Q+d_H)$ in the category $\Mod$ of $(\smf_{\cM}, \, Q)$-modules.
Here $\overline{P_1} \in \tot \Dpp$
(resp. $\overline{P_2} \in \tot \Dpp$)
is the corresponding element of $P_1 \in \Dpp$
(resp. $P_2 \in \Dpp$).

\begin{lemma}\label{inductionlemma}
The Atiyah cocycle $\at_{\DD_{\mathrm{poly}}}^{\nabla^{\mathrm{can}} }$ is a derivation
with respect to the multiplication $\cup$ of the
Hopf algebra object
$(\tot \Dpp, \, L_Q+d_H)$ in the category $\Mod$, i.e.,
\begin{equation*}\label{induction}
\at_{\DD_{\mathrm{poly}}}^{\nabla^{\mathrm{can}} }
(\overline{X}, \overline{P_1}\cup\overline{P_2})
=
\at_{\DD_{\mathrm{poly}}}^{\nabla^{\mathrm{can}} }
(\overline{X}, \overline{P_1})\cup\overline{P_2}
+
(-1)^{(\widetilde{X}+1)|P_1|}
\overline{P_1} \cup
\at_{\DD_{\mathrm{poly}}}^{\nabla^{\mathrm{can}} }
(\overline{X}, \overline{P_2}),
\end{equation*}
for all homogeneous elements
$X\in \XX(\cM)$, $P_1 \, , P_2\in \Dpp$,
and $\overline{X} \in \XX(\cM)[-1]$ the corresponding element of $X$.
\end{lemma}
\begin{proof}
It is easy to verify that the following diagram
\begin{footnotesize}
\begin{equation}	\label{pdt diagram}
\xymatrix{
(\XX(\cM)[-1], \, L_Q)\otimes_{\smf_{\cM}}(\tot\Dpp \otimes_{\smf_{\cM}} \tot\Dpp, \, L_Q+d_H)
\ar[r]^-{\mathrm{\id \otimes \cup}}
\ar[d]_{ \at_{ \tot\Dpp \otimes \tot\Dpp }^{\nabla^{\mathrm{can}}\otimes \id + \id \otimes \nabla^{\mathrm{can}}} } &
(\XX(\cM)[-1], \, L_Q)\otimes_{\smf_{\cM}}
(\tot \Dpp, \, L_Q+d_H)
\ar[d]^{\at_{\tot\Dpp}^{\nabla^{\mathrm{can}} } }\\
(\tot\Dpp \otimes_{\smf_{\cM}} \tot\Dpp, \, L_Q+d_H)
\ar[r]^-{\cup } &
(\tot \Dpp, \, L_Q+d_H)
}
\end{equation}
\end{footnotesize}
commutes in the category $\Mod$.
Therefore, we have
\begin{small}
\begin{equation*}
\begin{split}
\at_{\DD_{\mathrm{poly}}}^{\nabla^{\mathrm{can}} }
(
\overline{X},
\overline{P_1}\cup\overline{P_2}
)
&=
\cup(
\at_{\tot\DD_{\mathrm{poly}}{\otimes_{\smf_{\cM}}} \tot\DD_{\mathrm{poly}}}^{\nabla^{\mathrm{can}}\otimes \id + \id \otimes \nabla^{\mathrm{can}}}(\overline{X}, \overline{P_1}\otimes_{\smf_{\cM}} \overline{P_2}))
\quad\quad
\text{(by commutativity of Diagram \eqref{pdt diagram})}   \\
& \quad \quad =
\cup(
\at_{\tot\DD_{\mathrm{poly}}}^{\nabla^{\mathrm{can}} }
(\overline{X}, \overline{P_1})\otimes_{\smf_{\cM}} \overline{P_2}
+
(-1)^{(\widetilde{X}+1)|P_1|}
\overline{P_1} \otimes_{\smf_{\cM}}
\at_{\tot\DD_{\mathrm{poly}}}^{\nabla^{\mathrm{can}} }
(\overline{X}, \overline{P_2})
) \\
& \qquad \qquad \qquad \qquad \qquad \qquad \qquad \qquad \qquad \qquad \qquad \qquad \quad \quad \quad\quad \,
\text{(by Proposition \ref{Atiyahcocyleoftensorproduct})} \\
& \quad\quad\quad =
\at_{\DD_{\mathrm{poly}}}^{\nabla^{\mathrm{can}} }
(\overline{X}, \overline{P_1})\cup \overline{P_2}
+
(-1)^{(\widetilde{X}+1)|P_1|}
\overline{P_1} \cup
\at_{\DD_{\mathrm{poly}}}^{\nabla^{\mathrm{can}} }
(\overline{X}, \overline{P_2}) \, ,
\end{split}
\end{equation*}
\end{small}
for homogeneous elements $X \in \XX(\cM)$ 
and $P_1, \, P_2\in \Dpp$.
\end{proof}

Recall that
$\Dpp$ is endowed with
the canonical Lie bracket $\nbk$
 which is $\smf_{\cM}$-bilinear (see Equation  \eqref{freeLiebracket}).
It induces a Lie bracket on $\tot\Dpp$, which is also denoted by $\nbk$.
The subspace $\tot\Lp \subset \tot \Dpp$ is a
Lie subalgebra with respect to $\nbk$.
\begin{proposition}\label{bracket}
The Atiyah cocycle
$\at_{\DD_{\mathrm{poly}}}^{\nabla^{\mathrm{can}}}$ is given by
\[
\at_{\DD_{\mathrm{poly}}}^{\nabla^{\mathrm{can}}}
(\overline{X}, \overline{D_1}\cup \cdots \cup \overline{D_n})=
\textup{\textlbrackdbl}
\overline{X},
\overline{D_1}\cup \cdots \cup \overline{D_n}
\textup{\textrbrackdbl} \, ,
\]
for all homogeneous elements
$X \in \XX(\cM)$, $D_i\in \Dp$,
$\overline{X} \in \XX(\cM)[-1]$ the corresponding element of $X$, and
$\overline{D_i} \in \tot\Dpp$ the corresponding element of $D_i$.
\end{proposition}
\begin{proof}
By Lemma \ref{inductionlemma},
$\at_{\DD_{\mathrm{poly}}}^{\nabla^{\mathrm{can}}}(\overline{X}, \,\cdot \,)$ is a derivation of degree
$\widetilde{X}+1$ (with respect to the $\cup$ product).
Clearly,
$\textup{\textlbrackdbl}
\overline{X}, \, \cdot \,\textup{\textrbrackdbl}$ is also a derivation of degree $\widetilde{X}+1$ with respect to $\cup$.
Thus, to prove the proposition,
it suffices to show the $n=1$ case, i.e.,
\begin{equation}\label{Eqt:toproven=1case}
\at_{\DD_{\mathrm{poly}}}^{\nabla^{\mathrm{can}} }(\overline{X}, \overline{D})=\textup{\textlbrackdbl}
\overline{X}, \overline{D}\textup{\textrbrackdbl}\, ,
\end{equation}
 for all $X\in \XX(\cM)$ and $D \in \Dp$.
 In fact, by the definition of Atiyah cocycle in Equation  \eqref{Atiyahcocycleformula}, we have
\begin{equation*}
\begin{split}	
\at_{\DD_{\mathrm{poly}}}^{\nabla^{\mathrm{can}} }
(\overline{X}, \overline{D})
={} &
(-1)^{\widetilde{X}}
(d_H+L_Q)
\nabla^{\mathrm{can}}_{X}
(D) \\
& \quad -
(-1)^{\widetilde{X}}
\nabla^{\mathrm{can}}
(\id \otimes(d_H+L_Q)
+
L_Q\otimes \id)(X \widetilde{\otimes} D) \\
={} &
(-1)^{\widetilde{X}}( d_H\nabla^{\mathrm{can}}_{X}
(D)
-
(-1)^{\widetilde{X}}\nabla^{\mathrm{can}}_{X}
d_H(D)  )\\
& \quad + (-1)^{\widetilde{X}}
(L_Q\nabla^{\mathrm{can}}_{X}(D)
-
(-1)^{\widetilde{X}}\nabla^{\mathrm{can}}_{X}
L_Q(D)
-
\nabla^{\mathrm{can}}_{[Q,X]}(D)
) \, .
\end{split}
\end{equation*}
We now examine   the last two lines:
\begin{itemize}
\item
By Equation  \eqref{coproduct}, we have
\begin{equation*}
\begin{split}
d_H\nabla^{\mathrm{can}}_{X}D & = \, \, d_H(XD) \\
& \quad = \,(-1)^{\widetilde{X}+|D|}(XD\widetilde{\otimes}1-\Delta(XD)-(-1)^{\widetilde{X}+|D|}1\widetilde{\otimes}XD)\, ,
\end{split}	
\end{equation*}
and
\begin{equation*}
\begin{split}
(-1)^{\widetilde{X}}\nabla^{\mathrm{can}}_{X}d_H D
& = (-1)^{\widetilde{X}+|D|}\nabla^{\mathrm{can}}_{X}(D\widetilde{\otimes}1-\Delta(D)-(-1)^{|D|}1\widetilde{\otimes}D)  \\
& \quad = (-1)^{\widetilde{X}+|D|}(XD\widetilde{\otimes}1+(-1)^{\widetilde{X}|D|}D\widetilde{\otimes}X-\nabla^{\mathrm{can}}_{X}\Delta(D) \\
& \quad\quad -(-1)^{|D|}X\widetilde{\otimes}D -(-1)^{\widetilde{X}+|D|}1\widetilde{\otimes}XD )\, .
\end{split}
\end{equation*}
Thus by Lemma \ref{canonicalconnectionandDelta}, we see that
\begin{equation*}
\begin{split}
d_H\nabla^{\mathrm{can}}_{X}D
-
(-1)^{\widetilde{X}}\nabla^{\mathrm{can}}_{X}d_H D
& = (-1)^{\widetilde{X}}(X\widetilde{\otimes}D-
(-1)^{(\widetilde{X}+1)|D|}D\widetilde{\otimes}X)\\
& \quad = (-1)^{\widetilde{X}}
\textup{\textlbrackdbl}
\overline{X}, \overline{D}\textup{\textrbrackdbl}\, .
\end{split}
\end{equation*}

\item
By direct computation, we have
\begin{equation*}
\begin{split}
& \quad \, \, \, L_Q\nabla^{\mathrm{can}}_{X}
D-
(-1)^{\widetilde{X}}\nabla^{\mathrm{can}}_{X}
L_Q D
-
\nabla^{\mathrm{can}}_{[Q,X]}D \\
& =
[Q,XD]-(-1)^{\widetilde{X}}X[Q,D]-[Q,X]D =0\, .
\end{split}
\end{equation*}
\end{itemize}
Now Equation  \eqref{Eqt:toproven=1case} is clear.
\end{proof}

As an immediate consequence, we have
\begin{corollary}\label{bracket2}
The Atiyah cocycle $\alpha^{\nabla^{\mathrm{can}}}_{\Lp}$
is given by
\[
\alpha^{\nabla^{\mathrm{can}}}_{\Lp}
(\overline{X}, \text{ }
\textup{\textlbrackdbl} \overline{D_1},\cdots , \textup{\textlbrackdbl}
\overline{D_{n-1}}, \overline{D_{n}} \textup{\textrbrackdbl},\cdots  \textup{\textrbrackdbl}
)
=
\textup{\textlbrackdbl} \overline{X}, \textup{\textlbrackdbl} \overline{D_1}, \cdots , \textup{\textlbrackdbl}
\overline{D_{n-1}}, \overline{D_{n}} \textup{\textrbrackdbl},\cdots  \textup{\textrbrackdbl}
\textup{\textrbrackdbl}\, ,
\]
for all $X\in \XX(\cM)$ and
$D_i\in \Dp$.
\end{corollary}
\subsection{The main result}\

The main result in this section is the following:
\begin{theorem}\label{maintheorem}
\
\begin{itemize}
\item[(1)]
The natural inclusion map
$\theta:~ (\XX(\cM)[-1], \, L_Q) \rightarrow (\tot\Lp, \, L_Q+d_H)$ is an isomorphism in the homotopy category $\Pi(\Mod)$ of $(\smf_{\cM}, \, Q)$-modules.
Moreover, it is an isomorphism of Lie algebra objects
$(\XX(\cM)[-1], \, L_Q; \, \at_{\cM}) \rightarrow (\tot\Lp, \, L_Q+d_H; \, \nbk)$,
i.e., the   diagram
\begin{equation}\label{maindiagramII}
\xymatrix{
(\XX(\cM)[-1], \, L_Q) \otimes_{\smf_{\cM}} (\XX(\cM)[-1], \, L_Q)  \ar[r]^>>>>>>>>>>>{\at_{\cM}} \ar[d]_{\theta \otimes \theta} &
(\XX(\cM)[-1], \, L_Q)\ar[d]^{\theta}\\
(\tot\Lp, \, L_Q+d_H)
\otimes_{\smf_{\cM}}
(\tot\Lp, \, L_Q+d_H)\ar[r]^>>>>>>{\nbk} & (\tot \functorL(\DD_{\mathrm{poly}}^1), \, L_Q+d_H)\\
}
\end{equation}
commutes in $\Pi(\Mod)$.

\item[(2)]
The $(\smf_{\cM}, \, Q)$-module
$(\tot \Dpp, \, L_Q+d_H)$
is a Hopf algebra object in $\Pi(\Mod)$,
and is the universal enveloping algebra of
the Lie algebra object
$(\XX(\cM)[-1], \, L_Q; \,\at_{\cM} )$
in $\Pi(\Mod)$.	
\end{itemize}
\end{theorem}

\begin{proof}
Statement $(1)$ follows from Proposition \ref{betaquasi-isomorphism}, Proposition \ref{functoriality} (the functoriality of Atiyah classes)  and Corollary \ref{bracket2}.

Recall that in Section \ref{Section Hopf algebra}, we have defined the Hopf algebra structure on
$(\Dpp, \, L_Q, \, d_H)$, which   induces a Hopf algebra structure on
the $(\smf_{\cM}, \, Q)$-module
$(\tot \Dpp, \, L_Q+d_H)$.
By Corollary \ref{UL2},
$(\Dpp, \, L_Q, \, d_H)$ is the universal enveloping algebra of
the Lie algebra object $(\Lp, \, L_Q, \, d_H)$
in the category $Ch(\Mod)$ of
$(\smf_{\cM}, \, Q)$-complexes.
By the $\mathrm{tot}$ functor, we see that
the $(\smf_{\cM}, \, Q)$-module
$(\tot \Dpp, \, L_Q+d_H)$ is the universal enveloping algebra of
the Lie algebra object $(\tot\Lp, L_Q+d_H)$
in the category $\Mod$, as well as in the homotopy category $\Pi(\Mod)$ of
$(\smf_{\cM}, \, Q)$-modules.
As we have an isomorphism
$\theta: (\XX(\cM)[-1], \, L_Q; \, \at_{\cM}) \rightarrow (\tot\Lp, \, L_Q+d_H; \, \nbk)$
of Lie algebra objects
in $\Pi(\Mod)$,
$(\tot \Dpp, \, L_Q+d_H)$ is the universal enveloping algebra of the Lie algebra object
$(\XX(\cM)[-1], \, L_Q; \, \at_{\cM})$	
in
$\Pi(\Mod)$. This proves Statement $(2)$.
\end{proof}

The following corollaries are direct consequences of Theorem \ref{R1} and Theorem \ref{maintheorem}.
They are, respectively, analogues of Corollary $1$ and Theorem $3$ in \cite{Ramadoss}.
\begin{corollary}\label{Rcorollary1}
The following diagram commutes in the homotopy category $\Pi(\Mod)$:
\begin{equation*}\label{Rcorollary1diagram}
\xymatrix{
S^{\bullet}(\XX(\cM)[-1], \, L_Q) \otimes_{\smf_{\cM}} (\XX(\cM)[-1], \, L_Q)  \ar[r]^-{\mu \circ \frac{\widetilde{\omega}}{1-e^{- \widetilde{\omega}}}} \ar[d]_{\hkr \otimes \theta} &
S^{\bullet}(\XX(\cM)[-1], \, L_Q)\ar[d]^{\hkr}\\
(\tot\Dpp, \, L_Q+d_H)
\otimes
(\tot\Lp, \, L_Q+d_H)\ar[r]^-{\cup} & (\tot\Dpp, \, L_Q+d_H) \, . \\
}
\end{equation*}
Here $\mu$ denotes the natural symmetric product of $S^{\bullet}(\XX(\cM)[-1], \, L_Q)$, and the map
\[
\widetilde{\omega}:~~ S^{\bullet}(\XX(\cM)[-1], \, L_Q)  \otimes_{\smf_{\cM}} (\XX(\cM)[-1], \, L_Q) \rightarrow
S^{\bullet}(\XX(\cM)[-1], \, L_Q)  \otimes_{\smf_{\cM}} (\XX(\cM)[-1], \, L_Q)
\]
is defined by
\begin{equation*}\label{tildeomegaformula}
\widetilde{\omega}( (\overline{X}_1 \odot_{\smf_{\cM}} \cdots \odot_{\smf_{\cM}} \overline{X}_n) \otimes_{\smf_{\cM}}\overline{X}):~~=
\sum_{i=1}^n (-1)^{\Diamond_i}(\overline{X}_1 \odot_{\smf_{\cM}} \cdots \widehat{\overline{X}}_i \cdots \odot_{\smf_{\cM}} \overline{X}_n)\otimes_{\smf_{\cM}}
\at_{\cM}(\overline{X}_i, \, \overline{X})
\end{equation*}
for homogeneous elements $\overline{X}_i, \overline{X} \in \XX(\cM)[-1]$, where we denote by
$\odot_{\smf_{\cM}}$ the $\smf_{\cM}$-symmetric tensor product in
$S^{\bullet}(\XX(\cM)[-1])$
and
$\Diamond_i=(\widetilde{X}_{i+1}+\cdots+\widetilde{X}_{n}+n-i)(\widetilde{X}_{i}+1)$.
\end{corollary}

\begin{corollary}\label{R3}
For any $(\smf_{\cM}, \, Q)$-module
$(\cMmodule, \, L_Q)$, the Atiyah class
\[
\at_{\cMmodule}:~~
(\XX(\cM)[-1], \, L_Q)\otimes_{\smf_{\cM}}(\cMmodule, \, L_Q) \rightarrow
(\cMmodule, \, L_Q)
\]
can be uniquely lifted to a morphism
$\at_{\cMmodule}':~~ (\tot \Dpp, \, L_Q + d_H)\otimes_{\smf_{\cM}} (\cMmodule, \, L_Q)\rightarrow (\cMmodule, \, L_Q)$ in the homotopy category $\Pi(\Mod)$, i.e., the following diagram
\begin{equation*}\label{lifting}
\begin{xy}
(0,25)*+{(\XX(\cM)[-1], \, L_Q)\otimes_{\smf_{\cM}}(\cMmodule, \, L_Q)}="a";
(0,0)*+{(\tot\Dpp, \, L_Q+d_H) \otimes_{\smf_{\cM}} (\cMmodule, \, L_Q)}="b";
(50,25)*+{(\cMmodule, \, L_Q) \, .}="c";
{\ar^{\theta \otimes \id} "a";"b"};
{\ar_{\exists \mathrm{!} \, \at_{\cMmodule}' } "b";"c"};
{\ar^>>>>>>>>>{\at_{\cMmodule}} "a";"c"};
\end{xy}
\end{equation*}
commutes in the homotopy category $\Pi(\Mod)$.
\end{corollary}

\section{Hopf algebras arising from Lie pairs}\label{Section6}
 Throughout this section, we denote by $R=\smf(M, \fK)$ the ring of $\fK$-valued smooth functions on a smooth manifold $M$.
\subsection{The Atiyah class $\at_{L/A}$ of a Lie pair $(L, \, A)$}\

The notion of Lie pair is a natural framework encompassing a range of diverse geometric contexts including complex manifolds, foliations, and $\mathfrak{g}$-manifolds (i.e., manifolds endowed with an infinitesimal action of a Lie algebra
$\mathfrak{g}$). By a \textbf{Lie pair}, we mean an inclusion $i:A \hookrightarrow L$ of Lie $\fK$-algebroids over a smooth manifold $M$. Recall that a \textbf{Lie $\fK$-algebroid} is a $\fK$-vector bundle $L \rightarrow M$, whose space of sections is endowed with a Lie bracket $\embk$, together with a bundle map $\rho: ~~ L \rightarrow T_M\otimes_{\R} \fK$ called \textbf{anchor} such that
$\rho:~~ \Gamma(L) \rightarrow \XX(M)\otimes_{\R}\fK$ is a morphism of Lie algebras and
\[
[X, \, fY]=f[X, \,Y]+(\rho(X)f)Y,
\]
for all $X, Y \in \Gamma(L)$ and $f\in R$.

Given a Lie pair $(L, \, A)$.
For convenience, let us denote the quotient
vector bundle $L/A$ by $B$.
Denote by $\mathrm{pr}_{B}: L \rightarrow B$
the projection map.
Note that $B$ is naturally an $A$-module:
\begin{equation}\label{Bott A-module}
\begin{split}
\Bott:~ & \Gamma(A)\otimes_{\fK}\Gamma(B)\rightarrow \Gamma(B) \, , \\
& \Bott_{a}b:=\mathrm{pr}_{B}[a, \, l] \, ,
\end{split}
\end{equation}
where $a\in \Gamma(A)$, $b\in \Gamma(B)$ and
$l\in \Gamma(L)$ is any section of $L$ satisfying
$\mathrm{pr}_{B}(l)=b$.
The flat $A$-connection $\Bott$ on $B$
is also known as the Bott connection
\cites{Bott, C-S-X2}.

We have a natural short exact sequence
of vector bundles over $M$:
\begin{equation}\label{Liepairshortexactsequence}
0\rightarrow A \stackrel{i}{\longrightarrow} L \stackrel{\mathrm{pr}_B}{\longrightarrow} B \rightarrow 0 \, .
\end{equation}
A splitting of Sequence
\eqref{Liepairshortexactsequence}
is a map $j: B \rightarrow L$ of vector bundles
such that $\mathrm{pr}_B \circ j=\id_{B}$.
The choice of a splitting determines an isomorphism of vector bundles over $M$:
\begin{equation*}
\begin{split}
A \oplus B & \simeq L \, , \\
(a,b)& \mapsto i(a)+j(b) \, ,
\end{split}	
\end{equation*}
for all $a\in \Gamma(A)$ and $b\in \Gamma(B)$.
In what follows, we fix the splitting $j$ of Sequence \eqref{Liepairshortexactsequence}
so that one treats $L=A\oplus B$ directly.

Recall the definition of Atiyah class of the Lie pair
$(L, \, A)$ introduced in \cite{C-S-X2}.
Let $\nabla:~~ \Gamma(L) \otimes_{\fK} \Gamma(B) \rightarrow \Gamma(B)$ be a smooth $L$-connection on $B$ which extends the Bott connection, i.e.,
\begin{equation}\label{compatible with Bott}
\nabla_{i(a)}b=\Bott_{a}b ,  \quad
\forall a\in \Gamma(A) , \, b \in \Gamma(B) \, .
\end{equation}
Denote by $A^{\vee}$ (resp. $B^{\vee}$)
the $\fK$-dual of the vector bundle $A$
(resp. $B$).
 We have a $1$-cocycle $\at^{\nabla}_{B}$
in the Chevalley-Eilenberg complex
$(\Gamma(B^{\vee}\otimes\End(B))\otimes_R \Gamma(\wedge^{\bullet}(A^{\vee})), \, d_{\mathrm{CE}})$
of the Lie algebroid $A$ valued in the $A$-module $B^{\vee}\otimes\End(B)$:
\begin{equation*}
\at^{\nabla}_{B} \in
\Gamma( (B^{\vee}\otimes\End(B)) \otimes A^{\vee})
\simeq \Hom_{R}(\Gamma(A\otimes B), \Gamma(\mathrm{End}(B)))\, .	
\end{equation*}
It is defined by the following formula:
\begin{equation*}\label{LiepairAtiyahcocycleformula}
\at^{\nabla}_B(a, b)e:~~=\Bott_{a}\nabla_{j(b)}e
-\nabla_{j(b)}\Bott_{a}e-\nabla_{[i(a),j(b)]}e \, ,
\end{equation*}
for all $a \in \Gamma(A)$ and $b, e \in \Gamma(B)$.
We call $\at^{\nabla}_{B}$
the \textbf{Atiyah cocycle of the Lie pair $(L, \, A)$}.
The cohomology class
\begin{equation*}\label{LiepairAtiyahclassdefinition}
\at_B=[\at^{\nabla}_B]\in H^1_{\mathrm{CE}}(A, B^{\vee}\otimes\End(B))	
\end{equation*}
does not depend on the choice of $j$ and $\nabla$.
We call $\at_{B}$ the \textbf{Atiyah class of the Lie pair
$(L, \, A)$}.

According to a theorem of Va\u{\i}ntrob \cite{Vaintrob},
given a $\fK$-vector bundle $A$
over a smooth manifold $M$, the homological vector fields on the graded manifold
$A[1]$ are in one-one correspondence with
the Lie algebroid structures on $A$.
Indeed,
the space of functions on the graded manifold $A[1]$ is
$\Omega(A):=\Gamma(\wedge^{\bullet}(A^{\vee}))$,
and the homological vector field on
$A[1]$ is the Chevalley-Eilenberg differential
$d_{\mathrm{CE}}^{A}: \Gamma(\wedge^{\bullet}(A^\vee))\to \Gamma(\wedge^{\bullet+1}(A^\vee))$.
In what follows,
we denote by
$(A[1], \, d_{\mathrm{CE}}^{A})$ the dg manifold
arising from the Lie algebroid structure of $A$,
and by
$(\Omega(A), \, d_{\mathrm{CE}}^{A})$
the dg algebra of functions on
$(A[1], \, d_{\mathrm{CE}}^{A})$.

Let $B^{!}$ be the graded vector bundle over
$A[1]$ which is the
pullback of the bundle $B\rightarrow M$ along the natural projection
$A[1] \rightarrow M$:
\[
\xymatrix{
B^{!}  \ar[r] \ar[d] & B \ar[d]\\
A[1] \ar[r] & M \, .\\
}
\]
By definition we have
$\Gamma(B^{!})=\Gamma(B \otimes \wedge^{\bullet}(A^{\vee}))=\Gamma(B)\otimes_{R}\Omega(A)$,
so $\Gamma(B^{!})$ has a natural $\Omega(A)$-module structure.
Meanwhile, the natural $A$-module structure of $B$ (see Equation  \eqref{Bott A-module})
defines
the Chevalley-Eilenberg differential
$d_{\mathrm{CE}}^{B^{!}}: \Gamma(B^{!}) \rightarrow \Gamma(B^{!})$.
It turns out that the Chevalley-Eilenberg complex
$(\Gamma(B^{!}), \, d_{\mathrm{CE}}^{B^{!}})$
is a
$(\Omega(A), \, d_{\mathrm{CE}}^{A})$-module.
In other words, we have a dg vector bundle
$(B^{!}, \, d_{\mathrm{CE}}^{B^{!}}) \rightarrow (A[1], \, d_{\mathrm{CE}}^{A})$.

As the Lie pair Atiyah class $\at_{B}$ is an element in the space
\begin{equation*}
\begin{split}
H^1_{\mathrm{CE}}(A, & B^{\vee}\otimes \End(B))
\simeq \\
&\Hom_{\mathrm{H}((\Omega(A),\, d_{\mathrm{CE}}^A)\mathbf{-mod})}
((\Gamma(B^{!})[-1], \, d_{\mathrm{CE}}^{B^{!}})\otimes_{\Omega(A)}
(\Gamma(B^{!})[-1], \, d_{\mathrm{CE}}^{B^{!}})
, (\Gamma(B^{!})[-1], \, d_{\mathrm{CE}}^{B^{!}})) \, ,
\end{split}
\end{equation*}
we consider it as a morphism
\[
\at_{B}:~~ (\Gamma(B^{!})[-1], \, d_{\mathrm{CE}}^{B^{!}})\otimes_{\Omega(A)}
(\Gamma(B^{!})[-1], \, d_{\mathrm{CE}}^{B^{!}})
\rightarrow (\Gamma(B^{!})[-1], \, d_{\mathrm{CE}}^{B^{!}})
\]
in the homology category
$\mathrm{H}((\Omega(A), \, d_{\mathrm{CE}}^A)\mathbf{-mod})$ of
$(\Omega(A), \, d_{\mathrm{CE}}^{A})$-modules.

\subsection{The Hopf algebra $(D_{\mathrm{poly}}^{\bullet}(L/A), \, d_H)$}\label{Sec:HopfalgebraDpolyLA}\

It is known \cite{Xu} that the universal enveloping algebra $U(L)$ of a Lie algebroid $L$ admits a cocommutative coassociative coproduct $\Delta:~~ U(L) \rightarrow U(L)\otimes_R U(L)$, which is defined on generators as follows:
\begin{equation*}
\Delta(f)=f\otimes_R 1=1 \otimes_R f, \, \forall f \in R, \,
\quad \quad
\Delta(p)=p\otimes_R 1 + 1 \otimes_R p, \, \forall p \in \Gamma(L).	
\end{equation*}
Moreover, $U(L)$ is an $L$-module since
each section $l$ of $L$ acts on $U(L)$
by left multiplication.

Now given the Lie pair $(L, \, A)$, consider the quotient $D_{\mathrm{poly}}^1(B):=\frac{U(L)}{U(L)\Gamma(A)}$.
It is straightforward to see that the coproduct on $U(L)$ induces a coproduct
$\Delta: D_{\mathrm{poly}}^1(B) \rightarrow D_{\mathrm{poly}}^1(B)\otimes_R D_{\mathrm{poly}}^1(B)$
on $D_{\mathrm{poly}}^1(B)$ and the action of $L$ on $U(L)$ determines
an action of $A$ on
$D_{\mathrm{poly}}^1(B)$.
It turns out that
the quotient $D_{\mathrm{poly}}^{1}(B)$
is simultaneously a cocommutative coassociative
$R$-coalgebra and an $A$-module.

Let $D_{\mathrm{poly}}^{n}(B)$ denote
the $n$-th tensorial power
$D_{\mathrm{poly}}^{1}(B)\otimes_R \cdots \otimes_R D_{\mathrm{poly}}^{1}(B)$
of $D_{\mathrm{poly}}^{1}(B)$ and,
for $n=0$, set $D_{\mathrm{poly}}^{0}(B)=R$.
Sti\'enon, Xu, and the second author defined a complex
$(D_{\mathrm{poly}}^{\bullet}(B), \, d_H)$
in the category of $A$-modules \cite{C-S-X},
where
$d_H:~~ D_{\mathrm{poly}}^n(B) \rightarrow D_{\mathrm{poly}}^{n+1}(B)$
is the Hochschild differential\footnote{Our definition of $d_H$ on $D_{\mathrm{poly}}^{\bullet}(B)$ is different from the   one   in   \cite{C-S-X} by a negative sign.
The reason is that we wish  
  $d_H$ on $  D_{\mathrm{poly}}^{\bullet}(B) $ to be  compatible with the one on $  \Verticalthings{D}_{\mathrm{poly}}$ given by Equation \eqref{Hochschild differential definition}. See Diagram \eqref{Imap}.} given by
\begin{equation*}\label{HochschildofLiepair}
\begin{split}
d_H(p_1 \otimes_R \cdots \otimes_R p_n)& =-(1\otimes_R p_1 \cdots \otimes_R p_n - \Delta(p_1) \otimes_R p_2 \cdots \otimes_R p_n + p_1 \otimes_R \Delta(p_2) \cdots \otimes_R p_n - \\
& \quad \quad \cdots +(-1)^n p_1 \otimes_R \cdots \otimes_R \Delta(p_n) +(-1)^{n+1}p_1\otimes_R \cdots \otimes_R p_n \otimes_R 1),
\end{split}
\end{equation*}
for $p_1, \cdots , p_n \in D_{\mathrm{poly}}^1(B)$.

Consider the inclusion
$\eta: R \hookrightarrow D_{\mathrm{poly}}^{\bullet}(B)$,
the projection
$\varepsilon : D_{\mathrm{poly}}^{\bullet}(B) \rightarrow  R$,
and the maps
$t: D_{\mathrm{poly}}^{\bullet}(B) \rightarrow D_{\mathrm{poly}}^{\bullet}(B)$
and
$\widetilde{\Delta}: D_{\mathrm{poly}}^{\bullet}(B) \rightarrow
D_{\mathrm{poly}}^{\bullet}(B)\otimes_{R} D_{\mathrm{poly}}^{\bullet}(B)$ defined, respectively, by
\begin{equation*}
t(p_1\otimes_R p_2 \otimes_R  \cdots \otimes_R p_n)
=(-1)^{\frac{n(n-1)}{2}}	
p_n \otimes_R p_{n-1} \otimes_R \cdots \otimes_R p_1
\end{equation*}
and
\begin{equation*}
\widetilde{\Delta}(p_1 \otimes_R p_2 \otimes_R \cdots \otimes_R  p_n)=\sum_{i+j=n}\sum_{\sigma \in \mathrm{Sh}(i, \, j)}\kappa(\sigma)\,
(p_{\sigma(1)}\otimes_R \cdots \otimes_R p_{\sigma(i)})
\otimes_R
(p_{\sigma(i+1)}\otimes_R \cdots \otimes_R p_{\sigma(n)}) \, ,
\end{equation*}
where $\mathrm{Sh}(i, \, j)$ denotes the set of
$(i, \, j)$-shuffles.
With the multiplication $\otimes_{R}$,
the comultiplication $\widetilde{\Delta}$,
the unit $\eta$,
the counit $\varepsilon$,
and the antipole $t$,
$(D_{\mathrm{poly}}^{\bullet}(B), \, d_H)$ is a Hopf algebra object in the category of cochain complexes of
$A$-modules \cite{C-S-X}.

Let $\functorL(D_{\mathrm{poly}}^{1}(B))$
be the free graded Lie algebra generated over $R$
by $D_{\mathrm{poly}}^{1}(B)$, where
$D_{\mathrm{poly}}^{1}(B)$ is placed in degree
$(+1)$.
In other words,
$\functorL(D_{\mathrm{poly}}^{1}(B))$
is the smallest Lie subalgebra of
$D_{\mathrm{poly}}^{\bullet}(B)$
containing $D_{\mathrm{poly}}^{1}(B)$.
The Lie bracket of two elements
$u\in D_{\mathrm{poly}}^{i}(B)$
and $v\in D_{\mathrm{poly}}^{j}(B)$
is the element
$\textup{\textlbrackdbl} u, \, v  \textup{\textrbrackdbl}
=u\otimes_R v- (-1)^{ij} v\otimes_R u \in D_{\mathrm{poly}}^{i+j}(B)$.
It turns out that
$(\functorL(D_{\mathrm{poly}}^{1}(B)), \, d_H; \, \nbk)$ is a Lie algebra object
in the category of cochain complexes of $A$-modules.

\subsection{The Fedosov dg Lie algebroid $(\cF, \, \LQF)$}\label{The Fedosov dg Lie algebroid}\

Let $(L, \, A)$ be a Lie pair
and let $\nabla:~~\Gamma(L)\otimes_{\fK} \Gamma(B) \rightarrow \Gamma(B)$ be an $L$-connection which extends the Bott $A$-connection on $B$ (see Equation  \eqref{compatible with Bott}).
Suppose further that $\nabla$ is torsion-free, i.e.,
\begin{equation*}
\nabla_{l_1}\mathrm{pr}_{B}(l_2)-\nabla_{l_2}\mathrm{pr}_{B}(l_1)=\mathrm{pr}_{B}([l_1, \, l_2]) \, , \,\,
\forall \, l_1 ,\, l_2 \in \Gamma(L) \, .
\end{equation*}
Here $\mathrm{pr}_{B}:~ \Gamma(L) \rightarrow \Gamma(B)$ denotes the projection map.
In \cite{S-X}, Sti\'enon and Xu defined a dg manifold
\[
(\Mfedo, \, \Qfedo):~~=(L[1]\oplus B, \, \fedo)
\]
which they called the \textbf{Fedosov dg manifold}
associated with the Lie pair $(L, \, A)$.
There are two equivalent constructions of
the homological vector field $\Qfedo$ arising from $(L, \, A)$.
One is by way of Fedosov's iteration method, the other is by way of the PBW map.
We will briefly recall the second one.
Let
\begin{equation*}\label{exp}
\pbw^{\nabla,j}:~~ \Gamma(SB) \rightarrow D_{\mathrm{poly}}^1(B)	
\end{equation*}
be the PBW map introduced in \cites{LG-S-X1,LG-S-X2}. Here $SB$ denotes the symmetric tensor algebra of the vector bundle $B$.
The map $\pbw^{\nabla,j}$ is an isomorphism of $R$-modules.
There is a canonical flat $L$-connection
\begin{equation}\label{canonical connection on Dpoly1B}
\begin{split}
\nabla^{\mathrm{can}}:~~ \Gamma(L)\otimes_{\fK} & D_{\mathrm{poly}}^1(B)\rightarrow D_{\mathrm{poly}}^1(B) \, , \\
& \nabla^{\mathrm{can}}_{l}u:~~ =l \cdot u
\end{split}
\end{equation}
for all $l\in \Gamma(L)$ and
$u\in D_{\mathrm{poly}}^1(B)$.
Pulling back   $\nabla^{\mathrm{can}}$
via the map $\pbw^{\nabla,j}$, we obtain a flat $L$-connection on $SB$:
\begin{equation*}\label{lighting connection on SB}
\begin{split}
\nabla^{\lightning}:~~ & \Gamma(L)\otimes_{\fK}\Gamma(SB) \rightarrow \Gamma(SB)\, , \\\nabla^{\lightning}_{l}s:~~ = & (\pbw^{\nabla,j})^{-1} \circ \nabla^{\mathrm{can}}_{l}\circ \pbw^{\nabla,j}(s)
\end{split}
\end{equation*}
for $l \in \Gamma(L)$ and $s \in \Gamma(SB)$.
Let $\hat{S}(B^{\vee})=\Hom_{\fK}(SB, \, M\times \fK)$ be the $\fK$-dual bundle of $SB$.
It turns out that there is a flat connection on
$\hat{S}(B^{\vee})$,
which is denoted by the same symbol $\nabla^{\lightning}$.
We denote the corresponding Chevalley-Eilenberg
differential by
\begin{equation*}
d_{L}^{\nabla^{\lightning}}:~
\Gamma(\hat{S}(B^{\vee}))\otimes_R
\Gamma(\wedge^{\bullet}(L^{\vee}))
\rightarrow
\Gamma(\hat{S}(B^{\vee}))\otimes_R
\Gamma(\wedge^{\bullet+1}(L^{\vee})) \, .
\end{equation*}
Henceforth we will write
$\Omega(L)$ for $\Gamma(\wedge^{\bullet}(L^{\vee}))$.

The ring of functions on the graded manifold
$\Mfedo=L[1]\oplus B$ is clearly
$\Gamma(\hat{S}(B^{\vee}))\otimes_{R}\Omega(L)$.
Together with the homological vector field
$\Qfedo:=\fedo$,
the pair $(\Mfedo, \, \Qfedo)$ is called the Fedosov dg manifold.
(For more details, see \cite[Section $2$]{S-X} and \cites{L-S,B-S-X}.)
Clearly, $\Mfedo$ can be thought of as a vector bundle over $L[1]$. We denote by
$\pi:~~ \Mfedo \rightarrow L[1]$ the natural projection.
 It can be verified that $\pi:~(\Mfedo, \, \Qfedo)\to (L[1], \, d_{\mathrm{CE}}^L)$ is a morphism of dg manifolds (see \cite{B-V} and \cite{S-X} for details).
 Hence, the dg subbundle
 $(\cF, \, L_{\Qfedo}):=(\ker \pi_{\ast}, \, L_{\Qfedo}) \subset  (T_{\Mfedo}, \, L_{\Qfedo})$ is a dg foliation of the dg manifold $(\Mfedo, \, \Qfedo)$.
Sti\'enon and Xu called it a \textbf{Fedosov dg Lie algebroid}.
 By the construction of $\cF$, we have the following fact (see \cite{B-S-X}):
\begin{equation*}\label{Fexpression}
\Gamma(\cF) =\mathrm{Der}_{\Omega(L)}(\smf_{\Mfedo}, \,\smf_{\Mfedo})
\simeq\Gamma(B)\otimes_{R}\smf_{\Mfedo}=\Gamma(B\otimes \wedge^{\bullet}(L^{\vee}) \otimes \hat{S}(B^{\vee})) \, .
\end{equation*}
In other words,
the space $\Gamma(\cF)$ consists of \textbf{vertical vector fields} along the natural projection 
$\pi:~~ \Mfedo \rightarrow L[1]$.

We have an obvious embedding map $\iota: A[1]\hookrightarrow \Mfedo=L[1]\oplus B$.
Moreover,
it is proved in \cites{S-X}
that
\begin{equation}\label{embedding of dg manifolds}
\iota: (A[1], \, d_{\mathrm{CE}}^{A})\hookrightarrow (\Mfedo, \, \Qfedo)
\end{equation}
is a morphism of dg manifolds as well as a quasi-isomorphism.
In other words, the induced map
$\iota^{\ast}:
(\smf_{\Mfedo}, \, \Qfedo)=(\Gamma(\wedge^{\bullet}(L^{\vee})\otimes\hat{S}(B^{\vee})), \,
\fedo)
\rightarrow (\Omega(A), \, d_{\mathrm{CE}}^{A})$
is a quasi-isomorphism of differential graded algebras.

By construction of the Fedosov dg Lie algebroid
$(\cF, \, L_{\Qfedo})$,
we have a natural pullback diagram of dg vector bundles
\[
\xymatrix{
(B^{!}, \, d_{\mathrm{CE}}^{B^{!}})  \ar[r]^{\iota} \ar[d] & (\cF, \, \LQF) \ar[d]\\
(A[1], \, d_{\mathrm{CE}}^{A}) \ar[r]^{\iota} & (\Mfedo, \, \Qfedo ) \, .\\
}
\]
In other words, we have
$\iota^{\ast}(\cF, \, \LQF)=
(B^{!}, \, d_{\mathrm{CE}}^{B^{!}})$, or
\[
(\Gamma(B^{!}), \, d_{\mathrm{CE}}^{B^{!}})=(\Gamma(\cF), \, \LQF)\otimes_{\smf_{\Mfedo}}
(\Omega(A), \, d_{\mathrm{CE}}^{A})\, .
\]	
By this fact, we can treat
the $(\Omega(A), \, d_{\mathrm{CE}}^{A})$-module
$(\Gamma(B^{!}), \, d_{\mathrm{CE}}^{B^{!}})$ as a
$(\smf_{\Mfedo}, \, \Qfedo)$-module.

Consider the
$(\smf_{\Mfedo}, \, \Qfedo)$-complexes of polyvector fields and polydifferential operators
on $(\Mfedo, \, \Qfedo)$, which are denoted by
$(\Tpp(\Mfedo), \, L_{\Qfedo}, \, \delta=0)$ and
$(\Dpp(\Mfedo), \, L_{\Qfedo}, \, \delta=d_H)$, respectively.
Let $\Verticalthings{D}_{\mathrm{poly}}^1\subset \cD_{\Mfedo}$ be the subspace
of \textbf{vertical differential operators} along the natural projection 
$\pi:~\Mfedo \rightarrow L[1]$.
Indeed, $\Verticalthings{D}_{\mathrm{poly}}^1$ can be expressed by
\begin{equation*}\label{D(F)expression}
\Verticalthings{D}_{\mathrm{poly}}^1=\Gamma(SB)\otimes_{R}\smf_{\Mfedo}=\Gamma(SB \otimes \wedge^{\bullet}(L^{\vee})\otimes \hat{S}(B^{\vee})) \, .
\end{equation*}
(See \cite{B-S-X}.)
As
$(\cF, \, L_{\Qfedo})$
is a dg vector bundle,
$(\Verticalthings{D}_{\mathrm{poly}}^1, \,  L_{\Qfedo})$
is a $(\smf_{\Mfedo}, \, \Qfedo)$-module.
We also have
the $(\smf_{\Mfedo}, \, \Qfedo)$-subcomplex
\begin{equation*}
(\Verticalthings{T}_{\mathrm{poly}}, \, L_{\Qfedo}, \, \delta=0)\subset (\Tpp(\Mfedo), \, L_{\Qfedo}, \, 0)	
\end{equation*}
of \textbf{vertical polyvector fields} and
the $(\smf_{\Mfedo}, \, \Qfedo)$-subcomplex
\begin{equation*}
(\Verticalthings{D}_{\mathrm{poly}}, \, L_{\Qfedo}, \, \delta=d_H)\subset
(\DD_{\mathrm{poly}}(\Mfedo), \, L_{\Qfedo}, \, d_H)
\end{equation*}
of \textbf{vertical polydifferential operators}.
Meanwhile, we have identifications of
$\smf_{\Mfedo}$-modules (see \cite{B-S-X}):
\begin{equation*}
\begin{split}
& \tot\Verticalthings{T}_{\mathrm{poly}}=
\Gamma(\wedge^{\bullet}B)\otimes_R \smf_{\Mfedo}
=
\Gamma(\wedge^{\bullet}B \otimes \wedge^{\bullet}(L^{\vee})\otimes \hat{S}(B^{\vee}))  \\
\text{and}\quad\quad\quad\quad
& \tot\Verticalthings{D}_{\mathrm{poly}}=\Gamma(\otimes^{\bullet}(SB))\otimes_R \smf_{\Mfedo}=
\Gamma(\otimes^{\bullet}(SB) \otimes \wedge^{\bullet}(L^{\vee})\otimes \hat{S}(B^{\vee}))  \, .
\end{split}
\end{equation*}


\subsection{The Atiyah class $\at_{\cF}$}\

In general, given a dg foliation
$(\cF, \, L_Q)$ of a dg manifold $(\cM, \, Q)$, one has the notions of $\cF$-Atiyah cocycle and $\cF$-Atiyah class of a $(\smf_{\cM}, \, Q)$-module
$(\cMmodule, \, L_Q)$.
Their definitions are completely analogous to the definitions of the Atiyah cocycle and the Atiyah class of the $(\smf_{\cM}, \, Q)$-module $(\cMmodule, \, L_Q)$.
For details, see the paper \cite{M-S-X}.
In particular, we consider the Fedosov dg Lie algebroid
 $(\cF, \, \LQF)$ and the $(\smf_{\Mfedo}, \, \Qfedo)$-module $(\Gamma(\cF)[-1], \, \LQF)$.
 In this situation, we have a canonical $\cF$-connection $\nabla^{\cF}$ on the $R$-module $\Gamma(\cF)$ characterized by the relation (see \cite{L-S-X2} for more details)
 \[
\nabla^{\cF}_{X}Y=0,\qquad
\forall X, Y \in \Gamma(B)\subset \Gamma(\cF).
\]
The operator $\nabla^{\cF}$ also defines a
$\cF$-connection on $\Gamma(\cF)[-1]$.
Accordingly, the \textbf{$\cF$-Atiyah cocycle} associated to the $\cF$-connection $\nabla^{\cF}$ on $\Gamma(\cF)[-1]$ is a morphism of
$(\smf_{\Mfedo}, \, \Qfedo)$-modules:
\[
\at^{\nabla^{\cF}}_{\Gamma(\cF)[-1]}:~~ (\Gamma(\cF)[-1], \, \LQF)\otimes_{\smf_{\Mfedo}} (\Gamma(\cF)[-1], \, \LQF)
\rightarrow
(\Gamma(\cF)[-1], \, \LQF) \, ,
\]
\[
\at^{\nabla^{\cF}}_{\Gamma(\cF)[-1]}(\overline{X},\overline{Y})
:=
(-1)^{\widetilde{X}}(\LQF(\nabla^{\cF}_X \overline{Y})-\nabla^{\cF}_{[\Qfedo,X]}\overline{Y}-(-1)^{\widetilde{X}}\nabla^{\cF}_{X}\LQF(\overline{Y})).
\]
Here the notation $\overline{X}\in\Gamma(\cF)[-1]$ (resp.  $\overline{Y}\in \Gamma(\cF)[-1]$) denotes the element that corresponds to $X\in \Gamma(\cF)$ (resp. $Y\in \Gamma(\cF)$).
Here and in the rest of this paper,
we denote by
$(C_{\Mfedo}^{\infty}, \Qfedo)\mathrm{-}\mathbf{mod}$
the category of
$(\smf_{\Mfedo}, \, \Qfedo)$-modules.
The \textbf{$\cF$-Atiyah class} of $(\Gamma(\cF)[-1], \, \LQF)$, denoted simply by $\at_{\cF}$, is the morphism
in the homology category
$\mathrm{H}( (C_{\Mfedo}^{\infty}, \Qfedo)\mathrm{-}\mathbf{mod} )$
induced by
$\at^{\nabla^{\cF}}_{\Gamma(\cF)[-1]}$:
\begin{equation*}
\at_{\cF}: (\Gamma(\cF)[-1], \, \LQF) \otimes_{\smf_{\Mfedo}}(\Gamma(\cF)[-1], \, \LQF)
\rightarrow
(\Gamma(\cF)[-1], \, \LQF) \, .	
\end{equation*}
Similar to Theorem \ref{Thm:AtiyahdefinesLiebracket},
 the Atiyah class $\at_{\cF}$ defines a Lie algebra object
$(\Gamma(\cF)[-1], \, \LQF; \, \at_{\cF})$ in
$\mathrm{H}( (C_{\Mfedo}^{\infty}, \Qfedo)\mathrm{-}\mathbf{mod} )$.

Let
$\functorL(\Verticalthings{D}_{\mathrm{poly}}^1)$ $\subset$   $\functorL(\Dpp^1(\Mfedo))$
be the free Lie algebra generated over
$\smf_{\Mfedo}$ by
$\Verticalthings{D}_{\mathrm{poly}}^1$ (where
$\Verticalthings{D}_{\mathrm{poly}}^1$ is placed in horizontal degree $(+1)$).
As the vertical differential
$L_{\Qfedo}$: $
\functorL(\Dpp^1(\Mfedo))$ $\rightarrow$ $\functorL(\Dpp^1(\Mfedo)) $
and the horizontal differential
$d_H$: $\functorL(\Dpp^1(\Mfedo))$ $\rightarrow$ $\functorL(\Dpp^1(\Mfedo))$
preserve
the subspace
$\functorL(\Verticalthings{D}_{\mathrm{poly}}^1)$ of $ \functorL(\Dpp^1(\Mfedo))$,
$(\functorL(\Verticalthings{D}_{\mathrm{poly}}^1)$,   $\LQF$, $d_H)$  
is a Lie subalgebra object of $(\functorL(\Dpp^1(\Mfedo) )$, $\LQF$, $d_H)$ in the
category
$Ch( (\smf_{\Mfedo}, \, \Qfedo)\mathrm{-}\mathbf{mod} )$
of $(\smf_{\Mfedo}$,   $\Qfedo)$-complexes.
  We have the following analogue of Theorem \ref{maintheorem}:
\begin{theorem}\label{betaforfoliation}
Let $(\cF, \, \LQF)$ be the Fedosov dg Lie algebroid as above.
\begin{itemize}
\item[(1)]	
The natural inclusion map
$\theta:~~ (\Gamma(\cF)[-1], \, \LQF) \rightarrow (\tot \functorL(\Verticalthings{D}_{\mathrm{poly}}^1), \, \LQF+d_H )$
is an isomorphism in
the homotopy category
$\Pi( (\smf_{\Mfedo}, \, \Qfedo)\mathrm{-}\mathbf{mod} )$
of $(\smf_{\Mfedo}, \, \Qfedo)$-modules.
Moreover, it is an isomorphism of Lie algebra objects
$(\Gamma(\cF)[-1], \, \LQF; \, \at_{\cF}) \rightarrow (\tot \functorL(\Verticalthings{D}_{\mathrm{poly}}^1),$
$\LQF+d_H; \, \nbk )$, i.e.,
the following diagram
\begin{equation*}\label{foliationdiagram}
\xymatrix{
(\Gamma(\cF)[-1], \, \LQF) \otimes_{\smf_{\Mfedo}} (\Gamma(\cF)[-1], \, \LQF)  \ar[r]^-{\at_{\cF}} \ar[d]_{\theta \otimes \theta} & (\Gamma(\cF)[-1], \, \LQF) \ar[d]^{\theta}\\
(\tot\functorL(\Verticalthings{D}_{\mathrm{poly}}^1), \, \LQF+d_H )
\otimes_{\smf_{\Mfedo}}
(\tot \functorL(\Verticalthings{D}_{\mathrm{poly}}^1), \, \LQF+d_H ) \ar[r]^-{\nbk} & (\tot \functorL(\Verticalthings{D}_{\mathrm{poly}}^1), \, \LQF+d_H )\, .\\
}
\end{equation*}
commutes in
$\Pi( (\smf_{\Mfedo}, \, \Qfedo)\mathrm{-}\mathbf{mod} )$.

\item[(2)]
The $(\smf_{\Mfedo},\Qfedo)$-module
$(\tot \Verticalthings{D}_{\mathrm{poly}},\LQF+d_H)$ is a Hopf algebra object in the category
$\Pi((\smf_{\Mfedo},\Qfedo)\mathrm{-}\mathbf{mod})$,
and is
the universal enveloping algebra of the Lie algebra object $(\Gamma(\cF)[-1], \, L_Q; \,\at_{\cM} )$ in
$\Pi((\smf_{\Mfedo}, \, \Qfedo)\mathrm{-}\mathbf{mod})$.
\end{itemize}
\end{theorem}

\subsection{Relation between $\at_{L/A}$ and $\at_{\cF}$}\

Recall the embedding map
$\iota: (A[1], \, d_{\mathrm{CE}}^{A})\hookrightarrow (\Mfedo, \, \Qfedo)$
of dg manifolds (see Equation  \eqref{embedding of dg manifolds}).
The following facts were proved in 
\cite[Proposition $3.3$]{B-S-X}
and \cite[Corollary $1.21$ and Proposition A.$11$]{L-S-X2}.

\begin{proposition}\label{leftface}\
\begin{itemize}
\item[(1)]
The restriction map $\iota^{\ast}:~~
(\Gamma(\cF), \, \LQF) \rightarrow
(\Gamma(B^{!}), \, d_{\mathrm{CE}}^{B^{!}})$
is a quasi-isomorphism of
$(\smf_{\Mfedo},$
$\Qfedo)$-modules.

\item[(2)]
The induced map
\begin{equation*}
\begin{split}
&  \iota^{\ast}:~~ (\Hom   _{\smf_{\Mfedo}}^{\bullet}(\Gamma(\cF)\otimes_{\smf_{\Mfedo}} \Gamma(\cF), \, \Gamma(\cF)), \, \LQF) \\
& \quad\quad\quad \rightarrow
( \Hom_{R}(\Gamma(B)\otimes_{R}\Gamma(B), \, \Gamma(B))\otimes_{R}\Omega(A), \, d_{\mathrm{CE}}^{B^{!}})
\end{split}
\end{equation*}
sends $\at_{\Gamma(\cF)[-1]}^{\nabla^{\cF}}$ to
$\at_{B}^{\nabla}$.

\item[(3)]
The following diagram commutes in the homotopy category
$\mathrm{\Pi}( (\smf_{\Mfedo}, \, \Qfedo)\mathrm{-}\mathbf{mod} )$:
\begin{equation*}
\xymatrix{
(\Gamma(\cF)[-1], \, \LQF)\otimes_{\smf_{\Mfedo}} (\Gamma(\cF)[-1], \, \LQF)
  \ar[r]^>>>>>>{\at_{\cF}} \ar[d]_{\iota^{\ast} \otimes \iota^{\ast}} &
 (\Gamma(\cF)[-1], \, \LQF) \ar[d]^{\iota^{\ast}}\\
(\Gamma(B^{!})[-1], \, d_{\mathrm{CE}}^{B^{!}})
\otimes_{\Omega(A)} (\Gamma(B^{!})[-1], \, d_{\mathrm{CE}}^{B^{!}}) \ar[r]^>>>>>{\at_{B}} & (\Gamma(B^{!})[-1], \, d_{\mathrm{CE}}^{B^{!}}) \, . \\
}
\end{equation*}
\end{itemize}
\end{proposition}
Here we have treated
the Lie pair Atiyah class
\[
\at_{B}:~~ (\Gamma(B^{!})[-1], \, d_{\mathrm{CE}}^{B^{!}})\otimes_{\Omega(A)}
(\Gamma(B^{!})[-1], \, d_{\mathrm{CE}}^{B^{!}})
\rightarrow (\Gamma(B^{!})[-1], \, d_{\mathrm{CE}}^{B^{!}})
\]
in Statement $(3)$
as a morphism in the homotopy category
$\Pi((C_{\Mfedo}^{\infty}, \Qfedo)\mathrm{-}\mathbf{mod})$.
Note that Statement $(3)$ is a direct consequence of Statements $(1)$ and $(2)$.

\subsection{Relation between $\Verticalthings{D}_{\mathrm{poly}}^1$ and $D_{\mathrm{poly}}^{\bullet}(L/A)$}\

Recall that the PBW map
$\pbw^{\nabla,j}:~~ \Gamma(SB) \rightarrow D_{\mathrm{poly}}^1(B)$
is an isomorphism of $R$-modules, and thus the induced map
\[
\pbw^{\nabla, j}\otimes_{R} \id:~~
\Verticalthings{D}_{\mathrm{poly}}^1=\Gamma(SB)\otimes_R \smf_{\Mfedo} \rightarrow D_{\mathrm{poly}}^1(B) \otimes_{R} \smf_{\Mfedo}
\]
is an isomorphism of graded $\smf_{\Mfedo}$-modules.
The isomorphism $\pbw^{\nabla, j}\otimes_{R} \id$ transports the differential $L_{\Qfedo}$ on
$\Verticalthings{D}_{\mathrm{poly}}^1$ to a differential $d_{\mathrm{can}}$ on $D_{\mathrm{poly}}^1(B) \otimes_{R} \smf_{\Mfedo}$ so that the following diagram commutes:
\begin{equation*}
\xymatrix{
\Verticalthings{D}_{\mathrm{poly}}^1 \ar[d]^{L_{\Qfedo}}\ar[rrr]^-{\pbw^{\nabla, j}\otimes_{R} \id}
& & &
D_{\mathrm{poly}}^1(B) \otimes_{R} \smf_{\Mfedo}
\ar[d]^{d_{\mathrm{can}}}
\\
\Verticalthings{D}_{\mathrm{poly}}^1 \ar[rrr]^-{\pbw^{\nabla, j}\otimes_{R} \id}
& & &
D_{\mathrm{poly}}^1(B) \otimes_{R} \smf_{\Mfedo} \, .  \\
}
\end{equation*}
In other words, we have the isomorphism of
$(\smf_{\Mfedo}, \, \Qfedo)$-modules:
\begin{equation*}\label{D(F)toD(B)}
\pbw^{\nabla, j}\otimes_{R} \id:~~
(\Verticalthings{D}_{\mathrm{poly}}^1, \, L_{\Qfedo}) \rightarrow (D_{\mathrm{poly}}^1(B) \otimes_{R} \smf_{\Mfedo}, \, d_{\mathrm{can}}).
\end{equation*}

Recall that $D_{\mathrm{poly}}^1(B)$ is a left
$A$-module.
We use the notation $d_{\mathrm{CE}}^{B}$ to denote the Chevalley-Eilenberg differential on
$D_{\mathrm{poly}}^1(B)\otimes_{R}\Omega(A)$.
We note that in \cite{Vitagliano},
for the Lie pair
$(L=T_M,\, A=T_\mathfrak{F})$ arising from a foliation $\mathfrak{F}$ on $M$,
it is proved that the space
$D_{\mathrm{poly}}^1(B)\otimes_{R}\Omega(A)$
admits an $A_\infty$-algebra structure whose
unary bracket is $d_{\mathrm{CE}}^{B}$.

Now we are able to draw a commutative diagram
 in the category
 $(\smf_{\Mfedo}, \, \Qfedo)\mathrm{-}\mathbf{mod}$ of
 $(\smf_{\Mfedo},$
 $\Qfedo)$-modules 
\begin{equation}\label{Imap}
\begin{xy}
(0,25)*+{(\tot \Verticalthings{D}_{\mathrm{poly}}, \, \LQF+d_H)}="a";
(70,25)*+{(D_{\mathrm{poly}}^{\bullet}(B)\otimes_{R}\smf_{\Mfedo}, \, d_{\mathrm{can}}+ d_H \otimes_{R} \id)}="b";
(70,0)*+{(D_{\mathrm{poly}}^{\bullet}(B)\otimes_{R}\Omega(A), \, d_{\mathrm{CE}}^{B}  + d_H \otimes_{R} \id) \, ,}="c";
{\ar^-{\pbw^{\nabla, j}\otimes_{R}\id} "a";"b"};
{\ar_-{ \id \otimes_{R}\,\iota^{\ast} } "b";"c"};
{\ar^>>>>>>>>>>>>>>>>>>>>>>>{I} "a";"c"};
\end{xy}
\end{equation}
where
$\id\otimes_{R}\,\iota^{\ast}:~~ D_{\mathrm{poly}}^{\bullet}(B)\otimes_{R}\smf_{\Mfedo} \rightarrow
D_{\mathrm{poly}}^{\bullet}(B)\otimes_{R}\Omega(A)$ is the restriction map and
$I$ is the composition of the horizontal and vertical arrows in the diagram.
According to \cite[Proposition $3.13$ and Lemma $3.14$]{B-S-X}, we have the following facts:
\begin{itemize}
	\item
	The horizontal map ${\pbw^{\nabla, j}\otimes_{R}\id}$ is an isomorphism of
$(\smf_{\Mfedo}, \, \Qfedo)$-modules;
	\item
	The vertical map $\id \otimes_{R}\,\iota^{\ast}$ is a quasi-isomorphism of
$(\smf_{\Mfedo}, \, \Qfedo)$-modules.
\end{itemize}
Consequently, $I$ is a quasi-isomorphism of
$(\smf_{\Mfedo}, \, \Qfedo)$-modules.


The next proposition can be verified directly.
\begin{proposition}\label{rightface}\

The following diagram commutes in the homotopy category
$\Pi( (\smf_{\Mfedo}, \, \Qfedo)\mathrm{-}\mathbf{mod} )$:
\begin{equation*}\label{rightfacediagram}
\xymatrix{
\otimes_{\smf_{\cM}}^{2}(\tot \functorL(\Verticalthings{D}_{\mathrm{poly}}^1), \, \LQF+d_H)  \ar[r]^-{\nbk} \ar[d]_{I^{\otimes 2} } &
 (\tot \functorL(\Verticalthings{D}_{\mathrm{poly}}^1), \, \LQF+d_H) \ar[d]^{I}\\
\otimes_{\Omega(A)}^{2}(\functorL(D_{\mathrm{poly}}^{1}(B))\otimes_{R}\Omega(A), \, d_{\mathrm{CE}}^{B}+ d_H\otimes_R \id)
\ar[r]^-{\nbk } & (\functorL(D_{\mathrm{poly}}^{1}(B))\otimes_{R}\Omega(A), \, d_{\mathrm{CE}}^{B}+ d_H\otimes_R \id) \, . \\
}
\end{equation*}
\end{proposition}

\subsection{The big diagram}\

Let
$\beta:~~ (\Gamma(B^{!})[-1], \, d_{\mathrm{CE}}^{B^{!}})
\rightarrow
(\functorL(D_{\mathrm{poly}}^{1}(B))\otimes_{R}\Omega(A), \, d_{\mathrm{CE}}^{B}+d_H\otimes_R \id)$
be the natural inclusion map of
$(\Omega(A), \, d_{\mathrm{CE}}^{A})$-modules.
As we have the morphism
$\iota^{\ast}: (\smf_{\Mfedo}, \, \Qfedo) \rightarrow (\Omega(A), \, d_{\mathrm{CE}}^{A})$ of differential graded algebras,
$\beta$ is also an inclusion map of
$(\smf_{\Mfedo}, \, \Qfedo)$-modules.

\begin{theorem}\label{bigdiagram}
The following diagram commutes in the homotopy category
$\Pi( (\smf_{\Mfedo}, \, \Qfedo)\mathrm{-}\mathbf{mod} )$:

\begin{equation}\label{CUBEDIAGRAM}
\begin{tiny}
\xymatrixcolsep{-1.8pc}
\xymatrix{
\otimes_{\smf_{\Mfedo}}^{2}(\Gamma(\cF)[-1], \, \LQF)\ar[rd]^{(\iota^{\ast})^{\otimes 2}}
\ar[rr]^-{\theta^{\otimes 2}}
    \ar[dd]^{\alpha_{\cF}}&&
    \otimes_{\smf_{\Mfedo}}^{2}(\tot \functorL(\Verticalthings{D}_{\mathrm{poly}}^1), \, \LQF+d_H )\ar[rd]^{I^{\otimes 2}} \ar[dd]_<<<<<<<{\textup{\textlbrackdbl}\; , \; \textup{\textrbrackdbl}}|(.51)\hole\\
    &\otimes_{\Omega(A)}^{2}(\Gamma(B^{!})[-1], \, d_{\mathrm{CE}}^{B^{!}})\ar[rr]^>>>>>>>>{\beta^{\otimes 2}}\ar[dd]_<<<<<<<{\alpha_{B}} &&
    \otimes_{\Omega(A)}^{2}(\functorL(D_{\mathrm{poly}}^{1}(B)){\otimes_{R}}\Omega(A), \, d_{\mathrm{CE}}^{B}+d_H\otimes_R\id)  \ar[dd]_{\textup{\textlbrackdbl}\; , \; \textup{\textrbrackdbl}}
    \\
    (\Gamma(\cF)[-1], \, \LQF)\ar[rd]^{\iota^{\ast}}\ar[rr]^>>>>>>>{\theta}|(.45)\hole && (\tot \functorL(\Verticalthings{D}_{\mathrm{poly}}^1), \, \LQF+d_H ) \ar[rd]^{I} \hole\\
    &(\Gamma(B^{!})[-1], \, d_{\mathrm{CE}}^{B^{!}})\ar[rr]^>>>>>>>>>>>>>{\beta}&&
    (\functorL(D_{\mathrm{poly}}^{1}(B)){\otimes_{R}}\Omega(A), \, d_{\mathrm{CE}}^{B}+d_H\otimes_R\id) \, . \\
}
\end{tiny}
\end{equation}
Moreover, the map $\beta$ in the front lower  edge is an isomorphism in
$\Pi(  (\smf_{\Mfedo}, \, \Qfedo)\mathrm{-}\mathbf{mod} )$.
\end{theorem}
\begin{proof}
It is not hard to see that the top and the bottom faces of the above cubic Diagram \eqref{CUBEDIAGRAM} are both commutative.
By Theorem \ref{betaforfoliation}, the back face commutes.
By Proposition \ref{leftface}, the left face commutes, and by Proposition \ref{rightface}, the right face commutes.
Note that $\iota^{\ast}$ and $I$ are isomorphisms in the category
$\Pi(  (\smf_{\Mfedo}, \, \Qfedo)\mathrm{-}\mathbf{mod}  )$.
Thus the front face commutes, and the map $\beta$ in the front lower edge is an isomorphism in
$\Pi( (\smf_{\Mfedo}, \, \Qfedo)\mathrm{-}\mathbf{mod}   )$.
\end{proof}

As an application of Theorem \ref{bigdiagram}, we recover the following result of
Chen-Sti\'enon-Xu \cite{C-S-X}:
\begin{corollary}
Let $\beta:~~ \Gamma(B)[-1] \rightarrow (\functorL(D_{\mathrm{poly}}^{1}(B)), \, d_H)$ be the natural inclusion map.
\begin{itemize}
\item[(1)]	
The map $\beta$ is an isomorphism in the derived category $D^b(\mathcal{A})$ of $A$-modules. Moreover, it is an isomorphism of Lie algebra objects, i.e., the following diagram commutes in $D^b(\mathcal{A})$:
\[
\xymatrix{
\Gamma(B)[-1] \otimes_{R} \Gamma(B)[-1]  \ar[r]^-{\beta \otimes \beta} \ar[d]_{\at_{B}} & (\functorL(D_{\mathrm{poly}}^{1}(B)), \, d_H)
\otimes_{R} (\functorL(D_{\mathrm{poly}}^{1}(B)), \, d_H) \ar[d]^{\nbk}\\
\Gamma(B)[-1] \ar[r]^-{\beta} & (\functorL(D_{\mathrm{poly}}^{1}(B)), \, d_H) \, .\\
}
\]

\item[(2)]
The cochain complex of $A$-modules
$(D_{\mathrm{poly}}^{\bullet}(B), \, d_H)$ is the universal enveloping algebra of the Lie algebra object
$(\Gamma(B)[-1]; \, \at_{B})$, and is a Hopf algebra object in $D^b(\mathcal{A})$.
\end{itemize}	
\end{corollary}

\section{Relation with Ramadoss's work}\label{Section7}
In \cite{Ramadoss}, Ramadoss studied the algebraic Atiyah classes of algebraic vector bundles over a field $\fK$ of characteristic zero.
In this part, we will describe without proofs the relations between some of Ramadoss's results and ours in the case $\fK=\C$.

Let $(X, \cO_{X})$ be a smooth scheme over $\C$, and $E$  an algebraic vector bundle over $X$.
The \textbf{algebraic Atiyah class} $\at(E) \in \Ext ^1_{\cO_X}(T_X \otimes_{\cO_X} E,E)$
of $E \rightarrow X$
is the extension class of the jet sequence
\begin{equation*}\label{ShortII}
0\rightarrow E
\stackrel{i}{\longrightarrow}
D_{X}^{\leqslant 1} \underset{\cO_{X}}{\otimes}E
\stackrel{j}{\longrightarrow}
T_{X}\otimes_{\cO_{X}}E
\rightarrow 0
\end{equation*}
in the category of $\cO_{X}$-modules.

We denote by $\at_{X}$ the algebraic Atiyah class of the algebraic tangent bundle $T_X$.
In the category of cochain complexes of coherent
$\cO_{X}$-modules, we have the Hopf algebra object
$(D_{\mathrm{poly}}^{\bullet}, \, d_H)$ and
the Lie algebra object
$(\functorL(D_{\mathrm{poly}}^1), \, d_H; \, \nbk)$ arising from the scheme $(X, \, \cO_X)$ \cite{Ramadoss}.
They are, respectively, algebraic versions of
$(\Dpp, \, L_Q, \, d_H)$ and $(\Lp, \, L_Q, \, d_H; \, \nbk)$ arising from a dg manifold.
The following theorem was obtained by Ramadoss \cite[Theorem 2]{Ramadoss}:
\begin{theorem}\label{Ramadossresult}
The diagram
\begin{equation}\label{Ramadossmaindiagram}
\xymatrix{
T_X [-1] \otimes_{\cO_X} T_X [-1]  \ar[r]^-{\at_{X}} \ar[d]_{\theta \otimes \theta} &
T_X [-1] \ar[d]^{\theta}\\
(\functorL(D_{\mathrm{poly}}^1), \, d_H)
\otimes_{\cO_X}
(\functorL(D_{\mathrm{poly}}^1), \, d_H)\ar[r]^-{\nbk} & (\functorL(D_{\mathrm{poly}}^1), \, d_H)
}
\end{equation}
commutes in  the derived category $D^{+}(\cO_X)$ of bounded below complexes of coherent $\cO_X$-modules, where $\theta: T_{X}[-1] \rightarrow (\functorL(D_{\mathrm{poly}}^1), \, d_H)$ is the natural inclusion map.	
\end{theorem}

Denote by $X^{sm}$ the underlying smooth manifold of the scheme $X$, and $X^{an}$ the underlying complex manifold.
Let $(\cX=T^{0,1}_{X^{sm}}[1], \, Q=\overline{ \partial})$ be the dg manifold associated with the Dolbeault resolution of the sheaf of holomorphic functions $\cO^{an}_X$. In other words, the support of $\cX$ is $X^{sm}$, the dg algebra of smooth functions
$\smf_{\cX}$ is $\Omega_{X^{sm}}^{0,\bullet}$, and the homological vector field $Q$ is
$\overline{\partial}$.

We fix the scheme $(X, \, \cO_X)$ and the corresponding
dg manifold $(\cX, \, Q)$ defined as above.
Then we construct a functor $c_0$ from the category of algebraic vector bundles over $X$ to the category of dg vector bundles over $(\cX, \, Q)$:
Let $E\to X$ be an algebraic vector bundle over $X$.
Denote by $E^{sm}$ (resp. $E^{an}$) the underlying smooth (resp. holomorphic) vector bundle over $X^{sm}$ (resp. $X^{an}$).
We define $c_0(E)$ to be the dg vector bundle
$(\cE, \, L_Q) \rightarrow (\cX, \, Q)$ associated with the Dolbeault resolution of $E^{an}$.
In other words, we have
$\Gamma(\cE)=\Omega_{X^{sm}}^{0,\bullet}\otimes_{C^\infty(X^{sm},\C)}\Gamma(E^{sm})$ and $L_Q=\overline{\partial}$.
Meanwhile,
let $g: E_1 \rightarrow E_2$ be a morphism of algebraic vector bundles.
We define
$c_0(g): (\cE_1, \, L_Q)\rightarrow (\cE_2, \, L_Q)$ to be the natural morphism of dg vector bundles arising from the induced map
$g:~\Gamma(E_1^{sm})\rightarrow \Gamma(E_2^{sm})$.

In a natural manner, we can upgrade the aforementioned functor $c_0$
to a functor
\[
c: ~~ D^{+}(\cO_X) \rightarrow \Pi((\smf_{\cX}, \, Q)\mathbf{-mod}) \, .
\]

\begin{proposition}\

With the notations as above,
the functor
$c$ sends
the Atiyah class
$\at(E)\in \Ext ^1_{\cO_X}(T_X \otimes_{\cO_{X}} E, \, E)$ of the algebraic vector bundle
$E \rightarrow X$
to the Atiyah class
$\at_{\cE}\in \Hom_{\Pi((\smf_{\cX}, \, Q)\mathbf{-mod}     )}(\XX(\cX)[-1]$
$\otimes_{\smf_{\cX}} \cE, \, \cE)$
of the dg vector bundle
$(\cE, \, L_Q) \rightarrow (\cX, \, Q)$.
\end{proposition}
In other words,
the algebraic Atiyah class $\at(E)$ and
the dg Atiyah class $\at_{\cE}$
coincide at the lower right corner of the following diagram:
\[
\begin{xy}
(0,15)*+{\at(E)\in\Ext ^1_{\cO_X}(T_X \otimes_{\cO_{X}} E, \, E)}="d";
(80,15)*+{\at_{\cE}\in\Hom_{\mathrm{H}((\smf_{\cX}, \, Q)\mathbf{-mod}   )}(\XX(\cX)[-1]\otimes_{\smf_{\cX}} \cE, \,\cE)}="a";
(0,0)*+{\Hom_{D^{+}(\cO_X)}(T_X[-1]\otimes_{\cO_{X}}E, \, E)}="b";
(80, 0)*+{\Hom_{\Pi((\smf_{\cX}, \, Q)\mathbf{-mod}     )}(\XX(\cX)[-1]\otimes_{\smf_{\cX}} \cE, \, \cE) \, .}="c";
{\ar "a"; "c"};
{\ar^-{c} "b"; "c"};
{\ar^{\simeq} "d"; "b" };
\end{xy}
\]
Here the vertical arrow on the right hand side comes from the natural functor
$\mathrm{H}((\smf_{\cX}, \, Q)\mathbf{-mod})$
$\rightarrow \Pi((\smf_{\cX}, \, Q)\mathbf{-mod})$.
\begin{proposition}
The functor
$c$
sends Ramadoss's Diagram \eqref{Ramadossmaindiagram} in the category  $D^{+}(\cO_X)$ to the commutative diagram
\begin{equation}\label{final diagram}
\xymatrix{
(\XX(\cX)[-1], \, L_Q) \otimes_{\smf_{\cX}} (\XX(\cX)[-1], \, L_Q)  \ar[r]^>>>>>>>>>>>{\at_{\cX}} \ar[d]_{\theta \otimes \theta} &
(\XX(\cX)[-1], \, L_Q)\ar[d]^{\theta}\\
(\tot\Lp, \, L_Q+d_H)
\otimes_{\smf_{\cX}}
(\tot\Lp, \, L_Q+d_H)\ar[r]^>>>>>{\nbk} & (\tot \functorL(\DD_{\mathrm{poly}}^1), \, L_Q+d_H) \\
}
\end{equation}
in the category
$\Pi((\smf_{\cX}, \, Q)\mathbf{-mod})$.
\end{proposition}

In fact, Diagram \eqref{final diagram} is a special case of Diagram \eqref{maindiagramII} where the dg manifold is $(\cX, \, Q)$.
So, Theorem \ref{maintheorem} can be seen as an analogue to Theorem \ref{Ramadossresult} (Theorem $2$ in \cite{Ramadoss}) in the setting of dg manifolds.
Also, among our results, Theorem \ref{R1}, Corollary \ref{Rcorollary1}, and Corollary \ref{R3}, are, respectively, analogous to  Ramadoss's Theorem $1$, Corollary $1$, and Theorem $3$ in \cite{Ramadoss}.

\section{Appendix}\label{Section8}
Let $\cM$ be a graded manifold.
In this part we will establish a key relation (Proposition \ref{Deltaformula}) between the coproduct $\Delta$ defined by
Equation  \eqref{Delta}, the canonical Lie bracket $\nbk$ defined by
Equation  \eqref{freeLiebracket}, and the Gerstenhaber product $\circ$ defined by
Equation  \eqref{Gerstenhaberproductformula}.

Recall that the space of differential operators
$\cD_{\cM}$ is spanned (over $\fK$) by elements of the form
\begin{equation}
\label{Eqt:Dform}D=X_1 \circ \cdots \circ X_n \, ,
\end{equation}
where $X_1,\cdots,X_n \in \XX(\cM)$ are homogeneous elements. In what follows, we will also treat the above $D$ as an element in $\DD_{\mathrm{poly}}^1$ and
$X_i$ as elements in  $\TT_{\mathrm{poly}}^1$.

\begin{proposition}\label{Deltaformula}
Given an element $D\in \DD_{\mathrm{poly}}^1$ as in
Equation  \eqref{Eqt:Dform},
we have
\begin{equation}\label{Eqt:Deltaexpressioninbrackets}
\Delta(D)=\sum_{\sigma \in \bS }
(-1)^{\tau(\sigma)}\kappa(\sigma)\,
\textup{\textlbrackdbl}X_1 \circ X_{\sigma(2)} \circ \cdots \circ X_{\sigma(p)}, \, X_{\sigma(p+1)} \circ \cdots \circ X_{\sigma(n)}
\textup{\textrbrackdbl} \, .
\end{equation}	
Here $\bS$ is the set of permutations $\sigma:~~\{1,\, \cdots, \, n\} \rightarrow \{1,\, \cdots, \, n\}$ such that $\sigma(1)=1$ and there exists   $p$ in $\{1,\, 2, \, \cdots,\, n\}$ for which 
$1<\sigma(2)< \cdots < \sigma(p)$ and
$\sigma(p+1)< \cdots < \sigma(n)$.
The number $\tau(\sigma)$ is the sum $\widetilde{X}_{\sigma(p+1)}+\cdots + \widetilde{X}_{\sigma(n)}$ and
$\kappa(\sigma)$ is the Koszul sign determined by the ordinary degrees   $\widetilde{X}_i$, i.e.,
\[
X_2\odot_{\smf_{\cM}} \cdots \odot_{\smf_{\cM}} X_n =\kappa(\sigma) X_{\sigma(2)} \odot_{\smf_{\cM}} \cdots \odot_{\smf_{\cM}} X_{\sigma(n)}\, .
\]
Here $\odot_{\smf_{\cM}}$ denotes the symmetric product in $S^{\bullet}(\XX(\cM))$.
\end{proposition}

We need the following lemma first.
Recall the Hochschild differential
$d_H: \Dpp^n \rightarrow \Dpp^{n+1}$ defined by
Equation  \eqref{Hochschild differential definition}.
\begin{lemma}
For all homogeneous elements $D, \, E \in \Dpp$ and $X \in \Tp$, we have
\begin{equation}\label{HochschildcoboundaryandGerstenhaberproduct}
\quad d_H(D\circ E)=d_H(D)\circ E +(-1)^{|D|-1}D \circ d_H(E) \, ,
\end{equation}
\begin{equation}\label{Gerstenhaberproductandcupproduct}
\text{and \quad}
X\circ \textup{\textlbrackdbl}
D, \, E
\textup{\textrbrackdbl}=
\textup{\textlbrackdbl}
X \circ D, \, E
\textup{\textrbrackdbl}+
(-1)^{|D||E|+1}
\textup{\textlbrackdbl}
X \circ E, \, D
\textup{\textrbrackdbl} \, .	
\end{equation}
\end{lemma}
\begin{proof}
Equation  \eqref{HochschildcoboundaryandGerstenhaberproduct} can be verified directly by definitions of the Hochschild differential $d_H$ and the Gerstenhaber product $\circ$. In fact, there is a similar identity for
Hochschild complexes of associative algebras in \cite{Gerstenhaber}.

To prove Equation  \eqref{Gerstenhaberproductandcupproduct},
we observe that
\[
X\circ(D\widetilde{\otimes}E)=(X\circ D)\widetilde{\otimes}E +(-1)^{(|X|+1)|D|}D\widetilde{\otimes}(X\circ E) \, ,
\]
for homogeneous elements
$D, \,E \in \Dpp$ and $X\in \Tp$,
and similarly, we have
\[
X\circ(E\widetilde{\otimes}D)=(X\circ E)\widetilde{\otimes}D +(-1)^{(|X|+1)|E|}E\widetilde{\otimes}(X\circ D) \, .
\]
Therefore, the right hand side of Equation  \eqref{Gerstenhaberproductandcupproduct}
can be computed as follows:
\begin{equation*}
\begin{split}
X\circ \textup{\textlbrackdbl}
D, \, E
\textup{\textrbrackdbl} & =X \circ (D\widetilde{\otimes} E- (-1)^{|D||E|}E \widetilde{\otimes}D) \\
& \quad = \textup{\textlbrackdbl}
X \circ D, \, E
\textup{\textrbrackdbl}+
(-1)^{(|X|+1)|D|}
\textup{\textlbrackdbl}
D, \, X \circ E
\textup{\textrbrackdbl} \\
& \quad\quad =
\textup{\textlbrackdbl}
X \circ D, \, E
\textup{\textrbrackdbl}+
(-1)^{(|X|+1)|D|+(|X|+|E|-1)|D|+1}
\textup{\textlbrackdbl}
X\circ E, \, D
\textup{\textrbrackdbl} \\
& \quad\quad\quad =
\textup{\textlbrackdbl}
X \circ D, \, E
\textup{\textrbrackdbl}+
(-1)^{|D||E|+1}
\textup{\textlbrackdbl}
X\circ E, \, D
\textup{\textrbrackdbl} \, .
\end{split}
\end{equation*}	
Thus Equation  \eqref{Gerstenhaberproductandcupproduct}
is proved.
\end{proof}

We now give the
\begin{proof}[Proof of Proposition \ref{Deltaformula}]
We adopt an inductive approach.
The $n=1$ case can be easily examined:
\begin{equation*}
\begin{split}
\Delta(X_1) & = X_1\widetilde{\otimes}\id_{\smf_{\cM}}+(-1)^{\widetilde{X}_1} \id_{\smf_{\cM}}\widetilde{\otimes}X_1 \\
& \quad = X_1\widetilde{\otimes}\id_{\smf_{\cM}}-(-1)^{\widetilde{X}_1+1} \id_{\smf_{\cM}}\widetilde{\otimes}X_1  =
\textup{\textlbrackdbl}
X_1, \, \id_{\smf_{\cM}}
\textup{\textrbrackdbl} \, .
\end{split}
\end{equation*}
Suppose that
Equation  \eqref{Eqt:Deltaexpressioninbrackets} is proved for some $n\geqslant 1$.
For the $n+1$ case,
we write the element $D=X_1 \circ \cdots \circ X_{n+1}$ as
\begin{equation*}
D=X_1 \circ \cdots \circ X_{n+1} =X_1 \circ D' \, ,
\end{equation*}
where $D'=X_2 \circ \cdots \circ X_{n+1}$.
By the induction assumption, we have
\begin{equation}\label{middle 1}
\begin{split}
\Delta(D')-	\textup{\textlbrackdbl}
D', \, \id_{\smf_{\cM}}
\textup{\textrbrackdbl} & = \sum_{\sigma' \in \bS'}
(-1)^{\tau(\sigma')}\kappa(\sigma')\,
\textup{\textlbrackdbl}X_2 \circ X_{\sigma'(3)} \circ \cdots \circ X_{\sigma'(p+1)}, \, X_{\sigma'(p+2)} \circ \cdots \circ X_{\sigma'(n+1)}
\textup{\textrbrackdbl} \, ,
\end{split}	
\end{equation}
where $\bS'$ is the set of permutations
$\sigma':~~\{2, \cdots, n+1\} \rightarrow \{2, \cdots, n+1\}$ such that $\sigma'(2)=2$,
$2< \sigma'(3) < \cdots  <\sigma'(p+1)$, and
$\sigma'(p+2) < \cdots < \sigma'(n+1)$ for some $1 \leqslant p \leqslant n-1$.

Applying $X_1\circ \cdot \,$ to each term in the right hand side of
Equation  \eqref{middle 1}, and using Equation  \eqref{Gerstenhaberproductandcupproduct}, we have
\begin{equation*}
\begin{split}
& X_1 \circ \textup{\textlbrackdbl}X_2 \circ X_{\sigma'(3)} \circ \cdots \circ X_{\sigma'(p+1)}, \, X_{\sigma'(p+2)} \circ \cdots \circ X_{\sigma'(n+1)}\textup{\textrbrackdbl} \\
& \quad \quad = \textup{\textlbrackdbl}X_1 \circ X_2 \circ X_{\sigma'(3)} \cdots \circ X_{\sigma'(p+1)}, \, X_{\sigma'(p+2)} \circ \cdots \circ X_{\sigma'(n+1)}\textup{\textrbrackdbl} \\
& \quad\quad\quad\quad +(-1)^{(A+1)(B+1)+1}
\textup{\textlbrackdbl}
X_1 \circ X_{\sigma'(p+2)} \circ \cdots \circ X_{\sigma'(n+1)}, \,
X_2 \circ X_{\sigma'(3)} \circ \cdots \circ X_{\sigma'(p+1)}
\textup{\textrbrackdbl} \, ,
\end{split}	
\end{equation*}
where
$A=\widetilde{X}_2 + \widetilde{X}_{\sigma'(3)}+ \cdots +\widetilde{X}_{\sigma'(p+1)}$ and $B= \widetilde{X}_{\sigma'(p+2)}+ \cdots + \widetilde{X}_{\sigma'(n+1)}$.
So by Equation  \eqref{middle 1} we get
\begin{equation}\label{middle 2}
\begin{split}
X_1 & \circ  (\Delta(D') -\textup{\textlbrackdbl}
D', \, \id_{\smf_{\cM}}
\textup{\textrbrackdbl})
= \\
& \sum_{\sigma' \in \bS'}
(-1)^{\tau(\sigma')}\kappa(\sigma')\,
X_1 \circ \textup{\textlbrackdbl}X_2 \circ X_{\sigma'(3)} \circ \cdots \circ X_{\sigma'(p+1)}, \, X_{\sigma'(p+2)} \circ \cdots \circ X_{\sigma'(n+1)}
\textup{\textrbrackdbl}
\\
& \,\, =
\sum_{\sigma' \in \bS'}
(-1)^{\tau(\sigma')}\kappa(\sigma')\,
\textup{\textlbrackdbl}X_1 \circ X_2 \circ X_{\sigma'(3)} \circ \cdots \circ X_{\sigma'(p+1)}, \, X_{\sigma'(p+2)} \circ \cdots \circ X_{\sigma'(n+1)}
\textup{\textrbrackdbl}\\
& \,\,\,\, +\!\! \sum_{\sigma' \in \bS'}
(-1)^{\tau(\sigma')+A+B+AB}\kappa(\sigma')\,
\textup{\textlbrackdbl}
X_1 \circ X_{\sigma'(p+2)} \circ \cdots \circ X_{\sigma'(n+1)}, \,
X_2 \circ X_{\sigma'(3)} \circ \cdots \circ X_{\sigma'(p+1)}
\textup{\textrbrackdbl} \, .
\end{split}
\end{equation}
Now we have
\begin{equation}\label{final 1}
\begin{split}
\Delta(D) &=\Delta(X_1 \circ D')\\
& =
\textup{\textlbrackdbl}
X_1 \circ D', \, \id_{\smf_{\cM}}
\textup{\textrbrackdbl}
-
(-1)^{\widetilde{X}_1+ \widetilde{D}'+1}
d_H(X_1 \circ D')  \quad\quad\,\,\,\,\,\,\, \text{(by Equation  \eqref{coproduct})}\\
& =
\textup{\textlbrackdbl}
X_1 \circ D', \, \id_{\smf_{\cM}}
\textup{\textrbrackdbl}
+(-1)^{\widetilde{D}'}X_1 \circ d_H(D') \quad\quad\quad\quad\quad\,\,\,\,\,
\text{(by Equation  \eqref{HochschildcoboundaryandGerstenhaberproduct} and $d_H(X_1)=0$) } \\
& =
\textup{\textlbrackdbl}
X_1 \circ D', \, \id_{\smf_{\cM}}
\textup{\textrbrackdbl}
+X_1 \circ
(
\Delta(D')-
\textup{\textlbrackdbl}
D', \, \id_{\smf_{\cM}}
\textup{\textrbrackdbl}
)
\quad\quad\,\, \,
\text{(by Equation  \eqref{coproduct}) } \\
& =
\textup{\textlbrackdbl}
X_1 \circ D', \, \id_{\smf_{\cM}}
\textup{\textrbrackdbl}\\
& \,\, + \!\!
\sum_{\sigma' \in \bS'}\!
(-1)^{\tau(\sigma')}\kappa(\sigma')
\textup{\textlbrackdbl}X_1 \circ X_2 \circ X_{\sigma'(3)} \circ \cdots \circ X_{\sigma'(p+1)}, \, X_{\sigma'(p+2)} \circ \cdots \circ X_{\sigma'(n+1)}
\textup{\textrbrackdbl}\\
&
\,\, + \!\!
\sum_{\sigma' \in \bS'}\!
(-1)^{\tau(\sigma')+A+B+AB}\kappa(\sigma')
\textup{\textlbrackdbl}
X_1 \circ X_{\sigma'(p+2)} \circ \cdots \circ X_{\sigma'(n+1)}, \,
X_2 \circ X_{\sigma'(3)} \circ \cdots \circ X_{\sigma'(p+1)}
\textup{\textrbrackdbl}.\\
& \quad\quad\quad\quad\quad\quad\quad\quad\quad\quad\quad\quad\quad\quad\quad\quad\quad\quad\quad\quad\quad\quad\quad \,\,\,
\text{ (by Equation  \eqref{middle 2}) }
\end{split}
\end{equation}

We then verify that we have the desired signs in the last two lines of Equation  \eqref{final 1}.
For each $\sigma' \in \bS' $ as above, we define two new permutations of $\{1, \cdots , n+1\}$ as follows:
\begin{itemize}
\item
Let
$\widetilde{\sigma}':~~ \{1, \cdots, n+1\} \rightarrow \{1, \cdots, n+1\}$
be the permutation defined by
$\widetilde{\sigma}'(1)=1$ and
$\widetilde{\sigma}'|_{\{2, \cdots, n+1\}}=\sigma'$. It is clear that we have
\begin{equation}\label{final 2}
\tau(\widetilde{\sigma}')=\tau(\sigma')=B\, , \quad \kappa(\widetilde{\sigma}')=\kappa(\sigma')\, .
\end{equation}
%

\item
Let
$\hat{\sigma}':~~ \{1, \cdots, n+1\} \rightarrow \{1, \cdots, n+1\}$
be the permutation defined by setting
\begin{gather*}
\hat{\sigma}'(1) =1 \, , \\
\quad\,\,\,\, \hat{\sigma}'(2) =\sigma'(p+2), \, \cdots, \, \hat{\sigma}'(n+1-p)=\sigma'(n+1) \, , \\
\hat{\sigma}'(n+2-p)=2, \, \cdots, \, \hat{\sigma}'(n+1)=\sigma'(p+1) \, .
\end{gather*}
It is clear that we have
$\tau(\hat{\sigma}')=A$ and
$\kappa(\hat{\sigma}')=(-1)^{AB}\kappa(\sigma')$.
Therefore, we have
\begin{equation}\label{final 3}
(-1)^{\tau(\hat{\sigma}')}\kappa(\hat{\sigma}')=(-1)^{A+AB}\kappa(\sigma')=(-1)^{\tau(\sigma')+A+B+AB}\kappa(\sigma') \, .
\end{equation}

\end{itemize}
Combining Equations \eqref{final 1}, \eqref{final 2},
and \eqref{final 3}, we finally get
\begin{equation*}
\begin{split}
\Delta(D) & =\textup{\textlbrackdbl}
X_1 \circ D', \, \id_{\smf_{\cM}}
\textup{\textrbrackdbl}\\
& \,\, +\!\! \sum_{\{\widetilde{\sigma}':~~ \sigma' \in\bS'\}}\!\!\!
(-1)^{\tau(\widetilde{\sigma}')}\kappa(\widetilde{\sigma}')\,
\textup{\textlbrackdbl}X_1 \circ X_2 \circ X_{\widetilde{\sigma}'(3)}  \circ \cdots \circ X_{\widetilde{\sigma}'(p+1)}, \, X_{\widetilde{\sigma}'(p+2)} \circ \cdots \circ X_{\widetilde{\sigma}'(n+1)}
\textup{\textrbrackdbl}\\
& \,\, +\!\! \sum_{\{\hat{\sigma}':~~ \sigma' \in \bS' \}}\!\!\!
(-1)^{\tau(\hat{\sigma}')}\kappa(\hat{\sigma}')\,
\textup{\textlbrackdbl}
X_1 \circ X_{\hat{\sigma}'(2)} \circ \cdots \circ X_{\hat{\sigma}'(n+1-p)}, \,
X_2 \circ X_{\hat{\sigma}'(n+3-p)} \circ \cdots \circ X_{\hat{\sigma}'(n+1)}
\textup{\textrbrackdbl} \, .
\end{split}
\end{equation*}
Observe that the union of sets $\{ \id_{\{1, \cdots, n+1 \}}\}\cup \{\widetilde{\sigma}':~~ \sigma' \in\bS'\} \cup \{\hat{\sigma}':~~ \sigma' \in \bS' \}$ is exactly the set of all permutations
$\sigma:~~\{1, \cdots ,n+1\} \rightarrow \{1, \cdots ,n+1\}$ such that $\sigma(1)=1$,
$1< \sigma(2) < \cdots < \sigma(p)$ and
$\sigma(p+1)< \cdots < \sigma(n+1)$ for some $p\geqslant 1$.
Thus Equation  \eqref{Eqt:Deltaexpressioninbrackets} is proved for the $n+1$ case.
\end{proof}

\end{document}